\documentclass{article}
\usepackage{amsmath}
\usepackage{amsthm}
\usepackage[title]{appendix}

\usepackage{textcomp}
\usepackage{times}
\usepackage{xcolor}
\usepackage{mathdots} 
\usepackage{enumitem}
\usepackage{caption}
\usepackage{chngcntr}
\counterwithin{figure}{section}
\counterwithin{table}{section}
\usepackage{tikz}
\usetikzlibrary{calc}
\usepackage{graphicx}
\usepackage[mathscr]{eucal}
\usepackage{scrextend}
\usepackage{wasysym}
\usepackage{etoolbox}
\makeatletter

\usepackage[all]{xy}

\usepackage{etoolbox}
\makeatletter
\patchcmd{\@makechapterhead}{50\p@}{\chapheadtopskip}{}{}

\patchcmd{\@makeschapterhead}{50\p@}{\chapheadtopskip}{}{}
\makeatother
\newlength{\chapheadtopskip}
\usepackage[parfill]{parskip} 
\usepackage{amssymb}
\usepackage{imakeidx}
\usepackage{enumitem}
\usepackage{xr}
\usepackage[mathlines]{lineno}
\DeclareMathAlphabet{\mathpzc}{OT1}{pzc}{m}{it}
\makeatletter
\newcommand{\mylabel}[2]{#2\def\@currentlabel{#2}\label{#1}}
\makeatother
\usepackage{fancyhdr}
\usepackage{enumitem}
\usepackage{xpatch}
\newtheorem{theorem}{Theorem}[section]
\newtheorem{lemma}[theorem]{Lemma}
\newtheorem{obs}[theorem]{Observation}
\newtheorem{defn}[theorem]{Definition}

\newtheorem{cor}[theorem]{Corollary}

\newtheorem{conj}[theorem]{Conjecture}

\newtheorem{subclaim}{Subclaim}[theorem]
\newcounter{claimlevel}[theorem]

\ExplSyntaxOn
\NewDocumentEnvironment{claim}{O{=}}
 {
  \str_case:nn { #1 }
   {
    {=}  { }
    {+}  { \stepcounter{claimlevel} }
    {-}  { \addtocounter{claimlevel}{-1} }
   }
  \begin{ Claim \int_to_Roman:n { \value{claimlevel} } }
 }
 {
  \end{ Claim \int_to_Roman:n { \value{claimlevel} } }
 }
\ExplSyntaxOff

\newenvironment{claimproof}[1]{\par\noindent\underline{Proof:}\space#1}{\hspace{1mm}$\blacksquare$}

\usepackage[left=2.54cm,top=2.54cm,right=2.54cm,bottom=2.54cm]{geometry}
\usepackage{apptools}
\usepackage[nodisplayskipstretch]{setspace}
\newlength\FHoffset
\fancyheadoffset{\FHoffset}
\newlength\FHright
\makeatletter
\usepackage{framed}

 \newtheoremstyle{TheoremNum}
        {\topsep}{\topsep}              
        {\itshape}                      
        {}                              
        {\bfseries}                     
        {.}                             
        { }                             
        {\thmname{#1}\thmnote{ \bfseries #3}}
    \theoremstyle{TheoremNum}
    \newtheorem{thmn}{Theorem}

\newtheoremstyle{PropNum}
        {\topsep}{\topsep}              
        {\itshape}                      
        {}                              
        {\bfseries}                     
        {.}                             
        { }                             
        {\thmname{#1}\thmnote{ \bfseries #3}}
    \theoremstyle{PropNum}

\newtheoremstyle{LemmaNum}
        {\topsep}{\topsep}              
        {\itshape}                      
        {}                              
        {\bfseries}                     
        {.}                             
        { }                             
        {\thmname{#1}\thmnote{ \bfseries #3}}
    \theoremstyle{LemmaNum}
    
\makeatletter
\renewcommand\subitem{\@idxitem\nobreak\hspace*{20\p@}}
\renewcommand\subsubitem{\@idxitem\nobreak\hspace*{20\p@}}
\makeatother

\date{}
\makeatother
\title{Distant 2-Colored Components on Embeddings Part I:\\Connecting Faces}
\author{Joshua Nevin}
\date{}
\begin{document}
\maketitle

\begin{center}\textbf{Abstract}\end{center} This is the first in a sequence of three papers in which we prove the following generalization of Thomassen's 5-choosability theorem: Let $G$ be a finite graph embedded on a surface of genus $g$. Then $G$ can be $L$-colored, where $L$ is a list-assignment for $G$ in which every vertex has a 5-list except for a collection of pairwise far-apart components, each precolored with an ordinary 2-coloring, as long as the face-width of $G$ is $2^{\Omega(g)}$ and the precolored components are of distance $2^{\Omega(g)}$ apart. This provides an affirmative answer to a generalized version of a conjecture of Thomassen and also generalizes a result from 2017 of Dvo\v{r}\'ak, Lidick\'y, Mohar, and Postle about distant precolored vertices.  

\section{Background and Motivation}

All graphs in this paper have a finite number of vertices. Given a graph $G$, a \emph{list-assignment} for $G$ is a family of sets $\{L(v): v\in V(G)\}$, where each $L(v)$ is a finite subset of $\mathbb{N}$. The elements of $L(v)$ are called \emph{colors}. A function $\phi:V(G)\rightarrow\bigcup_{v\in V(G)}L(v)$ is called an \emph{$L$-coloring of} $G$ if $\phi(v)\in L(v)$ for each $v\in V(G)$, and $\phi(x)\neq\phi(y)$ for any adjacent vertices $x,y$. Given an $S\subseteq V(G)$ and a function $\phi: S\rightarrow\bigcup_{v\in S}L(v)$, we call $\phi$ an \emph{ $L$-coloring of $S$} if $\phi$ is an $L$-coloring of the induced graph $G[S]$. A \emph{partial} $L$-coloring of $G$ is an $L$-coloring of an induced subgraph of $G$. Likewise, given an $S\subseteq V(G)$, a \emph{partial $L$-coloring} of $S$ is a function $\phi:S'\rightarrow\bigcup_{v\in S'}L(v)$, where $S'\subseteq S$ and $\phi$ is an $L$-coloring of $S'$. Given an integer $k\geq 1$, $G$ is called \emph{$k$-choosable} if it is $L$-colorable for every list-assignment $L$ for $G$ such that $|L(v)|\geq k$ for all $v\in V(G)$. Thomassen demonstrated in \cite{AllPlanar5ThomPap} that all planar graphs are 5-choosable. Actually, Thomassen proved something stronger. 

\begin{theorem}\label{thomassen5ChooseThm}
Let $G$ be a planar graph with facial cycle $C$. Let $xy\in E(C)$ and $L$ be a list assignment for $V(G)$ such that each vertex of $G\setminus C$ has a list of size at least five and each vertex of $C\setminus\{x,y\}$ has a list of size at least three, where $xy$ is $L$-colorable. Then $G$ is $L$-colorable.
\end{theorem}

Theorem \ref{thomassen5ChooseThm} has the following useful corollary.

\begin{cor}\label{CycleLen4CorToThom} Let $G$ be a planar graph with outer cycle $C$ and let $L$ be a list-assignment for $G$ where each vertex of $G\setminus C$ has a list of size at least five.  If $|V(C)|\leq 4$, then any $L$-coloring of $V(C)$ extends to an $L$-coloring of $G$. \end{cor}

We also rely on the following result from Postle and Thomas, which is an immediate consequence of Theorem 3.1 from \cite{2ListSize2PaperSeriesI} and is an analogue of Theorem \ref{thomassen5ChooseThm} where the precolored edge has been replaced by two lists of size two.

\begin{theorem}\label{Two2ListTheorem} Let $G$ be a planar graph with outer face $F$, and let $v,w\in V(F)$. Let $L$ be a list-assignment for $V(G)$ where $|L(v)|\geq 2$, $|L(w)|\geq 2$, and furthermore, for each $x\in V(F)\setminus\{v,w\}$, $|L(x)|\geq 3$, and, for each $x\in V(G\setminus F)$, $|L(x)|\geq 5$. Then $G$ is $L$-colorable. \end{theorem}

We now recall some notions from topological graph theory. Given an embedding $G$ on surface $\Sigma$, the deletion of $G$ partitions $\Sigma$ into a collection of disjoint, open connected components called the \emph{faces} of $G$. Our main objects of study are the subgraphs of $G$ bounding the faces of $G$. Given a subgraph $H$ of $G$, we call $H$ a \emph{facial subgraph} of $G$ if there exists a connected component $U$ of $\Sigma\setminus G$ such that $H=\partial(U)$. We call $H$ is called a \emph{cyclic facial subgraph} (or, more simply, a \emph{facial cycle}) if $H$ is both a facial subgraph of $G$ and a cycle. Given  a cycle $C\subseteq G$, we say that $C$ is \emph{contractible} if it can be contracted on $\Sigma$ to a point, otherwise we say it is \emph{noncontractible}. We now introduce two standard paramaters that measure the extent to which an embedding deviates from planarity. 

\begin{defn}\label{EWandFWDefn} \emph{Let $\Sigma$ be a surface and let $G$ be an embedding on $\Sigma$. The \emph{edge-width} of $G$, denoted by $\textnormal{ew}(G)$, is the length of the shortest noncontractible cycle in $G$. The \emph{face-width} of $G$, denoted by $\textnormal{fw}(G)$, is the smallest integer $k$ such that there exists a noncontractible closed curve of $\Sigma$ which intersects with $G$ on $k$ points. If $G$ has no noncontractible cycles, then we define $\textnormal{ew}(G)=\infty$, and if $g(\Sigma)=0$, then we define $\textnormal{fw}(G)=\infty$. The face-width of $G$ is also sometimes called the \emph{representativity} of $G$. Some authors consider the face-width to be undefined if $\Sigma=\mathbb{S}^2$, but, for our purposes, adopting the convention that $\textnormal{fw}(G)=\infty$ in this case is much more convenient. The notion of face-width was introduced by Robertson and Seymour in their work on graph minors and has been studied extensively.} \end{defn}

Our main result for these three papers is Theorem \ref{5ListHighRepFacesFarMainRes} below, a vast generalization of Theorem \ref{thomassen5ChooseThm}.

\begin{theorem}\label{5ListHighRepFacesFarMainRes} Let $\Sigma$ be a surface, $G$ be an embedding on $\Sigma$ of face-width at least $2^{\Omega(g(\Sigma))}$, and $F_1, \ldots, F_m$ be a collection of facial subgraphs of $G$ which are pairwise of distance at least $2^{\Omega(g(\Sigma))}$ apart. Let $x_1y_1, \ldots, x_my_m$ be a collection of edges in $G$, where $x_iy_i\in E(F_i)$ for each $i=1,\ldots, m$. Let $L$ be a list-assignment for $G$ such that 
\begin{enumerate}[label=\arabic*)]
\itemsep-0.1em
\item for each $v\in V(G)\setminus\left(\bigcup_{i=1}^mV(C_i)\right)$, $|L(v)|\geq 5$; AND
\item For each $i=1,\ldots, m$, $x_iy_i$ is $L$-colorable, and, for each $v\in V(F_i)\setminus\{x_i, y_i\}$, $|L(v)|\geq 3$.
\end{enumerate}
Then $G$ is $L$-colorable.  \end{theorem}

Theorem \ref{5ListHighRepFacesFarMainRes} is slightly stronger than imposing pairwise-far apart components with ordinary 2-colorings. That is, an immediate consequence of Theorem \ref{5ListHighRepFacesFarMainRes} is the following slightly weaker result.

\begin{theorem}\label{WVersionThmPrecCompFW}  Let $\Sigma$ be a surface, $G$ be an embedding on $\Sigma$ of face-width at least $2^{\Omega(g(\Sigma))}$, and $L$-be a list-assignment for $V(G)$ in which every vertex has a list of size at least five, except for the vertices of some connected subgraphs $K_1, \cdots, K_m$ of $G$ which are pairwise of distance at least $2^{\Omega(g(\Sigma))}$ apart, where, for each $i=1, \cdots, m$, there is an $L$-coloring of $K_i$ which is an ordinary 2-coloring. Then $G$ is $L$-colorable. 
 \end{theorem}

We now give some context to Theorems \ref{5ListHighRepFacesFarMainRes} and \ref{WVersionThmPrecCompFW}. Postle and Thomas have a proof (\cite{ManyFacesFarLuke}) of the planar version of Theorem \ref{5ListHighRepFacesFarMainRes} below. 

\begin{theorem}\label{GenusZeroThesisVersMainTh}
Let $G$ be a planar graph and let $F_1, \ldots, F_m$ be a collection of facial subgraphs of $G$ which are pairwise of distance at least $\Omega(1)$ apart. Let $x_1y_1, \ldots, x_my_m$ be a collection of edges in $G$, where $x_iy_i\in E(F_i)$ for each $i=1,\ldots, m$. Let $L$ be a list-assignment for $G$ such that
\begin{enumerate}[label=\arabic*)]
\itemsep-0.1em
\item for each $v\in V(G)\setminus\left(\bigcup_{i=1}^mV(C_i)\right)$, $|L(v)|\geq 5$; AND
\item For each $i=1,\ldots, m$, $x_iy_i$ is $L$-colorable, and, for each $v\in V(F_i)\setminus\{x_i, y_i\}$, $|L(v)|\geq 3$.
\end{enumerate}
Then $G$ is $L$-colorable. 
\end{theorem}

Theorem \ref{GenusZeroThesisVersMainTh} gives a positive answer to a conjecture posed at the end of \cite{LukeGraphSurfaceThesisMain} and gives a positive answer to a list-coloring version of the following conjecture from \cite{ColCritFixedSurfaceChoose} for ordinary colorings, albeit with a different distance constant. 

\begin{conj}\label{ThomConjOrdColAns} Let $G$ be a planar graph and $W\subseteq V(G)$ such that $G[W]$
is bipartite and any two components of $G[W]$ have distance at least 100
from each other. Can any coloring of $G[W]$ such that each component is
2-colored be extended to a 5-coloring of $G$? \end{conj}

A positive answer to Conjecture \ref{ThomConjOrdColAns} in the special case where each component of $W$ is a lone vertex was provided by Albertson in \cite{CantPaintCornerPap}. The main result of \cite{DistPrecVertChoosePap} provides a positive answer to a list-coloring version of Conjecture \ref{ThomConjOrdColAns} in the case where each component of $W$ is a lone vertex, albeit with a different distance constant. 

\bigskip

\begin{center}
\begin{tabular}{|c|c|c|}\hline  & \textnormal{Planar} & \textnormal{Locally Planar} \\ \hline \textnormal{5-choosable?} & \textnormal{Always} & \textnormal{$\Omega(\log(g))$ edge-width} \\ \textnormal{5-choosable with distant precolored vertices?} & \textnormal{$\Omega(1)$ distance} & \textnormal{$\Omega(1)$ distance and $\Omega(g)$ edge-width}  \\  \textnormal{5-choosable with distant 2-colored components?} & \textnormal{$\Omega(1)$ distance} & ? \\ \hline  \end{tabular}\captionof{table}{}\label{chapPmsiMaRe11}
\end{center}

What if we extend the setting to these questions from planar to locally planar graphs, i.e to embeddings on arbitrary surface that have high edge-width? In \cite{FirstCh5PapBoundExp}, DeVos, Kawarabayashi, and Mohar proved that an embedding $G$ on a surface $\Sigma$ is 5-choosable as long as the edge-width of $G$ is at least $2^{\Omega(g)}$, where $g$ is the genus of $\Sigma$, and this lower bound was improved by Postle and Thomas to $\Omega(\log(g))$ in \cite{HyperbolicFamilyColorSurfPap}, who also showed in the same paper that the result from \cite{DistPrecVertChoosePap} about planar embeddings with distant precolored vertices also has a locally planar analogue: That is, given an embedding $G$ on a surface $\Sigma$, a set $X\subseteq V(G)$, and a list-assignment $L$ for $V(G)$ in which each vertex of $G\setminus X$ has a list of size at least five and each vertex of $X$ has size at least one, $G$ is $L$-colorable as long as it has edge-width at least $\Omega(g)$ and the vertices of $X$ are pairwise of distance at least $X$ apart. The above results are summarized in Table \ref{chapPmsiMaRe11}, where ``2-colored" refers to an ordinary 2-coloring and ``distance" refers to the pairwise distance bounds on the precolored components. The purpose of this sequence of papers is to fill in the bottom right entry of Table \ref{chapPmsiMaRe11}, but it turns out that edge-width is not the right measure of local planarity to do this, as we illustrate below. We show now that Theorem \ref{WVersionThmPrecCompFW} fails if we replace the words ``face-width" with ``edge-width", even if we restrict ourselves to just a single 2-colored component and allow the edge-width to be arbitrarily large. In Figure \ref{FigUnboundEdgeWCounter}, we have an embedding $G$ on the torus in which the thick edge comes out of the page, and a list-assignment $L$ for $V(G)$ in which every vertex has the list $\{a,b,c\}$. Any contractible cycle in $G$ contains the edge $xy$, and the distance between $x$ and $y$ in $G-xy$ is $k+1$, so the edge width of $G$ can be made arbitrarily large. But, if $k\equiv 1\ \textnormal{mod 3}$, then $G$ is not $L$-colorable, since $x$ and $y$ need the same color. The above shows that the ``edge-width" version of Theorem \ref{5ListHighRepFacesFarMainRes} fails even we drop the condition that each of the designated faces is permitted a precolored edge, so the edge-width version of Theorem \ref{WVersionThmPrecCompFW} fails as well.

\begin{center}\begin{tikzpicture}[scale=0.9]
\node[shape=circle,draw=black] [label={[xshift=1.0cm, yshift=-0.6cm]\textcolor{red}{\small\textnormal{$\{a, b, c\}$}}}] (X) at (-7.5,0) {$x$};
\node[shape=circle,draw=black] [label={[xshift=-1.0cm, yshift=-0.6cm]\textcolor{red}{\small\textnormal{$\{a, b, c\}$}}}] (Y) at (7.5,0) {$y$};
\node[shape=circle,draw=black] (P1) at (-5.5, 2) {$p_1$};
\node[shape=circle,draw=black] (Q1) at (-5.5,-2) {$q_1$};
\node[shape=circle,draw=black] (P2) at (-4, 2) {$p_2$};
\node[shape=circle,draw=black] (Q2) at (-4,-2) {$q_2$};
\node[shape=circle,draw=black] (P3) at (-2.5, 2) {$p_3$};
\node[shape=circle,draw=black] (Q3) at (-2.5,-2) {$q_3$};
\node[shape=circle,draw=black] (P4) at (-1, 2) {$p_4$};
\node[shape=circle,draw=black] (Q4) at (-1,-2) {$q_4$};
\node[draw=none,node distance=1cm] (DT) at (0,2) {$\ldots$};
\node[draw=none,node distance=1cm] (DT2) at (0,-2) {$\ldots$};
\node[shape=circle,draw=black] (PK3) at (1, 2) {\tiny $p_{k-3}$};
\node[shape=circle,draw=black] (QK3) at (1,-2) {\tiny $q_{k-3}$};
\node[shape=circle,draw=black] (PK2) at (2.5, 2) {\tiny $p_{k-2}$};
\node[shape=circle,draw=black] (QK2) at (2.5,-2) {\tiny $q_{k-2}$};
\node[shape=circle,draw=black] (PK1) at (4, 2) {\tiny $p_{k-1}$};
\node[shape=circle,draw=black] (QK1) at (4,-2) {\tiny $q_{k-1}$};
\node[shape=circle,draw=black] (PK0) at (5.5, 2) {$p_{k}$};
\node[shape=circle,draw=black] (QK0) at (5.5,-2) {$q_{k}$};

\node[shape=circle,draw=white] (IN) at (-0.5,-1) {};
\node[shape=circle,draw=white] (IN2) at (0.5,1) {};
\draw [-] (X) to [out=90,in=180] (P1);
\draw [-] (X) to [out=270,in=180] (Q1);

\draw [-] (PK0) to [out=0, in=90] (Y);
\draw [-] (QK0) to [out=0,in=270] (Y);

\draw [-] (P1) to (Q1); 
\draw [-] (P1) to (P2) to (P3) to (P4); 
\draw [-] (Q1) to (Q2) to (Q3) to (Q4); 
\draw [-] (Q1) to (P2); 
\draw [-] (P2) to (Q2); 
\draw [-] (Q2) to (P3); 
\draw [-] (P3) to (Q3) to (P4) to (Q4); 
\draw [-] (PK3) to (PK2) to (PK1) to (PK0); 
\draw [-] (QK3) to (QK2) to (QK1) to (QK0); 
\draw [-] (PK3) to (QK3) to (PK2) to (QK2); 
\draw [-] (QK2) to (PK1) to (QK1) to (PK0) to (QK0);
\draw [-] (Q4) to (DT2) to (QK3);
\draw [-] (P4) to (DT) to (PK3);
\draw [-] (Q4) to (IN);
\draw [-] (PK3) to (IN2);
\draw[-, line width=1.8pt] (X) to [out=90, in=90] (Y);
\end{tikzpicture}\captionof{figure}{}\label{FigUnboundEdgeWCounter}\end{center}

\section{Conventions of this Paper}

Unless otherwise specified, all graphs are regarded as embeddings on a previously specified surface, and all surface are compact, connected, and have zero boundary. If we want to talk about a graph $G$ as an abstract collection of vertices and edges, without reference to sets of points and arcs on a surface then we call $G$ an \emph{abstract graph}. 

\begin{defn}\label{ContractNatCPartDefn}\emph{Let $\Sigma$ be a surface, let $G$ be an embedding on $\Sigma$, and let $C$ be a contractible cycle in $G$. Let $U_0, U_1$ be the two open connected components of $\Sigma\setminus C$. The unique \emph{natural $C$-partition} of $G$ is the pair $\{G_0, G_1\}$ of subgraphs of $G$ where, for each $i\in\{0,1\}$, $G_i=G\cap\textnormal{Cl}(U_i)$.} \end{defn}

\begin{defn}
\emph{Given a graph $G$, a subgraph $H$ of $G$, a subgraph $P$ of $G$, and an integer $k\geq 1$, we call $P$ a \emph{$k$-chord} of $H$ if $|E(P)|=k$ and $P$ is of the following form.}
\begin{enumerate}[label=\emph{\arabic*)}]
\itemsep-0.1em
\item \emph{$P:=v_1\cdots v_kv_1$ is a cycle with $v_1\in V(H)$ and $v_2, \cdots, v_k\not\in V(H)$}; OR
\item \emph{$P:=v_1\cdots v_{k+1}$, and $P$ is a path with distinct endpoints, where $v_1, v_{k+1}\in V(H)$ and $v_2,\cdots, v_k\not\in V(H)$.}
\end{enumerate}

\end{defn}

$P$ is called \emph{proper} if it is not a cycle, i.e $P$ intersects $H$ on two distinct vertices. Otherwise it is called \emph{improper}. Note that, for $1\leq k\leq 2$, any $k$-chord of $H$ is proper, as $G$ has no loops or duplicated edges. A 1-chord of $H$ is simply referred to as a \emph{chord} of $H$. In some cases, we are interested in analyzing $k$-chords of $H$ where the precise value of $k$ is not important. We call $P$ a \emph{generalized chord} of $H$ if there exists an integer $k\geq 1$ such that $P$ is a $k$-chord of $H$. We call $P$ a \emph{proper} generalized chord of $H$ if there is an integer $k\geq 1$ such that $P$ is a proper $k$-chord of $H$. (A proper generalized chord of $H$ is also called an \emph{$H$-path}). We define \emph{improper} generalized chords of $H$ analogously. For any $A, B\subseteq V(G)$, an \emph{$(A,B)$-path} is a path $P=x_0\cdots x_k$ with $V(P)\cap A=\{x_0\}$ and $V(P)\cap B=\{x_k\}$. Given a surface $\Sigma$, an embedding $G$ on $\Sigma$, a cyclic facial subgraph $C$ of $G$, and a proper generalized chord $Q$ of $C$, there is, under certain circumstances, a natural way to partition of $G$ specified by $Q$. 

\begin{defn}\label{ContractNatCQPartChordDefn} \emph{Let $\Sigma$ be a surface, let $G$ be an embedding on $\Sigma$,  let $C$ be a cyclic facial subgraph of $G$ and let $Q$ be a generalized chord of $C$, where each cycle in $C\cup Q$ is contractible. The unique \emph{natural $(C,Q)$-partition} of $G$ is the pair $\{G_0, G_1\}$ of subgraphs of $G$ such that $G=G_0\cup G_1$ and $G_0\cap G_1=Q$, where, for each $i\in\{0,1\}$, there is a unique open connected region $U$ of $\Sigma\setminus (C\cup Q)$ such that $G_i$ consists of all the edges and vertices of $G$ in the closed region $\textnormal{Cl}(U)$.}

\emph{If the facial cycle $C$ is clear from the context then we usually just refer to $\{G_0, G_1\}$ as the \emph{natural $Q$-partition} of $G$. Note that this is consistent with Definition \ref{ContractNatCPartDefn} in the sense that, if $Q$ is not a proper generalized chord of $C$ (i.e $Q$ is a cycle) then the natural $Q$-partition of $G$ is the same as the natural $(C,Q)$-partition of $G$. If $\Sigma$ is the sphere (or plane) then the natural $(C, Q)$-partition of $G$ is always defined for any $C,Q$.}
\end{defn}

 In this paper, our primary interest lies in embeddings without short separating cycles, so we define the following

\begin{defn} \emph{Let $\Sigma$ be a surface and let $G$ be an embedding on $\Sigma$. A \emph{separating cycle} in $G$ is a contractible cycle $C$ in $G$ such that each of the two connected components of $\Sigma\setminus C$ has nonempty intersection with $V(G)$. We call $G$ \emph{short-inseparable} if $\textnormal{ew}(G)>4$ and $G$ does not contain any separating cycle of length 3 or 4.}\end{defn}

Finally, we recall the following standard notation. 

\begin{defn}

\emph{For any graph $G$, vertex set $X\subseteq V(G)$, integer $j\geq 0$, and real number $r\geq 0$, we set $D_j(X, G):=\{v\in V(G): d(v, X)=j\}$ and $B_r(X, G):=\{v\in V(G): d(v, X)\leq r\}$. Given a subgraph $H$ of $G$, we usually just write $D_j(H, G)$ to mean $D_j(V(H), G)$, and likewise, we usually write $B_r(H, G)$ to mean $B_r(V(H), G)$.}
\end{defn}

If $G$ is clear from the context, then we drop the second coordinate from the above notation to avoid clutter. We now introduce some additional notation related to list-assignments. We frequently analyze the situation where we begin with a partial $L$-coloring $\phi$ of a graph $G$, and then delete some or all of the vertices of $\textnormal{dom}(\phi)$ and remove the colors of the deleted vertices from the lists of their neighbors in $G\setminus\textnormal{dom}(\phi)$. We thus define the following.  

\begin{defn}\label{ListLSvRemove}\emph{Let $G$ be a graph with list-assignment $L$. Let $\phi$ be a partial $L$-coloring of $G$ and $S\subseteq V(G)$. We define a list-assignment $L^S_{\phi}$ for $G\setminus (\textnormal{dom}(\phi)\setminus S)$ as follows.}
$$L^S_{\phi}(v):=\begin{cases} \{\phi(v)\}\ \textnormal{if}\ v\in\textnormal{dom}(\phi)\cap S\\ L(v)\setminus\{\phi(w): w\in N(v)\cap (\textnormal{dom}(\phi)\setminus S)\}\ \textnormal{if}\ v\in V(G)\setminus \textnormal{dom}(\phi) \end{cases}$$ \end{defn}

If $S=\varnothing$, then $L^{\varnothing}_{\phi}$ is a list-assignment for $G\setminus\textnormal{dom}(\phi)$ in which the colors of the vertices in $\textnormal{dom}(\phi)$ have been deleted from the lists of their neighbors in $G\setminus\textnormal{dom}(\phi)$. The situation where $S=\varnothing$ arises so frequently that, in this case, we drop the superscript and let $L_{\phi}$ denote the list-assignment $L^{\varnothing}_{\phi}$ for $G\setminus\textnormal{dom}(\phi)$. In some cases, we specify a subgraph $H$ of $G$ rather than a vertex-set $S$. In this case, to avoid clutter, we write $L^H_{\phi}$ to mean $L^{V(H)}_{\phi}$. Finally, given two partial $L$-colorings $\phi$ and $\psi$ of $V(G)$, where $\phi(x)=\psi(x)$ for all $x\in\textnormal{dom}(\phi)\cap\textnormal{dom}(\psi)$, and $\phi(x)\neq\psi(y)$ for all edges $xy$ with $x\in\textnormal{dom}(\phi)$ and $y\in\textnormal{dom}(\psi)$, there is natural well-defined $L$-coloring $\phi\cup\psi$ of $\textnormal{dom}(\phi)\cup\textnormal{dom}(\psi)$, where
$$(\phi\cup\psi)(x)=\begin{cases}\phi(x)\ \textnormal{if}\ x\in\textnormal{dom}(\phi)\\ \psi(x)\ \textnormal{if}\ x\in\textnormal{dom}(\psi)\end{cases}$$

Having introduce the notation above, we now state a simple but useful result which we use in this paper. The result below is consequence of a characterization in \cite{lKthForGoBoHm6} of obstructions to extending a precoloring of a cycle of length at most six in a planar graph. It is also straightforward to check just using Theorem \ref{thomassen5ChooseThm} that the result below holds. 

\begin{theorem}\label{BohmePaper5CycleCorList} Let $G$ be a short-inseparable planar embedding with facial cycle $C$, where $5\leq |V(C)|\leq 6$ . Let $G'=G\setminus C$ and $L$ be a list-assignment for $G$, where $|L(v)|\geq 5$ for all $v\in V(G')$. Let $\phi$ be an $L$-coloring of $V(C)$ which does not extend to $L$-color $G$. If $|V(C)|=5$, then $G'$ consists of a lone vertex $v$ adjacent to all five vertices of $C$, where $L_{\phi}(v)=\varnothing$. On the other hand, if $|V(C)|=6$, then $G'$ consists of one of the following.
\begin{enumerate}[label=\roman*)]
\itemsep-0.1em
\item A lone vertex $v$ adjacent to at least five vertices of $C$, where $L_{\phi}(v)=\varnothing$; OR
\item An edge where, for each $v\in V(G')$, $L_{\phi}(v)$ is the same 1-list and $G[N(v)\cap V(C)]$ is a length-three path; OR
\item A triangle where, for each $v\in V(G')$, $L_{\phi}(v)$ is the same 2-list and $G[N(v)\cap V(C)]$ is a length-two path.
\end{enumerate}
\end{theorem}

We frequently deal with situations where we have a set $S$ of vertices that we want to delete, and it is desirable to color as few of them as possible in such a way that we can safely delete the remaining vertices of $S$ without coloring them. We thus introduce the following definition.

\begin{defn}\emph{Let $G$ be a graph with a list-assignment $L$. Given an $S\subseteq V(G)$ and a partial $L$-coloring $\phi$ of $V(G)$, we say that $S$ is \emph{$(L, \phi)$-inert in $G$} if, for every extension of $\phi$ to a partial $L$-coloring $\psi$ of $G\setminus (S\setminus\textnormal{dom}(\phi))$ whose domain contains $D_1(S\setminus\textnormal{dom}(\phi))$, we get that $\psi$ extends to an $L$-coloring of $\textnormal{dom}(\psi)\cup V(S)$.}\end{defn}

We repeatedly implicitly rely on the following observation. 

\begin{obs}\label{CombineInert} Let $G$ be a graph with a list-assignment $L$. Let $S, S'\subseteq V(G)$ and $\phi, \phi'$ be partial $L$-colorings of $G$, where $\phi\cup\phi'$ is well-defined and the sets $S\setminus\textnormal{dom}(\phi)$ and $S'\setminus\textnormal{dom}(\phi')$ are of distance at least two apart. If $S$ is $(L, \phi')$-inert in $G$ and $S'$ is $(L, \phi')$-inert in $G$, then $S\cup S'$ is $(L, \phi\cup\phi')$-inert in $G$. \end{obs}

\section{The Main Results and Structure of this paper}

The main result of this paper are Theorems \ref{FaceConnectionMainResult} and \ref{SingleFaceConnRes}, which we use in Paper II as part of the proof of Theorem \ref{5ListHighRepFacesFarMainRes}. To motivate these results, we provide a brief an overview of the proof of Theorem \ref{5ListHighRepFacesFarMainRes}. The bulk of the proof of Theorem \ref{5ListHighRepFacesFarMainRes} consists of the proof that it holds in the special case where $G$ is short-inseparable. We prove this in Paper II In Paper III, we show that the short-inseparable case implies that Theorem \ref{5ListHighRepFacesFarMainRes} holds in general. Roughly speaking, we prove that Theorem \ref{5ListHighRepFacesFarMainRes} holds in the short-inseparable case in a later paper by considering a vertex-minimal counterexample $G$ to a carefully chosen set of induction conditions and then showing that this minimal counterexample satisfies some desirable properties. We can use these properties to produce a smaller counterexample by coloring and deleting a subgraph of $G$ that consists mostly of paths between some of the faces $F_1, \cdots, F_m$ and a more complicated structure near $F_1, \cdots, F_m$.  To state Theorems \ref{FaceConnectionMainResult}-\ref{SingleFaceConnRes}, we introduce the following definitions. 

\begin{defn}\label{UniqueKLSpecifiedDefn} \emph{A tuple $\mathcal{K}=[\Sigma, G, C, L]$ is called a \emph{collar} if $\Sigma$ is a surface, $G$ is an embedding on $\Sigma$, $C$ is a contractible facial cycle of $G$, and $L$ is a list-assignment for $V(G)$. Given an integer $k\geq 2$, we say that $\mathcal{K}$ is \emph{uniquely $k$-determined} if, for any generalized chord $Q$ of $C$ (proper or improper) of length at most $k$, each cycle in $C\cup Q$ is contractible, and, letting $G=G_0\cup G_1$ be the natural $(C,Q)$-partition of $G$, there is a (necessarily unique) $i\in\{0,1\}$ such that}
\begin{enumerate}[label=\emph{\alph*)}]
\itemsep-0.1em
\item\emph{Every vertex of $G_{1-i}\setminus C$ has an $L$-list of size at least five and the open component of $\Sigma\setminus (C\cup Q)$ whose closure contains $G_{1-i}$ is a disc}; AND
\item\emph{Either $G_i\setminus C$ contains a vertex $v$ with $|L(v)|<5$ or the open component of $\Sigma\setminus (C\cup Q)$ whose closure contains $G_{i}$ is not a disc}.
\end{enumerate}
\emph{If $\Sigma, G, L$ are all clear, then we just say that $C$ is uniquely $k$-determined. In the setting above, we denote $G_{1-i}$ by $G^{\textnormal{small}}_Q=G_{1-i}$ and $G_i$ by $G^{\textnormal{large}}_Q$. If $k\geq 3$ and $Q'$ is a generalized chord of $C$ (of arbitrary length) with the additional property that each vertex of $Q'$ is of distance at most one from $C$, then $Q'$ also satisfies all the partitioning properties above and we define $G^{\textnormal{small}}_{Q'}$ and $G^{\textnormal{large}}_{Q'}$ analogously in this case.} 
 \end{defn}

 In view of Definition \ref{UniqueKLSpecifiedDefn}, it is natural to introduce the following notation.

\begin{defn}\label{DefnSmallSideLargeSideUniqueKLD} \emph{Let $k\geq 1$ be an integer and $\mathcal{K}=[\Sigma, G, C, L]$ be a uniquely $k$-determined collar. We define $\textnormal{Sh}^k(C, \mathcal{K})$ to be the union of all sets of the form $V(G^{\textnormal{small}}_Q\setminus Q)$, where $Q$ runs over all generalized chords $Q$ of $C$ with $|E(Q)|\leq k$. If $\mathcal{K}$ is clear from the context, then we just write $\textnormal{Sh}^k(C)$.}
 \end{defn}

We now state our two main results for this paper.

\begin{theorem}\label{FaceConnectionMainResult} There exists a constant $d\geq 0$ such that the following holds: Let $G$ be a short-inseparable embedding on a surface $\Sigma$ with $\textnormal{fw}(G)\geq 6$, and let $\mathcal{C}$ be a family of facial cycles of $G$, where the elements of $\mathcal{C}$ are pairwise of distance $\geq d$ apart and every facial subgraph of $G$ not lying in $\mathcal{C}$ is a triangle. Let $L$ be a list-assignment for $V(G)$, where each vertex of $V(G)\setminus\left(\bigcup_{C\in\mathcal{C}}V(C)\right)$ has a list of size at least five. Let $\mathcal{D}\subseteq\mathcal{C}$ with $|\mathcal{D}|>1$, where each element of $\mathcal{D}$ is uniquely $4$-determined and each vertex of $\bigcup_{C\in\mathcal{D}}V(C)$ has a list of size three. Let $F\in\mathcal{D}$ and $\{P_C: C\in\mathcal{D}\setminus\{F\}\}$ be a family of paths, pairwise of distance $\geq d$ apart, where each $P_C$ is a shortest $(C,F)$-path. Then there exist an $A\subseteq V(G)$ and partial $L$-coloring $\phi$ of $A$ such that: 
\begin{enumerate}[label=\arabic*)]
\item $A$ is $(L, \phi)$-inert in $G$, and each vertex of $D_1(A)$ has an $L_{\phi}$-list of size at least three; AND
\item $G[A]$ is connected and has nonempty intersection with each element of $\mathcal{D}$; AND
\item $A\subseteq\left(B_2(F)\cup\textnormal{Sh}^4(F)\right)\cup\bigcup_{C\in\mathcal{D}\setminus\{F\}}\left(\textnormal{Sh}^4(C)\cup B_2(C\cup P_C)\right)$.
\end{enumerate}
In particular, any $d\geq 34$ suffices. 
 \end{theorem}

Informally, Theorem \ref{FaceConnectionMainResult} states that, starting with many pairwise far-apart faces with 3-lists, we can, under the specified conditions, color and delete a subgraph of $G$ which is far from all the elements of $\mathcal{C}\setminus\mathcal{D}$ and connects the elements of $\mathcal{D}$ to a single face of $G\setminus A$ (where $G\setminus A$ regarded as an embedding on $\Sigma$) where all the vertices of this new lone face all have 3-lists. We specify that the elements of $\mathcal{D}$ have vertices with lists of size precisely three, rather than at least three, so that, for each $C\in\mathcal{C}$, the requirement that $C$ is uniquely 4-determined enforces the desired property that, for any generalized chord $Q$ of $C$ of length at most four, all the elements of $\mathcal{D}\setminus\{C\}$ lie in the same component of $\Sigma\setminus (C\cup Q)$. Our second main result, Theorem \ref{SingleFaceConnRes},  is a ``single face" version of Theorem \ref{FaceConnectionMainResult}. We need Theorem \ref{SingleFaceConnRes} in a later paper as well is to deal with the case where an application of Theorem \ref{FaceConnectionMainResult} to a minimal counterexample to the short-inseparable case of Theorem \ref{5ListHighRepFacesFarMainRes} results in the new embedding with fewer vertices having a lower face-width than the original embedding, as face-width, unlike edge-width, is not monotone with respect to subgraphs. 

\begin{theorem}\label{SingleFaceConnRes} Let $\Sigma, G, \mathcal{C}, L, d$ be as in Theorem \ref{FaceConnectionMainResult}. Let $F\in\mathcal{C}$, where each vertex of $F$ has a list of size at least three and $F$ is uniquely 4-determined (possiby $\mathcal{C}=\{F\}$). Let $P:=v_0\cdots v_n$ be a path of length $\geq d$ with both endpoints in $F$, where $P$ is a shortest path between its endpoints and $V(P)\cap D_k(F)=\{v_k, v_{n-k}\}$ for each $0\leq k\leq 3$.  Suppose further that there is a noncontractible closed curve of $\Sigma$ contained in $F\cup P$. Then there exist an $A\subseteq V(G)$ and a partial $L$-coloring $\phi$ of $A$ such that 
\begin{enumerate}[label=\arabic*)]
\itemsep-0.1em
\item $A$ is $(L, \phi)$-inert in $G$ and each vertex of $D_1(A)$ has an $L_{\phi}$-list of size at least three.
\item $G[A]$ is connected and $V(F)\cup V(v_3Pv_{n-3})\subseteq A\subseteq B_2(F\cup P)\cup\textnormal{Sh}^4(F)$.
\end{enumerate}
 \end{theorem}

To prove these, we first introduce the following definition.

\begin{defn} \emph{Given a graph $G$ with list-assignment $L$, an \emph{$L$-reduction} (in $G$) is a pair $(A, \phi)$, where $A\subseteq V(G)$ and $\phi$ is a partial $L$-coloring of $A$, such that $A$ is $(L, \phi)$-inert in $G$ and, for each $v\in D_1(A)$, either $|L(v)|<5$ or $|L_{\phi}(v)|\geq 3$. We say that $(A, \phi)$ is a \emph{complete} reduction if each $v\in D_1(A)$ has an $L$-list of size at least five and thus an $L_{\phi}$-list of size at least three. A family $\{(S_1, \phi_1), \cdots, (S_n, \phi_n)\}$ of $L$-reductions is called \emph{consistent} if}
\begin{enumerate}[label=\emph{\arabic*)}]
\itemsep-0.1em
\item\emph{$\phi_1\cup\cdots\phi_n$ is a well-defined $L$-coloring of its domain}; AND
\item\emph{For any $1\leq i<j\leq n$, we have $d(S_i\setminus\textnormal{dom}(\phi_i), S_j\setminus\textnormal{dom}(\phi_j))>1$}; AND
\item\emph{For any $x\in D_1(S_1\cdots S_k)$, there exists an $i\in\{1,\cdots, n\}$ such that $N(x)\cap\textnormal{dom}(\phi_1\cup\cdots\cup\phi_n)\subseteq\textnormal{dom}(\phi_i)$.}
\end{enumerate}
\end{defn}

We note that Observation \ref{CombineInert} immediately implies the following. 

\begin{obs}\label{UnionConsistentObs} Let $G$ be a graph with list-assignment $L$ and $\{(S_1, \phi_1), \cdots, (S_n, \phi_n)\}$ be a consistent family of $L$-reductions. Then $(S_1\cup\cdots S_n, \phi_1\cdots\cup\phi_n)$ is also an $L$-reduction. \end{obs}

Thus, we prove Theorems \ref{FaceConnectionMainResult} and \ref{SingleFaceConnRes} by finding in each case a consistent family of $L$-reductions whose union satisfies the specified proeprties. The structure of the remainder of this paper is as follows. The proof of Theorems \ref{FaceConnectionMainResult} and \ref{SingleFaceConnRes} has two ingredients: We first need some results about reductions near the elements of $\mathcal{C}$, which we prove in Sections \ref{RainbowSec}-\ref{PrelimSecWrto1}. In particular, the main result we need is Theorem \ref{MainCollarResultCutProc}. We also need a similar result for the paths between the elements of $\mathcal{C}$, which we prove in Sections \ref{IntermResSecPathDel}-\ref{MainProofSubSecThm1PathCol}. This result is Theorem \ref{MainResultColorAndDeletePaths}. The actual proofs of Theorems \ref{FaceConnectionMainResult} and \ref{SingleFaceConnRes} are in Section \ref{ProofThmMainSec}, and are relatively short once we have Theorems \ref{MainCollarResultCutProc} and \ref{MainResultColorAndDeletePaths} in hand. 

\section{Rainbows}\label{RainbowSec}

To state our main result about boundary structures, we introduce one more definition.

\begin{defn} \emph{Let $G$ be an embedding on a surface $\Sigma$ and let $H\subseteq G$. Given a set $S\subseteq V(G)$, we say that $S$ is \emph{topologically reachable from $H$} if, for every $v\in S$, either $v\in V(H)$, or there is an arc of $\Sigma$ with $v$ as an endpoint and the other endpoint in $C$, where this arc intersects with the edges and vertices of $G$ only on at most $G[S]\cup H$.} \end{defn}

Sections \ref{RainbowSec}-\ref{PrelimSecWrto1} consist of the proof of the following result. 

\begin{theorem}\label{MainCollarResultCutProc} Let $\mathcal{K}=[\Sigma, G, C, L]$ be a uniquely 4-determined collar, where $G$ is short-inseparable. Let $uw\in E(G)$ with $d(u, C)=3$ and $d(w,C)=2$, where $u,w\not\in\textnormal{Sh}^4(C)$. Suppose every vertex of $C$ has a list of size at least three and every vertex of $B_4(C)\setminus V(C)$ has a list of size at least five. Then there is a complete $L$-reduction $(S, \phi)$ such that
\begin{enumerate}[label=\arabic*)]
\itemsep-0.1em
\item $u\in\textnormal{dom}(\phi)$ and $V(C)\subseteq S\subseteq B_3(C)\cup\textnormal{Sh}^4(C)$, and furthermore, $S$ is topologically reachable from $C$; AND
\item $(S\setminus\textnormal{Sh}^4(C))\cap D_3(C)=\{u\}$ and $(S\setminus\textnormal{Sh}^4(C))\cap D_2(C)=\{w\}$; AND
\end{enumerate}
 \end{theorem}

To prove this, we need some additional machinery and results from \cite{JNevinHolepunchI}, which we restate below.  

\begin{defn}\label{EndNotationColor}  \emph{Let $G$ be a graph with list-assignment $L$. Let $H$ be a subgraph of $G$ and $P\subseteq G$ with endpoints $p, p'$, where $|V(P)|\geq 3$. Let $pq$ and $p'q'$ be the terminal edges of $P$. We define the following.}
\begin{enumerate}[label=\emph{\arabic*)}]
\itemsep-0.1em
\item\emph{We let $\textnormal{End}_L(P, G)$ be the set of $L$-colorings $\phi$ of $\{p, p'\}$ such that any extension of $\phi$ to an $L$-coloring of $V(P)$ extends to $L$-color all of $G$.}
\item\emph{We let $\textnormal{Crown}_{L}(P, G)$ be the set of partial $L$-colorings $\phi$ of $V(C)\setminus\{q, q'\}$ such that}
\begin{enumerate}[label=\emph{\alph*)}]
\itemsep-0.1em
\item\emph{$p, p'\in\textnormal{dom}(\phi)$ and, if $|E(P)|>3$, then at least one vertex of $\mathring{P}\setminus\{q, q'\}$ lies in $\textnormal{dom}(\phi)$}; AND
\item\emph{Any extension of $\phi$ to an $L$-coloring of $\textnormal{dom}(\phi)\cup V(P)$ extends to $L$-color all of $G$, and furthermore, for each $x\in V(P)\setminus\textnormal{dom}(\phi)$, $|L_{\phi}(x)|\geq 3$}
\end{enumerate}
\end{enumerate}
 \end{defn}

We usually drop the subscript $L$ in the case where it is clear from the context. 

\begin{defn}\label{RainbowDefnX} \emph{A \emph{rainbow} is a tuple $(G, C, P, L)$, where $G$ is a planar graph with outer cycle $C$, $P$ is a path on $C$ with $|E(P)|\geq 1$, and $L$ is a list-assignment for $V(G)$ such that each endpoint of $P$ has a nonempty list and furthermore, $|L(v)|\geq 3$ for each $v\in V(C\setminus P)$ and $|L(v)|\geq 5$ for each $v\in V(G\setminus C)$. }
 \end{defn}

The following is straightforward to check using Theorems \ref{thomassen5ChooseThm} and \ref{Two2ListTheorem}. 

\begin{lemma}\label{PartialPathColoringExtCL0} Let $(G, C, P, L)$ be a rainbow and $\phi$ be a partial $L$-coloring of $V(P)$ which does not extend to $L$-color $G$, where $\textnormal{dom}(\phi)$ contains the endpoints of $P$. Then either there is a chord of $C$ with one endpoint in $\textnormal{dom}(\phi)\cap V(\mathring{P})$ and the other endpoint in $C\setminus P$, or there is a $v\in V(G\setminus C)\cup (V(\mathring{P})\setminus\textnormal{dom}(\phi))$ with $|L_{\phi}(v)|\leq 2$.
\end{lemma}

We now recall from \cite{JNevinHolepunchI} some facts about colorings of 2-paths. We first define the following. 

\begin{defn}\label{GUniversalDefinition}
\emph{Let $H$ be a graph. We call $H$ a \emph{broken wheel} if there is a $p\in V(H)$ such that $H-p$ is a path $q_1\ldots q_{\ell}$ with $\ell\geq 2$, where $N(p)=\{q_1, \ldots, q_{\ell}\}$. The path $q_1pq_{\ell}$ is called the \emph{principal path} of $H$. Given a length-two path $P':= pqp'\subseteq H$, a list-assignment $L$ for $V(H)$, and an $a\in L(p)$, we say that $a$ is \emph{$(P', H)$-universal} if, for each $b\in L(q)\setminus\{a\}$ and each $c\in L(p')\setminus\{b\}$, $(a,b,c)$ is a proper $L$-coloring of $pqp'$ which extends to $L$-color $H$.}
 \end{defn}

Note that, in the setting above, in order for $a$ to be $(P', H)$ universal, it is necessarily the case that either $pp'\not\in E(H)$ or $a\not\in L(p')$. The following two results are proven in \cite{JNevinHolepunchI}.

\begin{obs}\label{CorMainEitherBWheelAtM1ColCor} Let $(G, C, P, L)$ be a rainbow, where $P=p_1p_2p_3$, and $H$ is short-inseparable, and every chord of $C$ has $p_2$ as an endpoint. Suppose further that $|V(C)|>4$ and $|L(p_3)|\geq 2$. Then, letting $zxp_3$ be the length-two subpath of $C-p_2$ with endpoint $p_3$, either there is an $G$-universal $a\in L(p_3)$ or $G$ is a broken wheel with principal path $P$ such that $L(p_3)\subseteq L(x)\cap L(z)$.\end{obs}

\begin{theorem}\label{SumTo4For2PathColorEnds} Let $(G, C, P, L)$ be an end-linked rainbow, where $P:=p_0qp_1$. Then  $\textnormal{End}(P,G)\neq\varnothing$. Furthermore, either $|\textnormal{End}(P, G)|\geq 2$ or there is an even-length path $Q$ with $V(Q)\subseteq V(C)$, where $Q$ has endpoints $p_0, p_1$ and each vertex of $Q$ is adjacent to $q$.   \end{theorem}

The main results from \cite{JNevinHolepunchI} are the three theorems stated below.

\begin{theorem}\label{MainHolepunchPaperResulThm} Let $(G, C, P, L)$ be a rainbow, where $P$ has length four, the endpoints of $\mathring{P}$ have no common neighbor in $C\setminus P$, each internal vertex of $P$ has an $L$-list of size at least five, and at least one endpoint of $P$ has a list of size at least three. Then $\textnormal{Crown}_{L}(P, G)\neq\varnothing$.
 \end{theorem}

\begin{theorem}\label{ModifiedRes5ChordCaseDegen} Let $(G, C, P, L)$ be an end-linked rainbow, where $|E(P)|=5$, and suppose that each vertex of $\mathring{P}$ has a list of size at least five and is incident to a chord of $C$ whose other endpoint lies in $C\setminus\mathring{P}$. Then, for each endpoint $y$ of the middle edge of $P$, there is a $\phi\in\textnormal{Crown}(P, G)$ with $y\in\textnormal{dom}(\phi)$. \end{theorem}

If, in the statement of Theorem \ref{ModifiedRes5ChordCaseDegen}, we instead allow both endpoints of $P$ to have lists of size at least three, then something stronger holds:

\begin{theorem}\label{RainbowNonEqualEndpointColorThm} Let $\mathcal{G}:=(G, C, P, L)$ be a rainbow, where $P=p_0q_0y_0y_1q_1p_1$ and each vertex of $\mathring{P}$ has a list of size at least five and at least one neighbor in $C\setminus\mathring{P}$. Define a \emph{$\mathcal{G}$-obstruction} to be an $x\in V(C)$ such that $x$ is either adjacent to all four vertices of $\mathring{P}$ or, for some endpoint $p_i$ of $P$, $x$ is adjacent to $p_i, q_{1-i}, y_{1-i}$. Then either
\begin{enumerate}[label=\arabic*)]
\itemsep-0.1em
\item there is a $\mathcal{G}$-obstruction; OR
\item For for each $j\in\{0,1\}$, there is an $i\in\{0,1\}$ and an $a\in L(p_i)$ such that, for any $b\in L(p_{1-i})$, there is an element of $\textnormal{Crown}(P, G)$ which has $y_j$ in its domain and uses $a,b$ on $p_i, p_{1-i}$ respectively. 
\end{enumerate}
 \end{theorem}

\section{Link Colorings}\label{MainLinkSec}

During the proof of our main result about boundary structures, as well as the proof of the main result of this paper, we often partially color and delete the vertices of a subpath of a facial cycle $C$ so as to not create vertices with lists of size less than three in $D_1(C)$. In this section, we prove some results showing that we have such a procedure if the path ``respects" the partitioning from Definition \ref{UniqueKLSpecifiedDefn} in a sense defined below. 

\begin{defn}\label{uniquekDetLocalPVers} \emph{Let $k\geq 2$ and $\mathcal{K}=[\Sigma, G, C, L]$ be a uniquely $k$-determined collar.  For any subpath $P$ of $C$,}
\begin{enumerate}[label=\emph{\arabic*)}]
\itemsep-0.1em
\item\emph{we associate to $P$ a vertex set $\textnormal{Sh}^k(P)$, where $v\in\textnormal{Sh}^k(P)$ if there is a proper generalized chord $Q$ of $C$ of length at most $k$, such that both endpoints of $Q$ lie in $P$, and $v\in V(G^{\textnormal{small}}_Q\setminus Q)$, and furthermore, $G^{\textnormal{small}}_Q\cap P$ has one connected component and $G^{\textnormal{\large}}_Q$ has two connected components.}
\item \emph{We say $P$ is \emph{$k$-consistent} if both of the following hold.}
\begin{enumerate}[label=\emph{\alph*)}]
\itemsep-0.1em
\item\emph{There is no chord of $C$ with one endpoint in $P$ and the other endpoint in $C\setminus P$}; AND
\item \emph{For any proper generalized chord $Q$ of $C$ of length at most $k$, if $Q$ has at both endpoints in $P$ and at least one endpoint in $\mathring{P}$, then $G^{\textnormal{small}}_Q\cap P$ has one connected component and $G^{\textnormal{large}}_Q\cap P$ has two connected components.}
\end{enumerate}
\item\emph{We let $\textnormal{Link}_L(P)$ denote the set of $L$-colorings $\phi$ of $V(P)\setminus\textnormal{Sh}^2(P)$ such that $\textnormal{Sh}^2(P)$ is $(L, \phi)$-inert in $G$. We drop the subscript if $L$ is clear from the context. Note that, if $P$ is 2-consistent, then, for any $\phi\in\textnormal{Link}(P)$, the pair $(V(P)\cup\textnormal{Sh}^2(P), \phi)$ is an $L$-reduction.} 
\end{enumerate}
\end{defn}

The following is straightforward to check using repeated applications of Theorems \ref{thomassen5ChooseThm} and \ref{SumTo4For2PathColorEnds}.
 
\begin{theorem}\label{MainLinkingResultAlongFacialCycle}  Let $\mathcal{K}=[\Sigma, G, C, L]$ be a uniquely 2-determined collar and $P$ be a 2-consistent subpath of $C$ with endpoints $p, p'$, where each internal vertex of $P$ has an $L$-list of size at least three. Let $A\subseteq L(p)$ and $A'\subseteq L(p')$, where $A, A'$ are nonempty. If either 1) or 2) below hold,  then there is a $\phi\in\textnormal{Link}(P)$ with $\phi(p)\in A$ and $\phi(p')\in A'$. 
\begin{enumerate}[label=\arabic*)]
\itemsep-0.1em
\item $|V(P)|=1$ and $A=A'$; OR
\item $|V(P)|\geq 2$ and $|A|+|A'|\geq 4$, and furthermore, either $pp'\not\in E(G)\setminus E(C)$ or $A\cap A'=\varnothing$.
\end{enumerate}
\end{theorem}

In the setting above, an arbitrary precoloring of the endpoints of $P$ does not necessarily extend to an element of $\textnormal{Link}(P)$, but, under some conditions, if we strengthen the $2$-consistency condition to $3$-consistency, we can guarantee that something analogous holds if we are allowed to color one more vertex, which is the result in Theorem \ref{LinkPlusOneMoreVertx}. 

\begin{defn} \emph{Let $\Sigma, G, C, L, k$ be as in Definition \ref{uniquekDetLocalPVers}. A \emph{$P$-peak} is a vertex $v\in D_1(C)$ such that $|N(v)\cap V(P)|\geq 2$, where $v\not\in\textnormal{Sh}^2(P)$. We say $v$ is an \emph{internal} $P$-peak if it is not adjacent to any endpoint of $P$.} \end{defn}

Informally, the $P$-peaks are the midpoints of the ``maximal" 2-chords of $C$ with both endpoints in $P$.

\begin{theorem}\label{LinkPlusOneMoreVertx} Let $k\geq 3$ and $[\Sigma, G, C, L]$ be a uniquely 3-determined collar, where $G$ is short-inseparable. Suppose that, for each $v\in B_2(C)$, every facial subgraph of $G$ containing $v$, except possibly $C$, is a triangle, and either $v\in V(C)$ or $|L(v)|\geq 5$. Let $P$ be a $3$-consistent subpath of $C$ with endpoints $p_0, p_1$, where $V(C\setminus P)\neq\varnothing$ and each vertex of $\mathring{P}$ has a list of size at least three. Let $\phi$ be an $L$-coloring of $\{p_0, p_1\}$. Then, for all but at most three internal $P$-peaks $w$, there is an $L$-reduction $(T, \tau)$ with $V(P)\subseteq T\subseteq V(P+w)\cup\textnormal{Sh}^3(P)$, where $\tau$ is an extension of $\phi$. \end{theorem}

\begin{proof} Suppose the theorem does not hold, where $|V(G)|$ is minimal with respect to this property. 

\begin{claim} There is no chord of $C$ with an endpoint in $P$. \end{claim}

\begin{claimproof} Suppose there is such a chord $xy$ of $C$. Since $P$ is $3$-consistent, both of $x,y$ lie in $P$, and furthermore, $G^{\large}_{xy}\cap P$ has two connected components, and $G^{\textnormal{large}}_{xy}\cap P$ has one connected component. Say for the sake of definiteness that $x\in V(p_0Py)$. Let $P'$ be the subpath $p_0PxyPp_1$ of $G[V(P)]$ and $G':=G^{\textnormal{large}}_{xy}$. Let $C':=(C\cap G^{\textnormal{large}}_Q)+xy$. Now, $[\Sigma, G', C', L]$ is also a uniquely 3-determined collar, and $P'$ is still $3$-consistent with respect to this collar. Furthermore, $G'$ still satisfies the specified triangulation and list-assignment conditions, and $P'$ has the asme endpoints as $P$. Note that the set of internal $P$-peaks (with respect to $G, C, L$ is precisely equal to the set of $P'$-peaks (with respect to $G', C', L$). Furthermore, it follows from Theorem \ref{thomassen5ChooseThm} applied to $G^{\textnormal{small}}_{xy}$ that, for any $L$-reduction $(T, \tau)$ in $G'$, the pair $(T, \tau)$ is also an $L$-reduction in $G$, so we contradict our assumption that $G$ is a minimal counterexample. \end{claimproof}

We now note the following, which can be obtained from repeated applications of Theorem \ref{SumTo4For2PathColorEnds}. 

\begin{claim}\label{SingleUSeL2Prec} Let $P^*$ be a subpath of $P$. Let $q_0, q_1$ be the endpoints of $P$. For each $k=0,1$, let $A^k$ be a nonempty subset of $L(q_k)$, where $|A^0|+|A^1|\geq 3$. Then there is a $\phi\in\textnormal{Link}(P^*)$ with $\phi(q_0)\in A^0$ and $\phi(q_1)\in A^1$ unless, for each edge $e$ of $P^*$, there is an $x\in D_1(C)$ such that $G[N(x)\cap V(P^*)]$ is an even-length path containing $e$. \end{claim}

For each $e\in E(P)$, let $\mathcal{Q}_e$ be the set of 2-chords of $Q$ of $C$ such that $Q$ has both endpoints in $P$ and $e\in E(P)$. As $P$ is 3-consistent, it follows from our triangulation conditions that $\mathcal{Q}_e\neq\varnothing$ for each $e\in E(P)$, and, in particular, there exists a unique $Q_e\in\mathcal{Q}_e$ such that $G^{\textnormal{small}}_{Q_e}$ contains all the elements of $\mathcal{Q}_e$. Let $\{x^1, \cdots, x^r\}$ denote the set of midpoints of the elements of $\{Q_e: e\in E(P)\}$. Note that $\{x^1, \cdots, x^r\}$ is precisely the set of $P$-peaks, and $x^2, \cdots, x^{r-1}$ are the internal $P$-peaks. For each $i=1,\cdots, r$, we let $Q^i$ be the unique element of $\{Q_e: e\in E(P)\}$ with midpoint $x^i$, and we let $H^i:=G^{\textnormal{small}}_{Q^i}$. There is a natural ordering to the elements of $\{x_e: e\in E(P)\}$. That is, we suppose that, $x^1, \cdots, x^r$ are labelled such that, for any $1<i<j\leq r$, the removal of any edge of $H^i\cap P$ separates $p_0$ from $H^j\cap P$ in $P$. We suppose that $r\geq 4$, or else there is nothing to prove. Now, if $\phi$ extends to an element of $\textnormal{Link}(P)$, then we are done, so suppose there is no such extension. For each $k=1,\cdots, r$, we say that $x^k$ is \emph{even} if $H^k$ is a broken wheel with principal vertex $x^k$, where $H^k-x^k=H^k\cap P$ and $H^k\cap P$ is an even-length path. 

\begin{claim}\label{AllbutOneIntEv} For all but at most one $k\in\{1, \cdots, r\}$, $x^k$ is even. \end{claim}

\begin{claimproof} Suppose not. Thus, there exist indices $k, k'$ with $1\leq k<k'\leq r$ such that neither $x^k$ nor $x^{k'}$ is even. Let $v, w$ be the endpoints of $Q^k$ and $w', v'$ be the endpoints of $Q^{k'}$, where $w, w'\in V(\mathring{P})$ and, in particular, $v$ lies in the unique component of $P\setminus\{w, w'\}$ containing $p_0$ and $v'$ lies in the unique component of $P\setminus\{w, w'\}$ containing $p_1$. Possibly $w=w'$. Now, by Theorem \ref{MainLinkingResultAlongFacialCycle}, there is a $\tau\in\textnormal{Link}(p_0Pv)$ with $\tau(p_0)=\psi(p_0)$. Let $a:=\tau(v)$. By Theorem \ref{SumTo4For2PathColorEnds}, there is an $S\subseteq L(w)$ with $|S|\geq 2$ where, for any $b\in S$, there is an element of $\textnormal{End}_L(H_k, Q_k)$ using $a,b$ on $v, w$ respectively. Likwewise, by Claim \ref{SingleUSeL2Prec}, there is a $\tau'\in\textnormal{Link}(w'Pp_1)$ with $\tau'(p_1)=\phi(p_1)$ and $\tau'(w)\in S$. But then the union $\tau\cup\tau'$ is a proper $L$-coloring of its domain and lies in $\textnormal{Link}(P)$, contradicting our assumption that there is no such extension of $\phi$.  \end{claimproof}

Now, it follows from Claim \ref{AllbutOneIntEv} that, for all but at most three $k\in\{2, \cdots, r-1\}$, each of $x^{k-1}, x^k, x^{k+1}$ is even, so it suffices to prove that, for any such index $k$, there exist $T, \tau$ as in the statement of Theorem \ref{LinkPlusOneMoreVertx}, where $w=x^k$. Let $w=x^k$ be as above, but suppose there are no such $T, \tau$. Let $P':=H^k\cap P$. As $x^k$ is an internal $P$-peak and $P'$ has even length, there are two components of $P\setminus\mathring{P}'$, where each of these components has length at least two. Let $P_0, P_1$ denote these components and $Q^k:=v_0wv_1$, where, for each $i=0,1$, $P_i$ has endpoints $p_i, v_i$. 

Now, for each $i=0,1$, let $\mathcal{R}_i$ be the set of (not necessarily proper) 3-chords $R$ of $C$ such that $R$ has $v_iw$ as a terminal edge and $V(R\cap C)\subseteq V(P_i)$. By our triangulation conditions, each of $\mathcal{R}_0$ and $\mathcal{R}_1$ is nonempty. Each of $P_0, P_1$ is 3-consistent, and, since $G$ has no separating triangles, it follows that, for each $i=0,1$, there is a unique $R_i\in\mathcal{R}_i$ containing all the elements of $\mathcal{R}_i$. In particular, $G^{\textnormal{small}}_{R_0}\cup G^{\textnormal{small}}_{R_1}$ contains all the neighbors of $w$ among $\{x^1, \cdots, x^r\}$. For each $i=0,1$, let $R_i=q_iq_i'wv_i$ ( possibly $q_i=v_i$ and $R_i$ is a triangle). Finally, let $R^*:=q_0q_0'wq_1'q_1$. It follows from Theorem \ref{MainLinkingResultAlongFacialCycle} that, for each $i=0,1$, there is a $\psi_i\in\textnormal{Link}(p_iPq_i)$ with $\psi_i(p_i)=\phi(p_i)$. Let $a_i:=\phi_i(q_i)$. As $P'$ has length at least two, $\psi_0\cup\psi_1$ is a proper $L$-coloring of its domain.

\begin{claim}\label{ForAnyCOneFailure} For any $c\in L_{\phi_0\cup\phi_1}(w)$ using $a_0, a_1$ on $q_0, q_1$ respectively, there is at least one extension of $\sigma$ to an $L$-coloring of $V(R^*)$ which does not extend to $L$-color $G^{\textnormal{small}}_{R_0}\cup H^k\cup G^{\textnormal{small}}_{R_1}$.  \end{claim}

\begin{claimproof} Suppose not. Let $\tau$ be the extension of $\phi\cup\phi$ to an $L$-coloring of $\textnormal{dom}(\phi_0\cup\phi_1)\cup\{w\}$ obtained by coloring $w$ with $c$. Let $T:=\textnormal{dom}(\tau)\cup\textnormal{Sh}^2(p_0Pq_0)\cup\textnormal{Sh}^2(q_1Pp_1)\cup V(G^{\textnormal{small}}_{R_0}-q_0')\cup V(G_{R_1}-q_1')$. Then, by our assumption on $c$, we get that $T$ is $(L, \tau)$-inert in $G$. Furthermore, our choice of $R_0, R_1$ implies that each vertex of has an $L_{\tau}$-list of size at least three. In particular, $|L_{\tau}(q_0')|\geq 3$ and $|L_{\tau}(q_1')|\geq 3$ since each of $q_0', q_1'$ has precisely one neighbor in $p_0Pq_0\cup q_1Pp_1$. Thus, our choice of $w, \tau, T$ satisfies Theorem \ref{LinkPlusOneMoreVertx}, contradicting our assumption. \end{claimproof}

Now, for each $i=0,1$, we regard $G^{\textnormal{small}}_{R_i}$ as a planar embedding with outer cycle $q_iPv_iwq_i'$ and let $S_i$ denote the set of $c\in L_{\psi^i}(w)$ such that there is an element of $\textnormal{End}_L(q_iq_i'w, G^{\textnormal{small}}_{R_i})$ using $a_i, c$ on $q_i, w$ respectively.

\begin{claim} $|S_0\cap S_1|\geq 3$. \end{claim}

\begin{claimproof} Let $i\in\{0,1\}$, say $i=0$ for the sake of definitness. Since $|L(w)|\geq 5$, it suffices to show that $|S_0|\geq |L(w)|-1$. This is immediate if $G^{\textnormal{small}}_{R_0}$ is a triangle, since $S_0=L(w)\setminus\{a_0\}$ in that case, so suppose $G^{\textnormal{small}}_{R_0}$ is not a triangle. In particular, $R_0$ is a proper 3-chord of $C$ and $x^{k-1}\in V(G^{\textnormal{small}}_{R_0})$. Every chord of the outer cycle of $G^{\textnormal{small}}_{R_0}$ is incident to $q_0'$. If $G^{\textnormal{small}}_{R_0}$ is a broken wheel with principal vertex $q_0'$, then, since $G$ is short-inseparable, it follows that $q_0'=x^{k-1}$ and the path $G^{\textnormal{small}}_{R_0}-q_0'$ has odd length, as and $H^{k-1}-x^{k-1}$ has even length. In any case, it follows from Theorem \ref{SumTo4For2PathColorEnds} applied to $G^{\textnormal{small}}_{R_0}$ that $S_0$ consists of all but at most one vertex of $L(w)$, so we are done. \end{claimproof}

Note that $S_0\cap S_1\subseteq L_{\phi_0\cup\phi_1}(w)$, so it follows from Claim \ref{ForAnyCOneFailure} that, for each $c\in S_0\cap S_1$, there is an $L$-coloring $\sigma^s$ of $V(R^*)$ which does not extend to $L$-color $G^{\textnormal{small}}_{R_0}\cup H^k\cup G^{\textnormal{small}}_{R_1}$, where $\sigma^s$ uses $a_0, a_1$ on $q_0, q_1$ respectively. As $x^k$ is even, we have $d(v_0, v_1)>1$, so, by our choice of $c$, it follows that $\sigma^s$ extends to $L$-color $G^{\textnormal{small}}_{R_0}\cup G^{\textnormal{small}}_{R_1}$. Thus there is an $L$-coloring $\pi^s$ of $v_0wv_1$ which uses $s$ on $w$ and does not extend to $L$-color $H^k$. Since this holds for each $s\in S_0\cap S_1$ and $x^k$ is even, it follows that each internal vertex of the path $H^k-x^k$ has the same list of size three, where this list is precisely $S_0\cap S_1$, and furthermore, $S_0\cap S_1=\{\pi^s(v_0), s, \pi^1(v_1)\}$ for each $s\in S_0\cap S_1$. In particular, neither $G^{\textnormal{small}}_{R_0}$ nor $G^{\textnormal{small}}_{R_1}$ is a triangle, i.e $q_0, q_1\not\in N(w)$, and furthermore, $|(S_0\cap S_1)\cap L(v_i)|\geq 2$ for each $i=0,1$. As $|L(w)|\geq 5$, there is a $d\in L(w)\setminus (S_0\cap S_1)$, where $|L(v_i)\setminus\{d\}|\geq 3$ for each $i=0,1$. Yet, since every chord of the outer cycle of $G^{\textnormal{small}}_{R_i}$ is incident to $q_i'$, the latter condition implies that $d\in S_0\cap S_1$, a contradiction. \end{proof}

\section{Enclosures}\label{LinkEncSecM}

We now prove another intermediate result (Theorem \ref{GenThmTargetConnector}) that we need in order to deal with the boundary structures in Theorems \ref{FaceConnectionMainResult} and \ref{SingleFaceConnRes}.

\begin{defn}\label{MainEncDefBeforeMax1}\emph{Let $\mathcal{K}=[\Sigma, G, C, L]$ be a uniquely $4$-determined collar. Given a $w\in D_2(C)$, we say $w$ is degenerate if $G[N(w)\cap D_1(C)]$ is a path of length at most one. Otherwise we say $w$ is \emph{non-degenerate}. We define a \emph{$w$-enclosure} (in $\mathcal{K}$) to be a subgraph $Q$ of $G$ as follows: If $w$ is non-degenerate, then $Q$ is a proper 4-chord of $C$ with midpoint $w$, where the endpoints of $Q\setminus C$ are nonadjacent. If $w$ is degenerate, then $Q$ is a (not necessarily proper) 5-chord of $C$, where each vertex of $Q\setminus C$ lies in $D_1(C)$ and $N(w)\cap D_1(C)$ is contained in the middle edge of $Q\setminus C$.}  \end{defn}

In the setting above, if the collar $\mathcal{K}$ is clear from the context, then we just call $Q$ a $w$-enclosure. 

\begin{defn}\label{TargetDefForEncM} \emph{Let $\mathcal{K}=[\Sigma, G, C, L]$ be a uniquely 4-determined collar. Let $uw\in E(G)$, where $d(u, C)=3$ and $d(w, C)=2$. Given a $w$-enclosure $Q$, a $(Q, uw, \mathcal{K})$-\emph{target} is a partial $L$-coloring $\psi$ of $V(G^{\textnormal{small}}_Q+uw)$ such that $(V(G^{\textnormal{small}}_Q\setminus (Q\setminus\textnormal{dom}(\psi))), \psi)$ is an $L$-reduction, where $V(Q\cap C)\cup\{u\}\subseteq\textnormal{dom}(\psi)$, and}
\begin{enumerate}[label=\emph{\arabic*)}]
\itemsep-0.1em
\item\emph{If $w$ is non-degenerate, then $\textnormal{dom}(\psi)\cap V(Q\setminus C)=\{w\}$}; AND
\item\emph{If $w$ is non-degenerate, then, letting $yy'$ denote the middle edge of $Q\setminus C$, the set $\textnormal{dom}(\psi)\cap V(Q\setminus C)$ is a nonempty subset of $\{y, y'\}$ containing at least one neighbor of $w$.}
\end{enumerate}
 \end{defn}

In the setting above, if the collar $\mathcal{K}$ is clear from the context, then we just refer to $\psi$ as a $(Q, uw)$-target. 

\begin{theorem}\label{GenThmTargetConnector} Let $[\Sigma, G, C, L]$ be a uniquely 4-determined collar, where $G$ is short-inseparable and each vertex of $B_4(C)\setminus V(C)$ has a list of size at least five. Let $uw\in E(G)$, where $d(u, C)=3$ and $d(w, C)=2$. Let $Q$ be a $w$-enclosure such that, letting $p, p'$ be the vertices of $Q\cap C$, the following hold. 
\begin{enumerate}[label=\arabic*)]
\itemsep-0.1em
\item $u,w\not\in\textnormal{Sh}^4(C)$ and $|V(G^{\textnormal{large}}_Q\cap C)|>1$, and each vertex of $(G^{\textnormal{small}}_Q\cap C)\setminus Q$ has a list of size $\geq 3$; AND
\item Each of $p, p'$ has an nonempty list. and either $p=p'$ or at least one of $p, p'$ has a list of size at least three.
\end{enumerate}
Then 
\begin{enumerate}[label=\Alph*)]
\itemsep-0.1em
\item If $G^{\textnormal{small}}_Q$ induced in $G$, then there is a $(Q, uw)$-target; AND
\item Regarding $G^{\textnormal{small}}_Q$ as a planar embedding with outer cycle $(G^{\textnormal{small}}_Q\cap C)+Q$, if $p\neq p^*$ and $\phi\in\textnormal{Crown}_L(Q, G^{\textnormal{small}}_Q)$ is a proper $L$-coloring of its domain in $G$ with $d(w, \textnormal{dom}(\phi))\leq 1$, then $\phi$ extends to a $(Q, uw)$-target. 
\end{enumerate}
\end{theorem}

\begin{proof} This is immediate if $p=p'$, so now suppose $p\neq p'$. Let $C_Q:=(G^{\textnormal{small}}_Q\cap C)+Q$ and consider the rainbow $\mathcal{G}:=(G^{\textnormal{small}}_Q, C_Q, Q, L)$. In the case where $w$ is non-degenerate, we obtain A) by applying Theorem \ref{MainHolepunchPaperResulThm} to $\mathcal{G}$ and then just coloring $u$. In the case where $w$ is degenerate, we obtain A) similarly by applying Theorem \ref{ModifiedRes5ChordCaseDegen} to $\mathcal{G}$. Note that, once we have an element of $\textnormal{Crown}(Q, G^{\textnormal{small}}_Q)$, the same argument shows B). 
 \end{proof}

To prove Theorem \ref{MainCollarResultCutProc}, we actually prove something stronger, which is Theorem \ref{StrengthendVerManColRes}. To state this strengthening, we first introduce two more definitions.

\begin{defn}\label{X1X2Defns} \emph{Let $[\Sigma, G, C, L]$ be a uniquely 4-determined collar and let $uw$ be an edge of $G$, where $d(u, C)=3$ and $d(w, C)=2$. Given a $w$-enclosure $Q$ in $\mathcal{K}$, a $(Q, uw, \mathcal{K})$-\emph{pair} is a complete $L$-reduction $(S, \phi)$, where $S$ is a subset of $B_3(C)\cup\textnormal{Sh}^4(C)$, such that}
\begin{enumerate}[label=\emph{\arabic*)}]
\itemsep-0.1em
\item [\mylabel{}{\textit{X1)}}] $u\in\textnormal{dom}(\phi)$ and $S$ is topologically reachable from $C$; AND
\item [\mylabel{}{\textit{X2)}}] $(S\setminus\textnormal{Sh}^4(C))\cap D_3(C)=\{u\}$ and $(S\setminus\textnormal{Sh}^4(C))\cap D_2(C)=\{w\}$; AND
\item [\mylabel{}{\textit{X3)}}] $V(C)\subseteq S$ and $\textnormal{dom}(\phi)\cap V(G^{\textnormal{large}}_Q\cap C)=V(G^{\textnormal{large}}_Q\cap C)\setminus\textnormal{Sh}^2(C)$. 
\end{enumerate}
 \end{defn}

In the setting above, if $\mathcal{K}$ is clear from the context, then we just call $(S, \phi)$ a $(Q, uw)$-pair. 

\section{The Proof of Theorem \ref{MainCollarResultCutProc}}\label{PrelimSecWrto1}

We now prove our stronger version of Theorem \ref{MainCollarResultCutProc}. 

\begin{theorem}\label{StrengthendVerManColRes} Let $\mathcal{K}=[\Sigma, G, C, L]$ be a uniquely 4-determined collar, where $G$ is short-inseparable. Let $uw$ be an edge of $G$, where $d(u, C)=3$ and $d(w, C)=2$, and $u,w\not\in\textnormal{Sh}^4(C)$. Suppose further that every vertex of $C$ has a list of size at least three and every vertex of $B_4(C)\setminus V(C)$ has a list of size at least five. Then, for any maximal $w$-enclosure $Q$, there is a $(Q, uw)$-pair. \end{theorem}

\begin{proof} Suppose not and let $\mathcal{K}=(\Sigma, G, C, L)$ be a uniquely 4-determined collar violating Theorem \ref{StrengthendVerManColRes}, where $|V(G)|$ is minimized with respect to this condition. We also suppose that $|E(G)|$ is maximal among all vertex-minimal counterexamples. Thus, there is a $uw\in E(G)$ satisfying the conditions in the statement of Theorem \ref{StrengthendVerManColRes}, such that, for some maximal $w$-enclosure $Q$, there is no $(Q, uw)$-pair. For convenience, we assume, by removing some colors from some lists if necessary, that any vertex of $G$ with a list of size at least five has a list of size precisely five. 

\begin{claim}\label{NoChordofCMinCounterCLMain10} There is no chord of $C$. \end{claim}

\begin{claimproof}  Suppose there is a chord of $xy$ of $C$. Let $H=G^{\textnormal{large}}_{xy}$ and $H'=G^{\textnormal{small}}_{xy}$.  Let $C^*$ be the cycle $(C\cap H)+xy$ and let $\mathcal{K}^*=[\Sigma, H, C^*, L]$. Note that $\mathcal{K}^*$ is also a uniquely 4-determined collar. Furthermore, $H$ is still short-inseparable. Note that, in $\mathcal{K}^*$, $Q$ is still a maximal $w$-enclosure. This is true even if $xy\in E(G^{\textnormal{small}}_Q)$. In particular, $w$ is degenerate in $\mathcal{K}$ if and only if it is degenerate in $\mathcal{K}^*$. We note that the edge $uw$ and collar $\mathcal{K}^*$ satisfy all the other conditions in the statement of Theorem \ref{StrengthendVerManColRes}. As $|V(H)|<|V(G)|$, there is a $(Q, uw, \mathcal{K}^*)$-pair $(S, \phi)$. We claim now that $(S\cup V(H'), \phi)$ is a $(Q, uw, \mathcal{K})$-pair. If we show this, we are done, as we contradict our choice of $Q$. Since $V(C^*)\subseteq S$, we have $V(C)\subseteq S\cup V(H')$. Furthermore, since $S$ is $(L, \phi)$-inert in $H$, it follows from Theorem \ref{thomassen5ChooseThm} that $S\cup V(H')$ is $(L, \phi)$-inert in $G$. Now it just suffices to check that  $V(G^{\textnormal{large}}_Q\cap C)\setminus\textnormal{Sh}^2(C, \mathcal{K})=V(H^{\textnormal{large}}_Q\cap C^*)\setminus\textnormal{Sh}^2(C^*, \mathcal{K}^*)$. It is immediate that the rest of Definition \ref{X1X2Defns} is satisfied. Suppose first that $xy\in E(G^{\textnormal{large}}_Q)$. Thus, $G^{\textnormal{small}}_Q\subseteq H$.  We have $V(H')\subseteq\textnormal{Sh}^2(C, \mathcal{K})$, so the above equality is satisfied. Now suppose $xy\in E(G^{\textnormal{small}}_Q)$. Thus $G^{\textnormal{large}}_Q\subseteq H$, and $G^{\textnormal{large}}_Q=H^{\textnormal{large}}_Q$, and, again, the claimed equality is satisfied. \end{claimproof}

Later in the proof, we sometimes construct $(Q, uw)$-pairs by applying Theorem \ref{GenThmTargetConnector} and extending a $(Q, uw)$-target to a larger domain. Whenever we do this, we first need to check $G^{\textnormal{small}}_Q$ is an induced subgraph of $G$. The following observation is immediate from the maximality of $Q$, together with the fact that $G$ is short-inseparable. 

\begin{claim}\label{EitherInducedOrAtMostOnePD} If $G^{\textnormal{small}}_Q$ is not an induced subgraph of $G$, then there is precisely one edge of $E(G)\setminus E(G^{\textnormal{small}}_Q)$ with both endpoints in $V(G^{\textnormal{small}}_Q)$, and this edge is precisely $G^{\textnormal{large}}_Q\cap C$. \end{claim}

Applying Theorem \ref{RainbowNonEqualEndpointColorThm}, we have the following. 

\begin{claim}\label{MainConsLemmaIntStepGLargGGeq1} $|E(G^{\textnormal{large}}_Q\cap C)|\geq 1$ \end{claim}

\begin{claimproof} Suppose not. Thus, $w$ is degenerate and $Q$ is a 5-cycle, where $G^{\textnormal{large}}_Q\cap C=x$ for some $x\in V(C)$. Let $C-x=v_1\cdots v_m$ and let $xw, xw'$ be the edges of $Q$ incident to $x$. Regarding $G^{\textnormal{small}}_Q$ as a planar embedding with outer cycle $C$, w let $H$ be a planar embedding obtained from $G^{\textnormal{small}}_Q$ by replacing $x$ with a pair of vertices $y, y'$, where $N_H(y)=N_G(x)\cap V(G^{\textnormal{small}}_Q-w')$ and $N_H(y')=N_G(x)\cap V(G^{\textnormal{small}}_Q-w)$ and $H$ has outer cycle $D:=yv_1\cdots v_my'w'\mathring{Q}w$. Let $L^*$ be a list-assignment for $H$ where $L^*(y)=L^*(y')=L(x)$ and let $P:=yw\mathring{Q}w'y'$. Now we apply Theorem \ref{RainbowNonEqualEndpointColorThm}. The fact that $G$ is short-inseparable implies both that $H$ has no chord of $P$ and, letting $\mathcal{G}$ be the rainbow $(H, D, P, L)$, there is no $\mathcal{G}$-obstruction. Thus, there is a $\phi\in\textnormal{Crown}(P, H)$ using the same color on $y, y'$, where $\textnormal{dom}(\phi)$ contains at least one neighbor of $w$. Since $G^{\textnormal{small}}_Q$ is an induced in $G$, there is a partial $L$-coloring $\phi^*$ of $V(G^{\textnormal{small}}_Q)$ with $\textnormal{dom}(\phi^*)=(\textnormal{dom}(\phi)\setminus\{y, y'\})\cup\{x\}$, where $\phi^*(x)=\phi^*(y)=\phi(y')$ and otherwise $\phi^*=\phi$. Let $S:=(V(G^{\textnormal{small}}_Q)\setminus (Q\setminus\textnormal{dom}(\psi)))\cup\{u, w\}$. Analogous to Theorem \ref{GenThmTargetConnector}, it is straightfoward to check that $\phi^*$ extends to a $(Q, uw)$-target $\psi$, and $(S, \psi)$ is a $(Q, uw)$-pair, contradicting our assumption. \end{claimproof}

We now note the following. 

\begin{claim}\label{IfTargetFailsInertFails} Let $\sigma$ be a $(Q, uw)$-target and let $\phi$ be an $L$-coloring of $V(G^{\textnormal{large}}\cap C)\setminus\textnormal{Sh}^2(C)$ such that $\sigma\cup\phi$ is a proper $L$-coloring of its domain. Then $V(G^{\textnormal{large}}_Q)\cap\textnormal{Sh}^2(C)$ is not $(L, \phi)$-inert in $G$. \end{claim}

\begin{claimproof} Suppose $V(G^{\textnormal{large}}_Q)\cap\textnormal{Sh}^2(C)$ is $(L, \psi)$-inert in $G$. Let $H$ be the subgraph of $G$ induced by the vertices of $(C\cup (G^{\textnormal{small}}_Q+uw))\setminus (Q\setminus\textnormal{dom}(\sigma))$. We claim now that $(V(H), \sigma\cup\phi)$ is a $(Q, uw)$-pair.  Note that $V(H)$ is topologically reachable from $C$. It follows from our assumption that $V(H)$ is $(L, \sigma\cup\phi)$-inert in $G$. The only other nontrivial condition to check among those of Definition \ref{X1X2Defns} is in Subclaim \ref{If<3WithoutTar<3With} below.

\vspace*{-8mm}
\begin{addmargin}[2em]{0em}
\begin{subclaim}\label{If<3WithoutTar<3With} Every vertex of $D_1(H)$ has an $L_{\sigma\cup\phi}$-list of size at least three. \end{subclaim}

\begin{claimproof}  Suppose not. Thus, there is a $z\in V(G^{\textnormal{large}}_Q\setminus C)$ with $|L_{\sigma\cup\phi}(z)|<3$, where $z$ has a neighbor in $H$. If $N(z)\cap\textnormal{dom}(\sigma\cup\phi)\subseteq\textnormal{dom}(\sigma)$, then it follows from Definition \ref{TargetDefForEncM} that $|L_{\sigma}(z)|\geq 3$. Thus, $z$ has at least one neighbor in $\textnormal{dom}(\phi)\setminus\textnormal{dom}(\sigma)$, so $z\in D_1(C)\cap V(G^{\textnormal{large}}_Q)$. As every subpath of $G^{\textnormal{large}}_Q\cap C$ is 2-consistent, it follows that every vertex of $D_1(C)\cap V(G^{\textnormal{large}}_Q)$ has at most two neighbors in $G^{\textnormal{large}}_Q\cap C)\setminus\textnormal{Sh}^2(C)$. Thus, $|L_{\phi}(z)|\geq 3$, so $z$ has a neighbor in $\textnormal{dom}(\sigma)\setminus\textnormal{dom}(\phi)$. It follows from the definition of $\sigma$ that, if $w$ is degenerate, then at most the endpoints of the middle edge of $Q\setminus C$ are colored by $\sigma$, and, if $w$ is non-degenerate, then only the middle vertex of $Q\setminus C$ is colored by $\sigma$. Since $z$ has at least one neighbor in $\textnormal{dom}(\phi)\setminus V(Q)$, it follows from the maximality of $Q$ that $z$ has no neighbor in $\textnormal{dom}(\sigma)\setminus\textnormal{dom}(\phi)$, a contradiction. \end{claimproof}\end{addmargin}

Now we check X3). Since $V(G^{\textnormal{large}}_Q\cap C)\cap\textnormal{dom}(\sigma\cup\phi)=VG^{\textnormal{large}}_Q\cap C)\cap\textnormal{dom}(\phi)=V(G^{\textnormal{large}}\cap C)\setminus\textnormal{Sh}^2(C)$, it follows that $(V(H), \sigma\cup\phi)$ is indeed a $(Q, uw)$-pair, contradicting our assumption on $Q$. This proves Claim \ref{IfTargetFailsInertFails}. \end{claimproof}

We now analyze the vertices of distance one from $C$. The following observation is immediate from the fact that $Q$ is a maximal $w$-enclosure, together with our assumption that $u,w\not\in\textnormal{Sh}^4(C)$.

\begin{claim} For any $y\in D_1(C)$, if $y$ has a neighbor in $G^{\textnormal{large}}_Q\cap C$, then $N(y)\cap V(C)\subseteq V(G^{\textnormal{large}}_Q\cap C)$. Furthermore, for any distinct $x, v\in N(x)\cap V(C)$, the graph $V(G^{\textnormal{small}}_{xyv}\cap G^{\textnormal{small}}_Q)$ is a subset of $Q\cap C$ of size at most one.  \end{claim}

To analyze the vertices of $D_1(C)$, we introduce the following notation. We define $\textnormal{Sp}(C)$ to be the set of $y\in D_1(C)$ such that $y$ has at least one neighbor in $G^{\textnormal{large}}\cap C$ and $|N(y)\cap V(C)|\geq 2$. For each $y\in\textnormal{Sp}(C)$, we let $P_y$ be the unique 2-chord of $C$ which has midpoint $y$ and satisfies $N(y)\cap V(C)\subseteq V(G^{\textnormal{large}}_Q)$. Furthermore, to avoid clutter, for each $y\in\textnormal{Sp}(C)$, we let $G_y$ denote $G^{\textnormal{small}}_{P_y}$. We now prove that, for each $y\in\textnormal{Sp}(C)$, the graph $G_y$ has a very simple structure. We first introduce the following notation. We let $G^{\textnormal{large}}_Q\cap C=x_1\cdots x_n$ for some $n$. Possibly $Q$ is an improper 5-chord of $C$, but, in any case, $n\geq 2$ by Claim \ref{MainConsLemmaIntStepGLargGGeq1}. 

\begin{claim}\label{AllVertSame3ListSubCL} If $n\neq 2$, then the following hold: All the vertices of $ G^{\textnormal{large}}_Q\cap C$ have the same $L$-list of size three, and, for each $y\in\textnormal{Sp}(C)$, the graph $G_y$ is a broken wheel with principal path $P_y$, where $G_y-y$ has length at most three. 
\end{claim}

\begin{claimproof} Suppose $n\geq 3$. We first note the following.

\vspace*{-8mm}
\begin{addmargin}[2em]{0em}
\begin{subclaim}\label{YSpanNoUnivCol} For any $y\in\textnormal{Sp}(C)$ and any endpoint $x$ of $P_y$, no color of $L(x)$ is $(P_y, G_y)$-universal. \end{subclaim}

\begin{claimproof} Suppose that there is an endpoint $x$ of $P_y$ such that there is $(P_y, G_y)$-universal color of $L(x)$. Let $j,k$ be indices with $1\leq j<k\leq n$, where $G_y-y=x_jx_{j+1}\cdots x_k$. This is illustrated in Figure \ref{ElipseFirstOFigureSpM1}. Say without loss of generality that there is a  $(P_y, G_y)$-universal color $a\in L(x_j)$. By Theorem \ref{MainLinkingResultAlongFacialCycle}, there is a $\psi\in\textnormal{Link}(x_1x_2\cdots x_j)$ with $\psi(x_j)=a$. As $n\geq 3$, the graph $G^{\textnormal{small}}_Q$ is induced in $G$. By A) of Theorem \ref{GenThmTargetConnector}, there is a $(Q, uw)$-target $\sigma$ with $\sigma(x_1)=\psi(x_1)$. Again by Theorem \ref{MainLinkingResultAlongFacialCycle}, there is a $\phi\in\textnormal{Link}(x_kx_{k+1}\cdots x_n)$ with $\phi(x_n)=\sigma(x_n)$. The union $\sigma\cup\psi\cup\phi$ is a proper $L$-coloring of its domain. This is true even if $G_y$ is a triangle, as $a$ is $(P_y, G_y)$-universal. The domain of $\psi\cup\phi$ is precisely $V(G^{\textnormal{large}}_Q\cap C)\setminus\textnormal{Sh}^2(C)$. It follows from the definition of $\phi$ and $\psi$, together with our choice of $a$, that $V(G^{\textnormal{large}}_Q\cap C)\cap\textnormal{Sh}^2(C)$ is $(L, \psi\cup\phi)$-inert in $G$, contradicting Claim \ref{IfTargetFailsInertFails}. \end{claimproof}\end{addmargin}

By Subclaim \ref{YSpanNoUnivCol}, together with Observation \ref{CorMainEitherBWheelAtM1ColCor}, we immediately have the following: 

\vspace*{-8mm}
\begin{addmargin}[2em]{0em}
\begin{subclaim}\label{EachYSpanBWheelPPathM} For each $y\in\textnormal{Sp}(C)$, $G_y$ is a broken wheel with principal path $P_y$, and furthermore, if $G_y-y$ has length at least three, then, letting $G_y-y=x_j\cdots x_k$ for some $1\leq j<k\leq n$, we have $L(x_j)\subseteq L(x_{j+1})\cap L(x_{j+2})$ and $L(x_k)\subseteq L(x_{k-1})\cap L(x_{k-2})$. \end{subclaim}\end{addmargin}

The bulk of the proof of Claim \ref{AllVertSame3ListSubCL} consists of the key step in Subclaim \ref{SubMainSpanCAtM3}.

\begin{center}\begin{tikzpicture}
\draw (0,0) ellipse (8cm and 3cm);
\node[shape=circle, draw=black, inner sep=0pt, minimum size=0.55cm] (W) at (5, 0) {$w$};
\node[shape=circle, draw=black, inner sep=0pt, minimum size=0.55cm] (U) at (3.5, 0) {$u$};
\node[shape=circle, draw=white, inner sep=0pt, minimum size=0.55cm] (GSm) at (6.5, 0) {$G^{\textnormal{small}}_Q$};
\node[shape=circle, fill=black, inner sep=0pt, minimum size=0.25cm] (X1) at (6, 1.97) {};
\node[shape=circle, fill=black, inner sep=0pt, minimum size=0.25cm] (XN) at (6, -1.97) {};
\node[shape=circle, fill=black, inner sep=0pt, minimum size=0.25cm] (Z) at (5.5, 0.985) {};
\node[shape=circle, fill=black, inner sep=0pt, minimum size=0.25cm] (Z') at (5.5, -0.985) {};
\draw[-, line width=1.8pt, color=red] (X1) to (Z) to (W) to (Z') to (XN);
\node[shape=circle, draw=white, inner sep=0pt, minimum size=0.25cm] (X1L) at (6, 2.3) {\small $x_1$};
\node[shape=circle, draw=white, inner sep=0pt, minimum size=0.25cm] (XNL) at (6, -2.3) {\small $x_n$};

\node[shape=circle, fill=black, inner sep=0pt, minimum size=0.25cm] (XK) at (-2, 2.91) {};

\node[shape=circle, draw=white, inner sep=0pt, minimum size=0.25cm] (XK0L) at (-2.1, 3.4) {\small $x_k$};

\node[shape=circle, fill=black, inner sep=0pt, minimum size=0.25cm] (XK>) at (1, 2.97) {};

\node[shape=circle, fill=black, inner sep=0pt, minimum size=0.25cm] (XK<) at (-1, 2.97) {};

\node[shape=circle, fill=black, inner sep=0pt, minimum size=0.25cm] (XJ) at (2, 2.91) {};

\node[shape=circle, draw=white, inner sep=0pt, minimum size=0.25cm] (XJL) at (2.1, 3.4) {\small $x_j$};

\node[shape=circle, fill=white, inner sep=0pt, minimum size=0.55cm] (XL0+2) at (0, 3) {$\hdots$};

\node[shape=circle, draw=black, inner sep=0pt, minimum size=0.55cm] (Y0) at (0, 1.8) {\small $y$};

\draw[-] (U) to (W);

\draw[] (XJ) to (Y0) to (XK);

\draw[] (XK<) to (Y0) to (XK>);

\end{tikzpicture}\captionof{figure}{The ellipse enclosing the diagram represents the cycle $C$}\label{ElipseFirstOFigureSpM1}\end{center}

\vspace*{-8mm}
\begin{addmargin}[2em]{0em}
\begin{subclaim}\label{SubMainSpanCAtM3} For each $y\in\textnormal{Sp}(C)$, the path $G_y-y$ has length at most three. \end{subclaim}

\begin{claimproof}  Suppose $G_y-y$ has length greater than three. Let $G_y-y=x_j\cdots x_k$ for some $1\leq j<k\leq n$, where $k-j>3$. Now, let $H$ be a graph obtained from $G$ by deleting $x_{j+1}, x_{j+2}$ and replacing them with a new edge $x_jx_{j+3}$. Let $C^*$ be the cycle obtained from $C$ by the procedure above, i.e $C^*=(C-\{x_{j+1}, x_{j+2})+x_jx_{j+3}$. Note that, since $|V(G_y)|>3$ and $G$ is short-inseparable, we have $|V(C)|\geq 5$, so $C^*$ is indeed a cycle, i.e $|V(C^*)|\geq 3$. Let $\mathcal{K}^*$ be the collar $[\Sigma, H, C^*, L]$. It is clear that $\mathcal{K}^*$ is still a uniquely 4-determined collar. It is also clear that $w\in D_2(C^*)\setminus\textnormal{Sh}^4(C^*, \mathcal{K}^*)$ and $u\in D_3(C^*)\setminus\textnormal{Sh}^4(C^*, \mathcal{K}^*)$. Now, we want to apply the minimality of $G$ to $\mathcal{K}^*$, but first we need to check that $H$ is still short-inseparable. If $\textnormal{ew}(H)\leq 4$, then, since $\textnormal{ew}(G)>4$, there is a noncontractible cycle of $H$ of length at most four which contains a vertex of $C^*$, which is false, since $\mathcal{K}^*$ is uniquely 4-determined. Suppose $H$ is not short-inseparable. By Subclaim \ref{EachYSpanBWheelPPathM}, there is no $y'\in\textnormal{Sp}(C)\setminus\{y\}$ which is also adjacent to $x_j, x_k$, and, by Claim \ref{NoChordofCMinCounterCLMain10}, there is no chord of $C$. As $G$ is short-inseparable, but $H$ is not, $G$ contains a 2-chord of $C$ with endpoints $x_j, x_k$, and, in $H$, $x_jx_{j+3}$ is an edge of a 4-cycle which separates $y$ from another vertex of $H$, which is false, as $k-j>3$. Thus, $H$ is indeed short-inseparable. Since $|V(H)|<|V(G)$ and $Q$ is a still a maximal $w$-enclosure in $\mathcal{K}^*$, it follows from the minimality of $\mathcal{K}$ that there is a $(Q, uw, \mathcal{K}^*)$-pair $(S, \phi)$. Since $x_{j+1}, x_{j+2}\in\textnormal{Sh}^2(C, \mathcal{K})$, we obtain the following.
\begin{equation}\tag{Eq1}\label{HSmallGSmallDag} (H^{\textnormal{large}}_Q\cap C^*)\setminus\textnormal{Sh}^2(C^*, \mathcal{K}^*)=(G^{\textnormal{large}}_Q\cap C)\setminus\textnormal{Sh}^2(C, \mathcal{K})\end{equation}
Now, we let $H^*=H^{\textnormal{small}}_{x_jyx_k}$. Note that $H^*$ is a broken wheel with principal path $x_jyx_k$. In particular, $H^*-y=x_jx_{j+3}\cdots x_k$. By (\ref{HSmallGSmallDag}) ,  $V(H^*)\setminus\textnormal{Sh}^2(C^*, \mathcal{K}^*)=\{x_j, x_k\}$. We claim now that $(S\cup\{x_{j+1}, x_{j+2}\}, \phi)$ is a $(Q, uw, \mathcal{K})$-pair. If we show this, then we are done, as this contradicts our assumption on $Q$. 

Since $V(C^*)\subseteq S$, we have $V(C)\subseteq S\cup\{x_{j+1}, x_{j+2}\}$. The only other nontrivial thing to check from Definition \ref{X1X2Defns} is that $S\cup\{x_{j+1}, x_{j+2}\}$ is $(L, \phi)$-inert in $G$. Since $(S, \phi)$ is a $(Q, uw, \mathcal{K}^*)$-pair, we have $\textnormal{dom}(\phi)\cap V(H^{\textnormal{large}}_Q\cap C^*)=V(H^{\textnormal{large}}_Q\cap C^*)\setminus\textnormal{Sh}^2(C^*, \mathcal{K}^*)$, so we get $\{x_j, x_k\}\subseteq\textnormal{dom}(\phi)\cap V(H^*)\subseteq\{x_j, y, x_k\}$. Now, suppose that $S\cup\{x_{j+1}, x_{j+2}\}$ is not $(L, \phi)$-inert in $G$. Note that $\phi$ is a proper $L$-coloring of its domain in $H$, and therefore also a proper $L$-coloring of its domain in $G$. As indicated above, we have $V(C^*)\subseteq S$ and $\textnormal{dom}(\phi^*)\cap V(H^*)=\{x_j, x_k\}$. Since $S$ is $(L, \phi)$-inert in $H$, but $S\cup\{x_{j+1}, x_{j+2}\}$ is not $(L, \phi)$-inert in $G$, it follows that there exists an $L$-coloring $\psi$ of $x_jyx_k$ which extends to an $L$-coloring of $H^*$ but does not extend to an $L$-coloring of $G_y$. Now let $\psi^*$ be an extension of $\psi$ to an $L$-coloring of $H^*$. Note that $\psi^*$ is also a proper $L$-coloring of its domain in $G$, which is $G_y-\{x_{j+1}, x_{j+2}\}$, and furthermore, since $\psi^*$ is an $L$-coloring of $H^*$, we have $\psi^*(x_j)\neq\psi^*(x_{j+3})$. By Subclaim \ref{EachYSpanBWheelPPathM}, we have $L(x_j)\subseteq L(x_{j+1})\cap L(x_{j+2})$. Now, since $\psi^*(x_j)$ is distinct from both $\psi^*(x_{j+1})$ and $\psi^*(x_{j+3})$, we can extend $\psi^*$ to an $L$-coloring of $G_y$ by first coloring $x_{j+2}$ with $\psi^*(x_{j+2})$, leaving a color for $x_{j+1}$, a contradiction. We conclude that $(S\cup\{x_{j+1}, x_{j+2}\}$ is $(L, \phi)$-inert in $G$. Now we just need to check that $(S\cup\{x_{j+1}, x_{j+2}\}, \phi)$ satisfies X3). Since $(S, \phi)$ is $(Q, uw, \mathcal{K}^*)$-pair, this is immediate from (\ref{HSmallGSmallDag}), together with the fact that $H^{\textnormal{small}}_Q=G^{\textnormal{small}}_Q$. This proves Subclaim \ref{SubMainSpanCAtM3}. \end{claimproof}\end{addmargin}

To finish the proof of Claim \ref{AllVertSame3ListSubCL}, it suffices to show that all of the vertices of $ G^{\textnormal{large}}_Q\cap C$ have the same $L$-list of size three. Suppose not. Without loss of generality, there is an $i\in\{1, \cdots n-1\}$ and an $a\in L(x_i)$ with $|L(x_{i+1})\setminus\{a\}|\geq 3$.

\vspace*{-8mm}
\begin{addmargin}[2em]{0em}
\begin{subclaim}\label{xixi+1UniqCommSub} $x_i, x_{i+1}$ have a unique common neighbor in $D_1(C)$. \end{subclaim}

\begin{claimproof} Suppose not. By Subclaim \ref{EachYSpanBWheelPPathM}, there is no $y\in\textnormal{Sp}(C)$ such that $x_ix_{i+1}$ is an edge of $G_y$. By Theorem \ref{MainLinkingResultAlongFacialCycle}, there is a $\phi\in\textnormal{Link}(x_1\cdots x_i)$ with $\phi(x_i)=a$. By A) of Theorem \ref{GenThmTargetConnector}, since $n\geq 3$, there is a $(Q, uw)$-target $\sigma$ with $\sigma(x_1)=\phi(x_1)$. Again by Theorem \ref{MainLinkingResultAlongFacialCycle}, there is a $\psi\in\textnormal{Link}(x_{i+1}\cdots x_n)$ with $\psi(x_n)=\sigma(x_n)$. Since $a\not\in L(x_{i+1})$ and $C$ is induced, the union $\phi\cup\psi\cup\sigma$ is a proper $L$-coloring of its domain. Furthermore, by our assumption on the edge $x_ix_{i+1}$, we have $\textnormal{Sh}^2(C)\cap V(G^{\textnormal{large}}_Q\cap C)=\textnormal{Sh}^2(x_1\cdots x_i)\cup\textnormal{Sh}^2(x_{i+1}\cdots x_n)$. Thus, the domain of $\psi\cup\phi$ is precisely $V(G^{\textnormal{large}}_Q\cap C)\setminus\textnormal{Sh}^2(C)$, contradicting Claim \ref{IfTargetFailsInertFails}. \end{claimproof}\end{addmargin}

Applying Subclaim \ref{xixi+1UniqCommSub}, we let $y$ be the unique common neighbor of $x_i$ and $x_{i+1}$ in $D_1(C)$. Thus, $y\in\textnormal{Sp}(C)$ and $x_ix_{i+1}$ is a subpath of $G_y-y$. By Subclaim \ref{SubMainSpanCAtM3}, $G_y-y$ has length at most three. By Subclaim \ref{YSpanNoUnivCol}, for each endpoint $x$ of $G_y-y$, no color of $L(x)$ is $(P_y, G_y)$-universal with respect to the list-assignment $L$. Thus, $|L(x)|=3$ and $L(x)$ is precisely equal to the list of its lone neighbor on the path $G_y-y$ (this is true even if $G_y$ is a triangle). If $G_y-y$ has length at most two, this implies that all the vertices of $G_y-y$ have the same 3-list, contradicting our assumption on $x_ix_{i+1}$. If $G_y-y$ has length precisely three, then, by Subclaim \ref{EachYSpanBWheelPPathM}, we again get that all the vertices of $G_y-y$ have the same 3-list, contradicting our assumption on $x_ix_{i+1}$. This completes the proof of Claim \ref{AllVertSame3ListSubCL}.  \end{claimproof}

Note that Claim \ref{AllVertSame3ListSubCL} implies the following.

\begin{claim}\label{AnyDiffColorExtUseLinkM5} Let $y\in\textnormal{Sp}(C)$, where $G_y-y$ is a path of odd length. Let $x, x'$ be the endpoints of $P_y$ and let $\sigma$ be an $L$-coloring of $\{x, x'\}$ which uses different colors on $x, x'$. Then any extension of $\sigma$ to an $L$-coloring of $V(P_y)$ extends to $L$-color all of $G_y$. \end{claim}

\begin{claimproof} This is is immediate if $G_y$ is a triangle, so suppose that $|E(G_y-y)|>1$. Let $\tau$ be an extension of $\sigma$ to an $L$-coloring of $V(P_y)$, and let $a=\sigma(x)$ and $b=\sigma(x')$. By Claim \ref{AllVertSame3ListSubCL}, $G_y-y$ is a path of length three, and $|V(G_y)|=5$, and all the vertices of $G_y-y$ have the same list of size three. As $\tau(y)\neq a$ and $a\neq b$, we can extend $\tau$ to an $L$-coloring of $G_y$ by first using $a$ on the lone neighbor of $x'$ in $G_y-y$. \end{claimproof}

Later, we prove that $x_1\cdots x_n$ is a path of length either one or three and that $|\textnormal{Sp}(C)|\leq 1$, where the lone vertex of $\textnormal{Sp}(C)$, if it exists, is adjacent to all the vertices of $x_1\cdots x_n$. We actually prove something slightly stronger, which is stated in Claim \ref{LoneVertSpanPathOddLenCl}. To prove this, we first prove two intermediate results. Firstly, given an edge $e\in E(G^{\textnormal{large}}_Q\cap C)$, we say that $e$ is \emph{uncovered} if no vertex of $\textnormal{Sp}(C)$ is adjcent to both endpoints of $e$.

\begin{claim}\label{InterMResAtMostXSpanC} $|\{e\in E(G_Q^{\textnormal{large}}\cap C): e\ \textnormal{is uncovered}\}|+|\{y\in\textnormal{Sp}(C): G_y-y\ \textnormal{has odd length}\}|\leq 1$
\end{claim}

\begin{claimproof} Suppose Claim \ref{InterMResAtMostXSpanC} does not hold. Thus, there exist  indices $1\leq k^0<\ell^0\leq k^1<\ell^1\leq n$ such that, for each $r\in\{0,1\}$, precisely one of the following holds:
\begin{enumerate}[label=\roman*)]
\itemsep-0.1em
\item There is a $y^r\in\textnormal{Sp}(C)$ where $G_{y^r}-y^r=x_{k^r}x_{k^r+1}\cdots x_{\ell^r}$, and $x_{k^r}x_{k^r+1}\cdots x_{\ell^r}$ is a path of odd length; OR
\item $\ell^r=k^r+1$ and the vertices $x_{k^r}, x_{\ell^r}$ have no common neighbor in $\textnormal{Sp}(C)$
\end{enumerate}

In any case, each of $\ell^0-k^0$ and $\ell^1-k^1$ is odd, and it follows from Claim \ref{AllVertSame3ListSubCL} that, for each $r\in\{0,1\}$, $\ell^r-k^r$ is either one or three. This is illustrated in Figure \ref{ElipseFigureSpM1} in the case where there exist $y^0, y^1$ and $G_{y^0}-y^0$ has length three and $G_{y^1}-y^1$ has length one. Since $G^{\textnormal{large}}\cap C$ is not just an edge, it follows from Claim \ref{EitherInducedOrAtMostOnePD} that $G^{\textnormal{small}}_Q$ is induced in $G$. Possibly $x_1=x_n$, but, in any case, since $|L(x_n)|=3$, it follows from Theorem \ref{GenThmTargetConnector} that there is $(Q, uw)$-target $\sigma$. Let $a=\sigma(x_1)$ and $b=\sigma(x_n)$. By Theorem \ref{MainLinkingResultAlongFacialCycle}, there exist a $\phi^0\in\textnormal{Link}(x_1\cdots x_{k^0})$ and a $\phi^1\in\textnormal{Link}(x_{\ell_1}\cdots x_n)$ such that $\phi^0(x_1)=a$ and $\phi^1(x_n)=b$. Note that $\sigma\cup\phi^0\cup\phi^1$ is a proper $L$-coloring of its domain. Let $c^0:=\phi^0(x_{k^0})$ and $c^1:=\phi^1(x_{\ell^1})$. Since $|L(x_{\ell^0})\setminus\{c^0\}|\geq 2$ and $|L(x_{k^1})\setminus\{c^1\}|\geq 2$, it follows from Theorem \ref{MainLinkingResultAlongFacialCycle} that there is a $\psi\in\textnormal{Link}(x_{\ell^0}\cdots x_{k^1})$ with $\psi(x_{\ell^0})\neq c^0$ and $\psi(x_{k^1})\neq c^1$. Furthermore, $\sigma\cup\phi^0\cup\psi\cup\phi^1$ is a proper $L$-coloring of its domain, and the domain of $\phi^0\cup\psi\cup\phi^1$ is precisely $V(G^{\textnormal{large}}_Q\cap C)\setminus\textnormal{Sh}^2(C)$. Let $\tau=\phi^0\cup\psi\cup\phi^1$. It follows from Claim \ref{IfTargetFailsInertFails} that $V(G^{\textnormal{large}}_Q\cap C)\cap\textnormal{Sh}^2(C)$ is not $(L, \tau)$-inert in $G$. Thus,  $\tau$ extends to an $L$-coloring of $\textnormal{dom}(\tau)\cup\{y^0, y^1\}$ which does not extends to $L$-color all of $G_{y^0}\cup G_{y^1}$. Yet, by Claim \ref{AnyDiffColorExtUseLinkM5}, this contradicts out choice of colors for $x_{\ell^0}, x_{k^1}$. \end{claimproof}

\begin{center}\begin{tikzpicture}
\draw (0,0) ellipse (8cm and 3cm);
\node[shape=circle, draw=black, inner sep=0pt, minimum size=0.55cm] (W) at (5, 0) {$w$};
\node[shape=circle, draw=black, inner sep=0pt, minimum size=0.55cm] (U) at (3.5, 0) {$u$};
\node[shape=circle, draw=white, inner sep=0pt, minimum size=0.55cm] (GSm) at (6.5, 0) {$G^{\textnormal{small}}_Q$};
\node[shape=circle, fill=black, inner sep=0pt, minimum size=0.25cm] (X1) at (6, 1.97) {};
\node[shape=circle, fill=black, inner sep=0pt, minimum size=0.25cm] (XN) at (6, -1.97) {};
\node[shape=circle, fill=black, inner sep=0pt, minimum size=0.25cm] (Z) at (5.5, 0.985) {};
\node[shape=circle, fill=black, inner sep=0pt, minimum size=0.25cm] (Z') at (5.5, -0.985) {};
\draw[-, line width=1.8pt, color=red] (X1) to (Z) to (W) to (Z') to (XN);
\node[shape=circle, draw=white, inner sep=0pt, minimum size=0.25cm] (X1L) at (6, 2.3) {\small $x_1$};
\node[shape=circle, draw=white, inner sep=0pt, minimum size=0.25cm] (XNL) at (6, -2.3) {\small $x_n$};

\node[shape=circle, fill=black, inner sep=0pt, minimum size=0.25cm] (XL0) at (-6, 1.97) {};

\node[shape=circle, draw=white, inner sep=0pt, minimum size=0.25cm] (XL0L) at (-6, 2.35) {\small $x_{\ell^0}$};

\node[shape=circle, fill=black, inner sep=0pt, minimum size=0.25cm] (XK0) at (-3.5, 2.7) {};

\node[shape=circle, draw=white, inner sep=0pt, minimum size=0.25cm] (XK0L) at (-3.5, 3.09) {\small $x_{k^0}$};

\node[shape=circle, fill=black, inner sep=0pt, minimum size=0.25cm] (XL0+1) at (-5.167, 2.3) {};
\node[shape=circle, fill=black, inner sep=0pt, minimum size=0.25cm] (XL0+2) at (-4.33, 2.52) {};

\node[shape=circle, draw=black, inner sep=0pt, minimum size=0.55cm] (Y0) at (-4.5, 1.5) {\small $y^0$};

\draw[-] (Y0) to (XL0);
\draw[-] (Y0) to (XL0+1);
\draw[-] (Y0) to (XL0+2);
\draw[-] (Y0) to (XK0);

\node[shape=circle, draw=black, inner sep=0pt, minimum size=0.55cm] (Y1) at (-4.2, -1.5) {\small $y^1$};

\node[shape=circle, fill=black, inner sep=0pt, minimum size=0.25cm] (XL1) at (-5.167, -2.3) {};
\node[shape=circle, fill=black, inner sep=0pt, minimum size=0.25cm] (XK1) at (-4.33, -2.52) {};

\node[shape=circle, draw=white, inner sep=0pt, minimum size=0.25cm] (XK0L) at (-4.25, -2.9) {\small $x_{{\ell}^1}$};
\node[shape=circle, draw=white, inner sep=0pt, minimum size=0.25cm] (XK0L) at (-5.2, -2.7) {\small $x_{k^1}$};

\draw[-] (XL1) to (Y1) to (XK1);
\draw[-] (U) to (W);

\end{tikzpicture}\captionof{figure}{The ellipse enclosing the diagram represents the cycle $C$}\label{ElipseFigureSpM1}\end{center}

\begin{claim}\label{LoneVertSpanPathOddLenCl} 
Both of the following hold.
\begin{enumerate}[label=\arabic*)]
\itemsep-0.1em
\item $n\in\{2,3,4\}$. Furthermore, if $\textnormal{Sp}(C)=\varnothing$, then $n=2$; AND
\item If $y\in \textnormal{Sp}(C)$, then $G_y-y=x_1\cdots x_n$. In particular, $|\textnormal{Sp}(C)|\leq 1$.
\end{enumerate} \end{claim}

\begin{claimproof} If $n=2$, then it is is immediate from short-inseparability that $|\textnormal{Sp}(C)|\leq 1$ and so both 1) and 2) hold. Suppose now that Claim \ref{LoneVertSpanPathOddLenCl} does not hold. Thus, it follows from Claims \ref{AllVertSame3ListSubCL} and \ref{InterMResAtMostXSpanC} that $n=4$ and there is a $y\in\textnormal{Sp}(C)$ such that $G_y-y$ has even length. Again applying Claim \ref{AllVertSame3ListSubCL}, we let $1\leq k\leq 2$, where $G_y-y=x_kx_{k+1}x_{k+2}$. Note that $G_y-y$ is a proper subpath of of $G^{\textnormal{large}}_Q\cap C$. Let $H$ be a graph obtained from $G$ by contracting $x_k, x_{k+2}$ into $x_{k+1}$ and deleting the resulting parallel edges. Let $x^{\dagger}$ and $C^{\dagger}$ respectively be the new vertex and the facial cycle of $H$ obtained from $C$ by the contraction. Since $G$ is short-inseparable and $|V(G_y)|>3$, we have $|V(C)|\geq 5$, so $C^{\dagger}$ is indeed a cycle. In $H$, we have $N(y)\cap V(C^{\dagger})=\{x^{\dagger}\}$. By Claim \ref{AllVertSame3ListSubCL}, $L(x_k)=L(x_{k+2})$. Let $L^{\dagger}$ be a list-assignment for $H$, where $L^{\dagger}(x^{\dagger})=L(x_k)=L(x_{k+2})$, and otherwise $L^{\dagger}=L$. It is immediate that $\mathcal{K}^{\dagger}:=[\Sigma, H, C^{\dagger}, L^{\dagger}]$ is still a uniquely 4-determined collar. Since at most one vertex of $Q\cap C$ lies in $\{x_k, x_{k+2}\}$, we may regard $Q$ as a subgraph of $H$. Note that $Q$ is still a maximal $w$-enclosure in $\mathcal{K}^{\dagger}$. In particular, $H^{\textnormal{small}}_{Q^{\dagger}}=G^{\textnormal{small}}_Q$. 

\vspace*{-8mm}
\begin{addmargin}[2em]{0em}
\begin{subclaim}\label{HNotSSFSmGSub} $H$ is  not short-inseparable. \end{subclaim}

\begin{claimproof} Suppose $H$ is short-inseparable. Thus, $\mathcal{K}^{\dagger}, Q$, and $uw$ satisfy all the conditions in statement of Theorem \ref{StrengthendVerManColRes}. Since $|V(H)|<|V(G)|$, there is a $(Q, uw, \mathcal{K}^{\dagger})$-pair $(S, \phi)$. In particular, $\textnormal{dom}(\phi)\cap V(H^{\textnormal{large}}_Q\cap C^{\dagger})=V(H^{\textnormal{large}}_Q\cap C^{\dagger})\setminus\textnormal{Sh}^2(C^{\dagger}, \mathcal{K}^{\dagger})$. Note that $G$ contains no generalized chord $R$ of $C$ of length at most two such that $x_{k+1}\in V(G^{\textnormal{small}}_R\setminus R)$, except for $R=x_kyx_{k+2}$. Thus, $x^{\dagger}\not\in\textnormal{Sh}^2(C^{\dagger}, \mathcal{K}^{\dagger})$, so $x^{\dagger}\in\textnormal{dom}(\phi)$. Let $a:=\phi(x^{\dagger})$. We now define a partial $L$-coloring $\psi$ of $V(G)$, where $\textnormal{dom}(\psi)=(\textnormal{dom}(\phi)\setminus\{x^{\dagger}\})\cup\{x_k, x_{k+2}\}$, and, in particular, $\psi(x_k)=\psi(x_{k+2})=a$, and $\psi(v)=\phi(v)$ for each $v\in\textnormal{dom}(\phi)\setminus\{x^{\dagger}\}$. This is permissible as the vertices of $G^{\textnormal{large}}_Q\cap C$ have the same 3-list. As $x_kx_{k+2}\not\in E(G)$, $\psi$ is indeed a proper $L$-coloring of its domain. Let $T:=(S\setminus\{x^{\dagger}\})\cup\{x_k, x_{k+1}, x_{k+2}\}$. We now produce a contradiction to our assumption on $Q$ by showing that $(T, \psi)$ is a $(Q, uw, \mathcal{K})$-pair. It is immediate that the inertness condition is satisfied, since $x_{k+1}$ has only three neighbors are two of them use the same color. Likewise, $V(C^{\dagger})\subseteq S$, so $V(C)\subseteq T$. For each $z\in V(G\setminus T)$, we have $L_{\psi}(z)=L_{\phi}(z)$, and $V(G^{\textnormal{large}}_Q)\cap\textnormal{Sh}^2(C, \mathcal{K})=(V(H)\cap\textnormal{Sh}^2(C^{\dagger}, \mathcal{K}^{\dagger}))\cup\{x_{k+1}\}$, so we are done. \end{claimproof}\end{addmargin}

Note that $\textnormal{ew}(H)>4$, or else $H$ contains a noncontractible cycle of length at most four which has a vertex of $C^{\dagger}$, which is false, as $\mathcal{K}^{\dagger}$ is uniquely 4-determined.  Thus, $H$ contains a separating cycle of length at most four, and this cycle contains $x^{\dagger}$. Let $\mathcal{R}$ denote the set of generalized chords $R$ of $C$ such that $3\leq |E(R)|\leq 4$ and $R$ has endpoints $x_k, x_{k+2}$. Note that $\mathcal{R}\neq\varnothing$ and, for each $R\in\mathcal{R}$, we have $G^{\textnormal{small}}_R\subseteq G^{\textnormal{large}}_Q$. For any two $R, R'\in\mathcal{R}$, we have either $ G^{\textnormal{small}}_R\subseteq G^{\textnormal{small}}_{R'}$ or vice-versa, so there is a unique maximal $R^m\in\mathcal{R}$. There is a separating cycle $D\subseteq H$ with $3\leq |V(D)|\leq 4$, where $D$ is obtained from $R^m$ by the contraction of $x_kx_{k+1}x_{k+2}$. Let $U^m$ be the unique open component of $\Sigma\setminus D$ containing $y$ and let $H^*$ be the embedding obtained from $H$ by deleting from $H$ all of the vertices of $V(H)\cap U^m$. Let $\mathcal{K}^*:=[\Sigma, H^*, C^{\dagger}, L^{\dagger}]$. Now, $U^m$ is a disc and every vertex in $U^m$ has a list of size at least five, so $\mathcal{K}^*$ is uniquely 4-determined. Furthermore, by the maximality of $R^m$, $H^*$ is short-inseparable. 

It is straightforward to check that all the other conditions in Theorem \ref{StrengthendVerManColRes} are still satisfied by $Q, uw$, and $\mathcal{K}^*$. In particular, $Q$ is a maximal $w$-enclosure in $\mathcal{K}^*$. As $|V(H^*)|<|V(G)|$, there is a $(Q, uw, \mathcal{K}^*)$-pair $(S, \phi)$. As $x^{\dagger}$ is a vertex of $(H^*)^{\textnormal{large}}_Q\setminus\textnormal{Sh}^2(C^{\dagger}, \mathcal{K}^*)$, we have $x^{\dagger}\in\textnormal{dom}(\phi)$. Let $a:=\phi(x^{\dagger})$. Now let $T:=(S\setminus\{x^{\dagger}\})\cup\{x_k, x_{k+1}, x_{k+2}\}\cup V(G^{\textnormal{small}}_{R^m}\setminus R^m)$ and define a partial $L$-coloring $\psi$ of $V(G)$, where $\textnormal{dom}(\psi)=(\textnormal{dom}(\phi)\setminus\{x^{\dagger}\})\cup\{x_k, x_{k+2}\}$, and, in particular, $\psi(x_k)=\psi(x_{k+2})=a$, and $\psi(v)=\phi(v)$ for each $v\in\textnormal{dom}(\phi)\setminus\{x^{\dagger}\}$. We now show that $(T, \psi)$ is a $(Q, uw, \mathcal{K})$-pair, which contradicts our assumption on $Q$. The only thing we need to check that is not analogous to the construction in Subclaim \ref{HNotSSFSmGSub} is the inertness condition. Note that $\{x_k, x_{k+2}\}\subseteq\textnormal{dom}(\psi)\cap V(G^{\textnormal{small}}_{R^m})\subseteq V(R^m)$, so it suffices to show that, for any $L$-coloring $\tau$ of $V(R^m)$, if $\tau(x_k)=\tau(x_{k+2})$, then $\tau$ extends to an $L$-coloring of $V(G^{\textnormal{small}}_{R^m})$. Suppose we have such a $\tau$. Note that $|L_{\tau}(x_{k+1})|\geq 2$ and $|L_{\tau}(y)|\geq (|V(R^m)|+1)-|N(y)\cap\textnormal{dom}(\tau)|$, so, by Theorem \ref{BohmePaper5CycleCorList}, $\tau$ extends to $L$-color $V(G^{\textnormal{small}}_{R^m})$. This proves Claim \ref{LoneVertSpanPathOddLenCl}. \end{claimproof}

Note that $Q$ is a proper generalized chord of $C$, or else, by Claim \ref{LoneVertSpanPathOddLenCl}, $G$ has parallel edges. Let $x_1v^0, x_nv^1$ denote the terminal edges of $Q$.  Furthermore, we let $C_Q$ be the facial cycle $(G^{\textnormal{small}}_Q\cap C)+Q$ of $G^{\textnormal{small}}_Q$. In the remainder of the proof of Theorem \ref{StrengthendVerManColRes}, whenever we use the notation from Definition \ref{EndNotationColor}, it is with the understanding that $G^{\textnormal{small}}_Q$ is being regarded as a planar embedding outer cycle $C_Q$.  Given an $L$-coloring $\phi$ of $\{x_1, x_n\}$, we say that $\phi$ is \emph{parity-respecting} if either $n=3$ and $\phi(x_1)=\phi(x_n)$, or $n\in\{2,4\}$ and $\phi(x_1)\neq \phi(x_n)$.

\begin{claim}\label{SigmaSameColX1XNCL} 
Both of the following hold.
\begin{enumerate}[label=\arabic*)]
\itemsep-0.1em
\item All the vertices of $\{x_1, \cdots, x_n\}$ have the same 3-list. Furthermore, for any parity-respecting $\phi$, $\{x_1, \cdots, x_n\}$ is inert and $\phi$ does not extend to a $\phi'\in\textnormal{Crown}(C_Q, G^{\textnormal{small}}_Q)$ with $d(w, \textnormal{dom}(\phi'))\leq 1$; AND
\item There is no chord of $Q$ in $G$, except possibly $x_1x_2$ if $n=2$. In particular, there is no chord of $Q$ in $G^{\textnormal{small}}_Q$.
\end{enumerate}
 \end{claim}

\begin{claimproof}  Suppose not all the vertices of $\{x_1, \cdots, x_n\}$ have the same 3-list. By Claim \ref{AllVertSame3ListSubCL}, $n=2$. Without loss of generality, there is a $c\in L(x_1)$ with $|L(x_2)\setminus\{c\}|\geq 3$. Now, there is a $\pi\in\textnormal{Crown}(Q, G^{\textnormal{small}}_Q)$ with $d(w, \textnormal{dom}(\pi))\leq 1$. This follows from Theorem \ref{MainHolepunchPaperResulThm} if $Q$ is non-degenerate, and otherwise it follows from Theorem \ref{ModifiedRes5ChordCaseDegen}. Note that $\pi$ is a proper $L$-coloring of its domain in $G$, even though $G^{\textnormal{small}}_Q$ is not induced in $G$, and, by B) of Theorem \ref{GenThmTargetConnector}, $\pi$ extends to a $(Q, uw)$-target, contradicting Claim \ref{IfTargetFailsInertFails}. Thus, all the vertices of $\{x_1, \cdots, x_n\}$ have the same 3-list. Now let $\phi$ be as in the statement of 1). Since $x_1, \cdots, x_n$ have the same 3-list, the set $\{x_1, \cdots, x_n\}$ is $(L, \phi)$-inert in $G$. Suppose $\phi$ extends to a $\phi'$ as in 1). Again by B) of Theorem \ref{GenThmTargetConnector}, $\phi'$ extends to a $(Q, uw)$-target, and we again contradict Claim \ref{IfTargetFailsInertFails}. This proves 1). Now we prove 2). Suppose there is a such a chord $e$ of $Q$. Now, $e\in E(G^{\textnormal{small}}_Q)$ by Claim \ref{EitherInducedOrAtMostOnePD}. If $w$ is non-degenerate, then $e$ does not have $w$ as an endpoint, and $N(w)\cap D_1(C)\subseteq V(Q\setminus C)$, which is false by definition. Thus, $w$ is degenerate and every vertex of $Q\setminus C$ is adjacent to a vertex of $G^{\textnormal{small}}_Q\cap C$, so there is a chord of $Q$ in $G^{\textnormal{small}}_Q$, each endpoint of which is incident to a terminal edge of $Q$. We then have $n=4$. By Claim \ref{LoneVertSpanPathOddLenCl}, $\textnormal{Sp}(C)|=1$, and the lone vertex of $\textnormal{Sp}(C)$ is adjacent to each of $x_1,x_n$, contradicting short-inseparability. \end{claimproof}

We now let $T_Q$ be the set of vertices $z$ in $G^{\textnormal{large}}_Q\setminus (Q+uw)$ such that $z$ has at least three neighbors in $Q+uw$. We now use Theorem \ref{RainbowNonEqualEndpointColorThm} to rule out the possibility that $w$ is degenerate.

\begin{claim}\label{CompressWDegenCase1} $w$ is non-degenerate. \end{claim}

\begin{claimproof} Suppose $w$ is degenerate. Thus, we let $Q:=x_1v^0y^0y^1v^1x_n$, where $N(w)\cap D_1(C)$ is a nonempty subset of $\{y^0, y^1\}$. Let $\mathcal{G}$ denote the rainbow $(G^{\textnormal{small}}_Q, C_Q, Q, L)$.

 \vspace*{-8mm}
\begin{addmargin}[2em]{0em}
\begin{subclaim}\label{ExistGObstruc} There is a $\mathcal{G}$-obstruction $x^*\in V(C_Q)\setminus\{x_1, x_n\}$. Furthermore, $x^*$ is adjacent to all four vertices of $\mathring{Q}$ and, for any $z\in V(G^{\textnormal{large}}_Q)\setminus\textnormal{Sp}(C)$ with a neighbor in $Q$, the graph $G[N(z)\cap V(Q)]$ is a path of length at most one \end{subclaim}

\begin{claimproof} Suppose there is no $\mathcal{G}$-obstruction. As $L(x_1)=L(x_n)$, it follows from Theorem \ref{RainbowNonEqualEndpointColorThm} that there is a $\phi\in\textnormal{Crown}(Q, G^{\textnormal{small}}_Q)$ where $d(\textnormal{dom}(\phi), w)=1$ and the restriction of $\phi$ to $\{x_1, x_n\}$ is parity-respecting, contradicting Claim \ref{SigmaSameColX1XNCL}. Thus, there is such an $x^*$, and, by 2) of Claim \ref{SigmaSameColX1XNCL}, $x^*\not\in\{x_1, x_n\}$. If $x^*$ is adjacent to all four vertices of $\mathring{Q}$, then, since $G$ is $K_{2,3}$-free, it follows that that any $z\in V(G^{\textnormal{large}}_Q)\setminus\textnormal{Sp}(C)$ with a neighbor in $Q$ satisfies the specified conditions in the statement of Subclaim \ref{ExistGObstruc}, so it suffices to show that $x^*$ is adjacent to all four vertices of $\mathring{Q}$. Suppose not. Suppose without loss of generality that $x_n, v^0, y^0\in N(x^*)$ and let $H:=G^{\textnormal{small}}_{x^1v^0x^*}$. Note that $v^1\not\in N(x^*)$, and $y^1$ is adjacent to at least one of $x^*, x_n$.  Since $G$ is short-inseparable, we have $V(G^{\textnormal{small}}_Q)=V(H\cup Q)$.

\begin{center}\begin{tikzpicture}
\begin{scope}[xscale=-1]
 \draw (0,3) arc(90: 270:8cm and 3cm);
\end{scope}

\node[shape=circle, fill=black, inner sep=0pt, minimum size=0.25cm] (X1) at (6, 1.97) {};

\node[shape=circle, draw=black, inner sep=0pt, minimum size=0.55cm] (Y*) at (6, 1.1) {\small $v^0$};

\node[shape=circle, draw=black, inner sep=0pt, minimum size=0.55cm] (Zmid) at (6, 0.4) {\small $y^0$};

\node[shape=circle, draw=black, inner sep=0pt, minimum size=0.55cm] (Zmid+) at (6, -0.3) {\small $y^1$};

\node[shape=circle, draw=black, inner sep=0pt, minimum size=0.55cm] (Y') at (6, -0.985) {\small $v^1$};

\node[shape=circle, fill=black, inner sep=0pt, minimum size=0.25cm] (XN) at (6, -1.97) {};

\draw[-, line width=1.8pt, color=red] (X1) to (Y*) to (Zmid) to (Zmid+) to (Y') to (XN);

\node[shape=circle, draw=white, inner sep=0pt, minimum size=0.25cm] (X1L) at (6, 2.3) {\small $x_1$};
\node[shape=circle, draw=white, inner sep=0pt, minimum size=0.25cm] (XNL) at (6, -2.3) {\small $x_3$};
\node[shape=circle, draw=black, inner sep=0pt, minimum size=0.55cm] (Y0) at (0, 0) {\small $y$};

\draw[-] (XN) to (Y0) to (X1);

\node[shape=circle, fill=black, inner sep=0pt, minimum size=0.25cm] (Xnew) at (-1.73, 0) {};
\node[shape=circle, draw=white, inner sep=0pt, minimum size=0.55cm] (XNewlab) at (-2.2, 0) {\small $x_2$};

\draw[-] (0, 3) to [out=190, in=170] (0, -3);

\node[shape=circle, fill=black, inner sep=0pt, minimum size=0.25cm] (X*) at (7, -1.45) {};
\node[shape=circle, draw=white, inner sep=0pt, minimum size=0.55cm] (X*lab) at (7.5, -1.45) {\small $x^*$};

\draw[-] (X*) to [out=100, in=-30] (Y*);
\draw[-] (X*) to [out=100, in=-30] (Zmid);
\draw[-] (Y0) to (Xnew);

\end{tikzpicture}
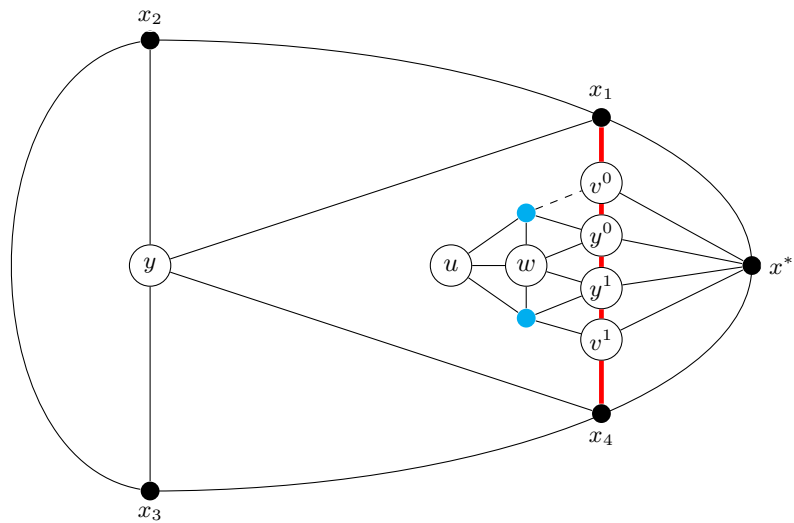
\captionof{figure}{The ellipse enclosing the diagram represents the cycle $C$}\label{FinEnterCase2}\end{center}

Since $v^1x^*\not\in E(G)$, it is straightforward to check by applying Theorem \ref{SumTo4For2PathColorEnds} applied to $H$ that there is an element of $\textnormal{Crown}(Q, G^{\textnormal{small}}_Q)$ which has domain $V(Q+x^*)\setminus\{v^0, v^1\}$ and uses different colors on $x_1, x_n$. Thus, by Claim \ref{SigmaSameColX1XNCL}, $n=3$, so $\textnormal{Sp}(C)=\{y\}$ for some vertex $y$. This is illustrated in Figure \ref{FinEnterCase2}. As $G$ is short-inseperable, $yv^0\not\in E(G)$ and the arc in Figure \ref{FinEnterCase2} from $x_1$ to $x^*$ is actually a path of length at least two. Now fix an arbitrary $L$-coloring $\phi$ of $\{x_1, x_n\}$ with $\phi(x_1)=\phi(x_n)$ and let $k\in\{0,1\}$, where $k=1$ if $y^1\in N(w)$ and otherwise $k=0$. Now we just leave $v^1, y^{1-k}$ uncolored. As any $L$-coloring of $\{x_1, x^*\}$ extends to $L$-color $H$, it is straightforward to check that $\phi$ extends to an $L$-coloring $\tau$ of $(V(Q+x^*)\cup\{w, u\})\setminus\{v^1, y^{1-k}\}$ such that each $T_Q\cup\{v^1, y^{1-k}\}$ has an $L_{\tau}$-list of size at least three and $V(G^{\textnormal{small}}_Q)\setminus\{v^1, y^{1-k}\}$ is $(L, \tau)$-inert in $G$. Thus, $(V(G_Q)\setminus\{v^1, y^{1-k}\}, \tau)$ is a $(Q, uw)$-pair, contradicting our assumption on $Q$. This proves Subclaim \ref{ExistGObstruc}.  \end{claimproof}\end{addmargin}

Let $x^*$ be as in Subclaim \ref{ExistGObstruc}. We now have the following.

 \vspace*{-8mm}
\begin{addmargin}[2em]{0em}
\begin{subclaim}\label{LeaveOneVertUnColCL} Let $k\in\{0,1\}$ and let $\tau$ be an $L$-coloring of $V(Q+x^*)\setminus\{v^k\}$, where each vertex of $\{v^k\}\cup\textnormal{Sp}(C)$ has an $L_{\tau}$-list of size at least three and the restriction of $\tau$ to $\{x_1, x_n\}$ is parity-respecting. Then $V(G^{\textnormal{small}}_Q-v^k)$ is not $(L, \tau)$-inert in $G$. \end{subclaim}

\begin{claimproof} Say $k=0$ for the sake of definiteness. Suppose $V(G^{\textnormal{small}}_Q-v^0)$ is $(L, \tau)$-inert in $G$. Let $S:=V(C\cup (G^{\textnormal{small}}_Q+uw))\setminus\{v^0\}$. Note that $S$ is topologically reachable from $C$ and, for any extension of $\tau$ to an $L$-coloring $\tau'$ of $\textnormal{dom}(\tau)\cup\{u, w\}$, the set $S$ is $(L, \tau')$-inert in $G$ by our choice of coloring of $\{x_1, x_n\}$. Thus, there is a vertex of $T_Q\setminus\textnormal{Sp}(C)$ with an $L_{\tau'}$ list of size $<3$, or else $(S, \tau')$ is a $(Q, uw)$-pair, contradicting our assumption on $Q$. Subclaim \ref{ExistGObstruc} implies $\tau$ extends to an $L$-coloring of $\textnormal{dom}(\tau)\cup\{w, u\}$ which violates the constraint above, unless the structure in Figure \ref{FinCaseChordWrap221} occurs (or the analogous structure with $n=2,3$ instead of $n=4$), where the elements of $T_Q\setminus\textnormal{Sp}(C)$ are indicated in blue, the dashed edge is not necessarily present, and the unique vertex of $T_Q$ adjacent to $v^1$ is adjacent to a subpath of $Q+uw$ which has length three and endpoints $u, v^1$. 

\begin{center}\begin{tikzpicture}
\begin{scope}[xscale=-1]
 \draw (0,3) arc(90: 270:8cm and 3cm);
\end{scope}

\node[shape=circle, fill=black, inner sep=0pt, minimum size=0.25cm] (X1) at (6, 1.97) {};

\node[shape=circle, draw=black, inner sep=0pt, minimum size=0.55cm] (Y*) at (6, 1.1) {\small $v^0$};

\node[shape=circle, draw=black, inner sep=0pt, minimum size=0.55cm] (Zmid) at (6, 0.4) {\small $y^0$};

\node[shape=circle, draw=black, inner sep=0pt, minimum size=0.55cm] (Zmid+) at (6, -0.3) {\small $y^1$};

\node[shape=circle, draw=black, inner sep=0pt, minimum size=0.55cm] (Y') at (6, -0.985) {\small $v^1$};

\node[shape=circle, fill=black, inner sep=0pt, minimum size=0.25cm] (XN) at (6, -1.97) {};

\draw[-, line width=1.8pt, color=red] (X1) to (Y*) to (Zmid) to (Zmid+) to (Y') to (XN);

\node[shape=circle, draw=white, inner sep=0pt, minimum size=0.25cm] (X1L) at (6, 2.3) {\small $x_1$};
\node[shape=circle, draw=white, inner sep=0pt, minimum size=0.25cm] (XNL) at (6, -2.3) {\small $x_4$};
\node[shape=circle, draw=black, inner sep=0pt, minimum size=0.55cm] (Y0) at (0, 0) {\small $y$};

\draw[-] (XN) to (Y0) to (X1);

\node[shape=circle, fill=black, inner sep=0pt, minimum size=0.25cm] (X2) at (0, 3) {};

\node[shape=circle, draw=white, inner sep=0pt, minimum size=0.25cm] (X2L) at (0, 3.3) {\small $x_2$};
\node[shape=circle, draw=white, inner sep=0pt, minimum size=0.25cm] (X2L) at (0, -3.3) {\small $x_3$};
\node[shape=circle, fill=black, inner sep=0pt, minimum size=0.25cm] (X3) at (0, -3) {};

\draw[-] (X2) to (Y0) to (X3);

\draw[-] (X2) to [out=190, in=170] (X3);

\node[shape=circle, draw=black, inner sep=0pt, minimum size=0.55cm] (W) at (5, 0) {$w$};
\node[shape=circle, draw=black, inner sep=0pt, minimum size=0.55cm] (U) at (4, 0) {$u$};

\draw[-] (U) to (W) to (Zmid);
\draw[-] (W) to (Zmid+);

\node[shape=circle, fill=cyan, inner sep=0pt, minimum size=0.25cm] (ZT) at (5, -0.7) {};
\node[shape=circle, fill=cyan, inner sep=0pt, minimum size=0.25cm] (ZT+) at (5, 0.7) {};

\draw[-] (ZT) to (U);
\draw[-] (ZT) to (W);
\draw[-] (ZT) to (Zmid+);
\draw[-] (ZT) to (Y');

\draw[-] (ZT+) to (U);
\draw[-] (ZT+) to (W);
\draw[-] (ZT+) to (Zmid);
\draw[-, dashed] (ZT+) to (Y*);

\node[shape=circle, fill=black, inner sep=0pt, minimum size=0.25cm] (X*) at (8, 0) {};
\node[shape=circle, draw=white, inner sep=0pt, minimum size=0.55cm] (X*lab) at (8.4, 0) {\small $x^*$};

\draw[-] (X*) to (Y*);
\draw[-] (X*) to (Zmid);
\draw[-] (X*) to (Zmid+);
\draw[-] (X*) to (Y');

\end{tikzpicture}\captionof{figure}{The ellipse enclosing the diagram represents the cycle $C$}\label{FinCaseChordWrap221}\end{center}

In particular, $w$ is adjacent to both endpoints of the middle edge of $Q$ and $w$ has precisely five neighbors. Now, there is an $L$-coloring $\pi$ of $\left(\textnormal{dom}(\tau)\cup\{u\}\right)\setminus\{y^1\}$ with $|L_{\pi}(w)|\geq 4$. Note that $|L_{\pi}(y^1)|\geq 2$ and, since $N(y^1)\subseteq\{w\}\cup\textnormal{dom}(\tau)$, it follows that $S$ is $(L, \pi)$-inert in $G$. As $w, y^1, v^0$ remain uncolored, each vertex of $T_Q$ has at most two neighbors in $\textnormal{dom}(\tau)$, so $(S, \pi)$ is a $(Q, uw)$-pair, contradicting our assumption $Q$. \end{claimproof}\end{addmargin}

Now, given an $e=xv\in\{x_1v^0, x_nv^1\}$, where $x\in\{x_1, x_n\}$, we say that $e$ is \emph{good} if both $xx^*\not\in E(G)$ and $v$ is adjacent to no vertex of $\textnormal{Sp}(C)$. Since $G$ is short-inseparable, at least one terminal edge of $Q$ is good, say $e=x_1v^0$ for the sake of definiteness. Now fix an arbitrary parity-preserving $L$-coloring $\phi$ of $\{x_1, x_n\}$. By applying Theorem \ref{SumTo4For2PathColorEnds} to $G^{\textnormal{small}}_{x_nv^1x^*}$ and using both the fact that $x^*x_1\not\in E(G)$ and the fact that any $L$-coloring of $\{x_1, x^*\}$ extends to $L$-color $G^{\textnormal{small}}_{x^*v^0x_1}$, we get that $\phi$ extends to an $L$-coloring $\tau$ of $V(Q+x^*)\setminus\{v^1\}$ such that $|L_{\tau}(v^1)|\geq 3$ and $V(G^{\textnormal{small}}_Q-v^1)$ is $(L, \tau)$-inert in $G$. Since $v^0$ has no neighbors in $\textnormal{Sp}(C)$, we contradict Subclaim \ref{LeaveOneVertUnColCL}.  This proves Claim \ref{CompressWDegenCase1}. \end{claimproof}

Since $w$ is non-degenerate, $Q$ is a proper 4-chord of $C$ with midpoint $w$, so $Q=x_1v^0wv^1x_n$. Now, it follows from the maximality of $Q$, together with short-inseparability, that, for any $z\in T_Q\setminus\textnormal{Sp}(C)$, we have $z\in D_2(C)$ and $G[N(z)\cap V(Q+uw)]$ is a path of length two with $uw$ as a terminal edge. We now rule out the possibility that $n=2$.

\begin{claim}\label{SeconInter|E|OddLen} $|E(G^{\textnormal{small}}_Q\cap C)|\neq 2$. \end{claim}

\begin{claimproof} Suppose  $|E(G^{\textnormal{small}}_Q\cap C)|=2$. Thus, $G^{\textnormal{large}}_Q\cap C=x_1x_2x_3$. By Claim \ref{InterMResAtMostXSpanC}, there is a lone vertex $y$ such that $\textnormal{Sp}(C)=\{y\}$ and $G_y-y=x_1x_2x_3$. By Claim \ref{EitherInducedOrAtMostOnePD}, $G^{\textnormal{small}}_Q$ is an induced subgraph of $G$. Note that $G^{\textnormal{small}}_Q\cap C$ is a path of length at least three, or else there is a cycle of length at most four separating $x_2$ from $u$. In particular, $|E(C)|\geq 5$. Since $G$ is short-inseparable, we note that $y$ is adjacent to at most one of $v^0, v^1$, so we suppose without loss of generality that $v^0\not\in N(y)$.  We now set $A:=(V(G^{\textnormal{small}}_Q\cup C)\cup\{u, y\})\setminus\{v^0, v^1\}$. 

\vspace*{-8mm}
\begin{addmargin}[2em]{0em}
\begin{subclaim}\label{TqSet1} $|T_Q\setminus\{y\}|=2$. \end{subclaim}

\begin{claimproof} By Claim \ref{AllVertSame3ListSubCL}, $L(x_1)=L(x_3)$. As $x_1x_3\not\in E(G)$ and $G^{\textnormal{small}}_Q$ is induced, it follows from Lemma \ref{PartialPathColoringExtCL0} that there is an $L$-coloring $\psi$ of $V(G^{\textnormal{small}}_Q)$ with $\psi(x_1)=\psi(x_3)$. Possibly $v^1\in N(y)$, but, in any case, we have $|L_{\psi}(y)|\geq 3$, and $\{w_2\}$ is $(L, \psi)$-inert in $G$. Now suppose $|T_Q\setminus\{y\}|\neq 2$. Now, $|L_{\psi}(u)|\geq 4$, and furthermore, $|T_Q\setminus\{y\}|<2$ and $\psi$ extends to an $L$-coloring $\psi^*$ of $\textnormal{dom}(\psi)\cup\{u\}$ such that each vertex of $T_Q\setminus\{y\}$ has an $L_{\psi^*}$-list of size at least three. As $uy\not\in E(G)$, each vertex of $T_Q$ has an $L_{\psi^*}$-list of size at least three. Thus, $(A\cup\{v^0, v^1\}, \psi^*)$ is a $(Q, uw)$-pair, contradicting our assumption. \end{claimproof}\end{addmargin}

It follows from Subclaim \ref{TqSet1} that there exist vertices $z^0, z^1\in V(G^{\textnormal{large}}_Q)\cap D_2(C)$, where, for each $k=0,1$, we have $N(z_k)\cap V(Q+uw)=\{u, w, v^k\}$. In particular, $N(w)\subseteq V(G^{\textnormal{small}}_Q)\cup\{u, z^0, z^1\}$. By Theorem \ref{MainHolepunchPaperResulThm}, there is a $\sigma\in\textnormal{Crown}(Q, G)$. Now, $1\leq |L_{\sigma}(x_2)|\leq 2$, so there is a $c\in L_{\sigma}(y)$ with $|L_{\sigma}(x_2)\setminus\{c\}|\geq 1$.

\vspace*{-8mm}
\begin{addmargin}[2em]{0em}
\begin{subclaim}\label{YAdV1R} $N(y)\cap\{v^0, v^1\}=\{v^1\}$ and $N(y)\cap\{z^0, z^1\}=\{z^1\}$. \end{subclaim}

\begin{claimproof} Suppose $N(y)\cap\{v^0, v^1\}\neq\{v^1\}$. Thus, $v^0, v^1\not\in N(y)$. Now, there is an extension of $\sigma$ to an $L$-coloring $\tau$ of $\textnormal{dom}(\sigma)\cup\{y, u\}$ such that $\tau(y)=c$. Each of $v^0, v^1, z^0, v^1$ has an $L_{\tau}$-list of size at least three. By our choice of color for $y$, $A$ is $(L, \tau)$-inert in $G$. Thus, $(A, \tau)$ is a $(Q, uw)$-pair, contradicting our assumption, so $N(y)\cap\{v^0, v^1\}=\{v^1\}$. Note that $z^0\not\in N(y)$ by short-inseparability, so we just need to show that $z^1\in N(y)$. Suppose not. We extend $\sigma$ to an $L$-coloring $\pi$ of $\textnormal{dom}(\sigma)\cup\{y, v^1, u\}$ with $\pi(y)=c$ and $|L_{\pi}(z^1)|\geq 3$. Now, $(A\cup\{v^1\}, \pi)$ is a $(Q, uw)$-pair, contradicting our assumption.  \end{claimproof}\end{addmargin}

\begin{center}\begin{tikzpicture}
\begin{scope}[xscale=-1]
 \draw (0,3) arc(90: 270:8cm and 3cm);
\end{scope}
\node[shape=circle, draw=black, inner sep=0pt, minimum size=0.55cm] (W) at (5, 0) {\small $w$};
\node[shape=circle, draw=black, inner sep=0pt, minimum size=0.55cm] (U) at (3.5, 0) {\small $u$};
\node[shape=circle, draw=white, inner sep=0pt, minimum size=0.55cm] (GSm) at (6.5, 0) {$G^{\textnormal{small}}_Q$};
\node[shape=circle, fill=black, inner sep=0pt, minimum size=0.25cm] (X1) at (6, 1.97) {};
\node[shape=circle, fill=black, inner sep=0pt, minimum size=0.25cm] (XN) at (6, -1.97) {};
\node[shape=circle, draw=black, inner sep=0pt, minimum size=0.55cm] (Z) at (5.67, 1.31) {\small $v^0$};
\node[shape=circle, draw=black, inner sep=0pt, minimum size=0.55cm] (Z') at (5.67, -1.31) {\small $v^1$};
\draw[-, line width=1.8pt, color=red] (X1) to (Z) to (W) to (Z') to (XN);
\node[shape=circle, draw=white, inner sep=0pt, minimum size=0.25cm] (X1L) at (6, 2.3) {\small $x_1$};
\node[shape=circle, draw=white, inner sep=0pt, minimum size=0.25cm] (XNL) at (6, -2.3) {\small $x_3$};
\node[shape=circle, draw=black, inner sep=0pt, minimum size=0.55cm] (Y0) at (0,0) {\small $y$};

\draw[-] (XN) to (Y0) to (X1);
\draw[-] (U) to (W);

\node[shape=circle, fill=black, inner sep=0pt, minimum size=0.25cm] (Xmid) at (-2, 0) {};
\node[shape=circle, draw=white, inner sep=0pt, minimum size=0.25cm] (X1L) at (-2.5, 0) {\small $x_2$};

\draw[-] (Y0) to (Xmid);
\draw[-] (Xmid) to [out=65, in=180]  (0,3);
\draw[-] (Xmid) to [out=-65, in=180]  (0,-3);

\node[shape=circle, draw=black, inner sep=0pt, minimum size=0.55cm] (Znew0) at (4.45, 0.6) {\small $z^0$};
\node[shape=circle, draw=black, inner sep=0pt, minimum size=0.55cm] (Znew1) at (4.45, -0.6) {\small $z^1$};

\draw[-] (Y0) to (Z');
\draw[-] (Y0) to (Znew1);
\draw[-] (U) to (Znew1) to (Z');
\draw[-] (U) to (Znew0) to (Z);
\draw[-] (Znew0) to (W) to (Znew1);

\end{tikzpicture}
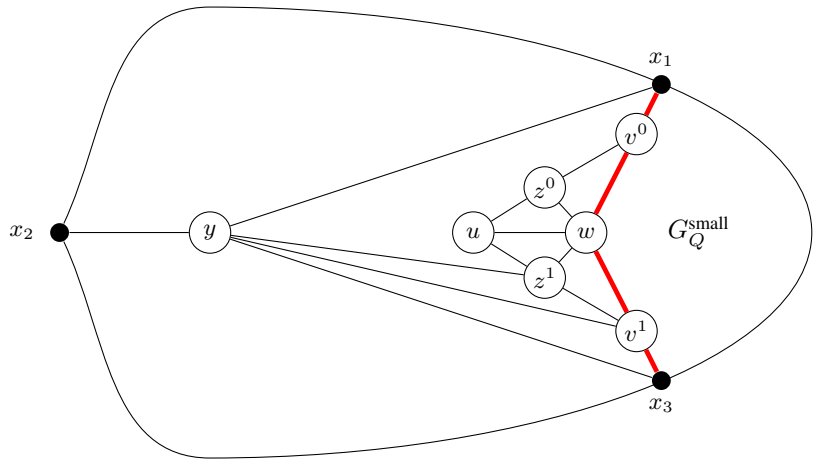
\captionof{figure}{The outer closed curve represents the cycle $C$}\label{CaseN=3PathLen2Fg}\end{center}

Since $y$ is adjacent to $z^1$, we have $N(v^1)\subseteq V(G^{\textnormal{small}}_Q)\cup\{z^1, y\}$ and we have the structure indicated in Figure \ref{CaseN=3PathLen2Fg}. Now we construct a $(Q, uw)$-pair by leaving $v^0, v^1$ uncolored but including $v^1$ in our deletion set. Note that $\sigma$ extends to an $L$-coloring $\tau$ of $\textnormal{dom}(\sigma)\cup\{u, y\}$ with $|L_{\tau}(z^1)|\geq 3$ and $\tau(y)=c$.  Now, $|L_{\tau}(v^1)|\geq 2$ and all but one neighbor of $v^1$ outside of $G^{\textnormal{small}}_Q$ is colored. This fact, together with our choice of $c$, implies that $A\cup\{v^1\}$ is $(L, \tau)$-inert in $G$, so $(A\cup\{v^1\}, \tau)$ is a $(Q, uw)$-pair, contradicting our assumption. This proves Claim \ref{SeconInter|E|OddLen}. \end{claimproof}

\begin{claim}\label{NonDegenCaseNoHalfTargetCL5M} 
\textcolor{white}{aaaaaaaaaaaaaaaaa}
\begin{enumerate}[label=\arabic*)]
\itemsep-0.1em
\item Let $v^*\in\{v^0, v^1\}$ and $\tau$ be a partial $L$-coloring of $V(G^{\textnormal{small}}_Q-v^*)$ whose domain contains $V(Q-v^*)$, where $\tau(x_1)\neq\tau(x_n)$ and each vertex of $\{v^*\}\cup\textnormal{Sp}(C)$ has an $L_{\tau}$-list of size at least three. Then $V(G^{\textnormal{small}}_Q-v^*)$ is not $(L, \tau)$-inert in $G$; AND
\item Let $x\tilde{v}$ be a terminal edge of $Q$, where $x\in\{x_1, x_n\}$ and $\tilde{v}\in\{v^0, v^1\}$, and there is a chord of $C_Q$ incident to $\tilde{v}$. Then $\textnormal{Sp}(C)|=1$ and the lone vertex of $\textnormal{Sp}(C)$ is adjacent to $\tilde{v}$.  
\end{enumerate} \end{claim}

\begin{claimproof} Suppose $V(G^{\textnormal{small}}_Q-v^*)$ is $(L, \tau)$-inert in $G$. As $w$ is non-degenerate and one of $v^0, v^1$ remains uncolored, it follows that $\tau$ extends to an $L$-coloring $\tau^*$ of $\textnormal{dom}(\tau)\cup\{u\}$ in which each vertex of $\{v^*\}\cup T_Q$ has an $L_{\tau}$-list of size at least three. Let $S:=V(G^{\textnormal{small}}_Q-v^*)\cup\{u\}\cup\{x_1, \cdots, x_n\}$. Since $\tau(x_1)\neq\tau(x_n)$, $S$ is $(L, \tau^*)$-inert in $G$ and $(S, \tau^*)$ is a $(Q, uw)$-pair, contradicting our assumption. This proves 1). Now we prove 2). Without loss of generality, let $\tilde{v}=v^1$. Suppose 2) does not hold. Thus, there is a chord $e$ of $C_Q$ incident to $v^1$, and $v^1$ has no neighbor in $\textnormal{Sp}(C)$. Let $x^*$ be the other endpoint of $e$. By 2) of Claim \ref{SigmaSameColX1XNCL}, $x^*$ is an internal vertex of the path $G^{\textnormal{small}}_Q\cap C$. Let $K=G^{\textnormal{small}}_{x^*v^1x_n}$. The case where $n=4$ (and thus $\textnormal{Sp}(C)\neq\varnothing$) is illustrated in Figure \ref{ElipseFinCaseChordWrap}.

Let $Q^1:=(Q-x_n)+v^1x^*$ and $C_{Q^1}$ be the facial cycle $(G^{\textnormal{small}}_{Q^1}\cap C)+Q^1$ of $G^{\textnormal{small}}_{Q^1}$. Note that $Q^1$ is also a $w$-enclosure, albeit not maximal. Now it follows from Theorem \ref{MainHolepunchPaperResulThm} that there is a partial $L$-coloring $\sigma$ of $V(C_{Q^1})$, where every vertex of $Q^1\setminus\{v^0, v^1\}$ lies in $\textnormal{dom}(\sigma)$, and furthermore, each of $v^0, v^1$ has an $L_{\sigma}$-list of size at least three, and any extension of $\sigma$ to an $L$-coloring of $\textnormal{dom}(\sigma)\cup\{v^0, v^1\}$ extends to $L$-color all of $G^{\textnormal{small}}_{Q^1}$. As $G^{\textnormal{small}}_{Q^1}$ is induced in $G$, $\sigma$ is a proper $L$-coloring of its domain in $G$. Let $a=\sigma(x_1)$ and $b=\sigma(x^*)$. 

\vspace*{-8mm}
\begin{addmargin}[2em]{0em}
\begin{subclaim}\label{ExtSigtoTauSubCL} There is an $L$-coloring $\psi$ of $K$ with $\psi(x^*)=b$, where $\psi(v^1)\in L_{\sigma}(v^1)$ and $\psi(x_n)\neq a$.  \end{subclaim}

\begin{claimproof} This is straightforward to check using Theorem \ref{thomassen5ChooseThm}. In particular, if $a\in L_{\sigma}(v^1)$ then we first color $v^1$ with $a$ and then apply Theorem \ref{thomassen5ChooseThm}. Otherwise, since $|L_{\sigma}(y'')\setminus (L(x_n)\setminus\{a\})|\geq 1$, we use a color $c$ of $L_{\sigma}(v^1)\setminus (L(x_n)\setminus\{a\})$ on $v^1$, replace $a$ with $c$ in $L(x_n)$, and then apply Theorem \ref{thomassen5ChooseThm}. \end{claimproof}\end{addmargin}

By Subclaim \ref{ExtSigtoTauSubCL}, $\sigma$ extends to an $L$-coloring $\tau$ of $\textnormal{dom}(\sigma)\cup V(K)\cup V(Q-v^0)$ with $\tau(x_1)\neq\tau(x_n)$. Now, $V(G^{\textnormal{small}}_Q-v^0)$ is $(L, \tau)$-inert in $G$ by our choice of $\sigma$. As there are no chords of $C_Q$, except possibly $x_1x_2$, we have $|L_{\tau}(v^0)|\geq 3$. Each vertex of $\textnormal{Sp}(C)$ has an $L_{\tau}$-list of size at least three, contradicting 1). \end{claimproof}

\begin{center}\begin{tikzpicture}
\begin{scope}[xscale=-1]
 \draw (0,3) arc(90: 270:8cm and 3cm);
\end{scope}
\node[shape=circle, draw=black, inner sep=0pt, minimum size=0.55cm] (W) at (5, 0) {\small $w$};
\node[shape=circle, draw=black, inner sep=0pt, minimum size=0.55cm] (U) at (3.5, 0) {\small $u$};
\node[shape=circle, draw=white, inner sep=0pt, minimum size=0.55cm] (GSm) at (6.5, 0) {$G^{\textnormal{small}}_Q$};
\node[shape=circle, fill=black, inner sep=0pt, minimum size=0.25cm] (X1) at (6, 1.97) {};
\node[shape=circle, fill=black, inner sep=0pt, minimum size=0.25cm] (XN) at (6, -1.97) {};
\node[shape=circle, draw=black, inner sep=0pt, minimum size=0.55cm] (Z) at (5.5, 0.985) {\small $v^0$};
\node[shape=circle, draw=black, inner sep=0pt, minimum size=0.55cm] (Z') at (5.5, -0.985) {\small $v^1$};
\draw[-, line width=1.8pt, color=red] (X1) to (Z) to (W) to (Z') to (XN);
\node[shape=circle, draw=white, inner sep=0pt, minimum size=0.25cm] (X1L) at (6, 2.3) {\small $x_1$};
\node[shape=circle, draw=white, inner sep=0pt, minimum size=0.25cm] (XNL) at (6, -2.3) {\small $x_4$};
\node[shape=circle, draw=black, inner sep=0pt, minimum size=0.55cm] (Y0) at (0, 0) {\small $y$};

\draw[-] (XN) to (Y0) to (X1);
\draw[-] (U) to (W);

\node[shape=circle, fill=black, inner sep=0pt, minimum size=0.25cm] (X2) at (0, 3) {};

\node[shape=circle, draw=white, inner sep=0pt, minimum size=0.25cm] (X2L) at (0, 3.3) {\small $x_2$};
\node[shape=circle, draw=white, inner sep=0pt, minimum size=0.25cm] (X2L) at (0, -3.3) {\small $x_3$};
\node[shape=circle, fill=black, inner sep=0pt, minimum size=0.25cm] (X3) at (0, -3) {};

\draw[-] (X2) to (Y0) to (X3);

\draw[-] (X2) to [out=190, in=170] (X3);

\node[shape=circle, draw=black, fill=white, inner sep=0pt, minimum size=0.55cm] (X') at (7.2, -1.3) {\small $x^*$};
\draw[-] (Z') to (X');
\node[shape=circle, draw=white, inner sep=0pt, minimum size=0.55cm] (Y*) at (6.3, -1.5) {$K$};

\end{tikzpicture}
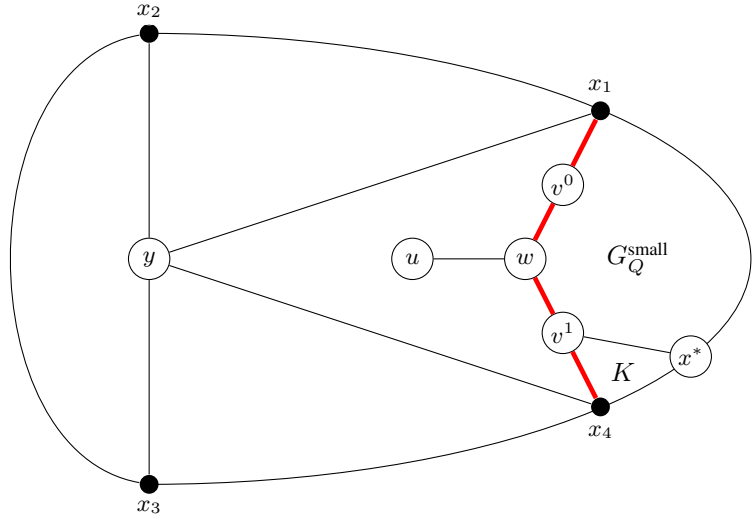
\captionof{figure}{The ellipse enclosing the diagram represents the cycle $C$}\label{ElipseFinCaseChordWrap}\end{center}

\begin{claim}\label{AnyLColQMinWExtAllSmall} Any $L$-coloring of $Q-w$ extends to $L$-color all of $V(G^{\textnormal{small}}_Q)$. \end{claim}

\begin{claimproof} Let $x'$ be the unique neighbor of $v^0$ on the path $G^{\textnormal{small}}_Q\cap C$ which is farthest from $x_1$. Possibly $x'=x_1$. Analogously, we let $x''$ be the unique neighbor of $v^1$ on the path $G^{\textnormal{small}}_Q\cap C$ which is farthest from $x_n$. We define a graph $H'$, where $H':=x_1v^0$ if $x'=x_1$ and otherwise $H':=G^{\textnormal{small}}_{x_1v^0x'}$. Define $H''$ analogously. Let $\sigma$ be an $L$-coloring of $V(Q-w)$, and suppose that $\sigma$ does not extend to an $L$-coloring of $G^{\textnormal{small}}_Q$. 

\vspace*{-8mm}
\begin{addmargin}[2em]{0em}
\begin{subclaim}\label{H'H''Close} $d(H', H'')\leq 1$ and at least one of $x', x''$ is an internal vertex of $G^{\textnormal{small}}_Q\cap C$. \end{subclaim}

\begin{claimproof} Suppose $d(H', H'')>1$. Thus, by Theorem \ref{thomassen5ChooseThm} applied to each of $H'$ and $H''$, $\sigma$ extends to an $L$-coloring $\tau$ of $V(H'\cup H'')$. Note that $|L_{\tau}(w)|\geq 3$, and it follows from Lemma \ref{PartialPathColoringExtCL0} applied to $G^{\textnormal{small}}_{x'v^0wv^1x''}$ that $\tau$ extends to an $L$-coloring of $G^{\textnormal{small}}_Q$, contradicting our assumption.  Thus, $d(H', H'')\leq 1$. Now suppose that neither of $x', x''$ is an internal vertex of $G^{\textnormal{small}}_Q\cap C$. Now, $G^{\textnormal{small}}_Q\cap C$ is a path of length one, so  $n=4$ and $x_1yx_n$ is a triangle separating $u$ from an internal vertex of $x_1\cdots x_n$, contradicting short-inseparability \end{claimproof}\end{addmargin}

Applying Subclaim \ref{H'H''Close}, we suppose without loss of generality that $x'$ is an internal vertex of $G^{\textnormal{small}}_Q\cap C$. By 2) of Claim \ref{NonDegenCaseNoHalfTargetCL5M}, $yv^0\in E(G)$. If $x''\neq x_n$, then, again by 2) of Claim \ref{NonDegenCaseNoHalfTargetCL5M}, we have $yv^1\in E(G)$ as well, contradicting short-inseparability. Thus, $v^1=x_n$ and we again contradict short-inseparability. \end{claimproof}

\begin{claim}\label{TQMinYObOP} $|T_Q\setminus\textnormal{Sp}(C)|=2$. \end{claim}

\begin{claimproof} Suppose not. Thus, $|T_Q\setminus\textnormal{Sp}(C)|\leq 1$. Since there is no chord of $Q$, except possibly $x_1x_2$ there is an $L$-coloring $\phi$ of $Q-w$ such that $\phi(x_1)\neq\phi(x_n)$, where each vertex of $\textnormal{Sp}(C)$ has an $L_{\phi}$-list of size at least three. By Claim \ref{AnyLColQMinWExtAllSmall}, $\phi$ extends to $L$-color $V(G^{\textnormal{small}}_Q)$. Since $|T_Q\setminus\textnormal{Sp}(C)|\leq 1$, it follows that $\phi$ extends to an $L$-coloring $\tau$ of $V(G^{\textnormal{small}}_Q+uw)$ such that each vertex of $T_Q\setminus\textnormal{Sp}(C)$ has an $L_{\tau}$-list of size at least three. Let $S:=V(G^{\textnormal{small}}_Q\cup C)\cup\{w\}$. Now, $\{x_1, \cdots, x_n\}$ is $(L, \tau)$-inert in $G$, so $S$ is $(L, \tau)$ inert in $G$. As $V(G^{\textnormal{large}}_Q\cap C)\cap\textnormal{dom}(\tau)=\{x_1, x_n\}=V(G^{\textnormal{large}}_Q\cap C)\setminus\textnormal{Sh}^2(C)$, the pair $(S, \tau)$ is a $(Q, uw)$-pair, contradicting our assumption on $Q$. \end{claimproof}

Applying Claim \ref{TQMinYObOP}, we let $T_Q\setminus\textnormal{Sp}(C)=\{z^0, z^1\}$, where, for each $k=0,1$, $z^k$ is adjacent to $u, w, v^k$. 

\begin{claim}\label{y'Andy''NoCommNbrw} $N(v^0)\cap N(v^1)=\{w\}$. Furthermore, if there exist $\tilde{v}\in\{v^0, v^1\}$ and $w^*\in V(G^{\textnormal{small}}_Q\setminus Q)$, where $G[N(w^*)\cap V(Q)]$ is a path of length two which does not contain $\tilde{v}$, then $|\textnormal{Sp}(C)|=1$ and the lone vertex of $\textnormal{Sp}(C)$ is adjacent to $y^*$. \end{claim}

\begin{claimproof} Suppose $N(v^0)\cap N(v^1)=\{w\}$. Thus, there is a $w^*\in V(G^{\textnormal{small}}_Q\setminus Q)$ with $N(v^0)\cap N(v^1)=\{w, w^*\}$, and $N(w)$ consists of the cycle $uz^0v^0w^*v^1z^1$. Note that $\{u, y', y''\}$ is an independent set, so there is an $L$-coloring $\phi$ of $V(Q-w)\cup\{u\}$ with $\phi(x_1)\neq\phi(x_n)$, where $|L_{\phi}(w)|\geq 4$. By Claim \ref{AnyLColQMinWExtAllSmall}, $\phi$ extends to $L$-color $V(G^{\textnormal{small}}_Q)\cup\{u\}$, so $\phi$ extends to an $L$-coloring $\tau$ of $V(G^{\textnormal{small}}_Q-w)\cup\{u\}$. Let $S:=V(G^{\textnormal{small}}_Q\cup C)\cup\{w\}$. Now, $|L_{\tau}(w)|\geq 3$ and thus $\{w\}$ is $(L, \tau)$-inert in $G$. Since $w$ is uncolored, each vertex of $T_Q$ has an $L_{\tau}$-list of size at least three. Thus, as in Claim \ref{TQMinYObOP}, $(S, \tau)$ is a $(Q, uw)$-pair, contradicting our assumption on $Q$. Thus, $N(v^0)\cap N(v^1)=\{w\}$. Now we prove the second part of the claim. Suppose  $\tilde{v}=v^1$ for the sake of definiteness. Thus, $G[N(w^*)\cap V(Q)]=wv^1x_n$. It suffices to show that there is a vertex of $\textnormal{Sp}(C)$ adajcent to $v^1$. Suppose not. Thus, there is an $L$-coloring $\phi$ of $\{x_1, x, x_n\}$ such that $\phi(x_1)\neq\phi(x_n)$ and $|L_{\phi}(v^0)|\geq 4$. Now it follows from Lemma \ref{PartialPathColoringExtCL0} that $\phi$ extends to an $L$-coloring $\tau$ of $V(G^{\textnormal{small}}_Q-v^0)$, and, by our choice of $\phi$, we have $|L_{\tau}(y')|\geq 3$. Since $v^0$ is not colored, each vertex of $\textnormal{Sp}(C)$ has an $L_{\tau}$-list of size at least three. Since $\tau$ already colors all of $V(G^{\textnormal{small}}_Q-v^0)$, we contradict 1) of Claim \ref{NonDegenCaseNoHalfTargetCL5M}.  \end{claimproof}

Now , if $\textnormal{Sp}(C)$ is nonempty, then the lone vertex of $\textnormal{Sp}(C)$ is adjacent to at most one of $v^0, v^1$, or else there is a separating 4-cycle, so suppose without loss of generality that $v^1$ has no neighbor in $\textnormal{Sp}(C)$. Let $A$ be the set of $w^*\in V(G^{\textnormal{small}}_Q\setminus Q)$ such that $w^*$ has at least three neighbors on $Q$. By Claim \ref{y'Andy''NoCommNbrw}, $N(v^0)\cap N(v^1)=\{w\}$, and furthermore, $|A|\leq 1$ and, if there is a $w^*\in A$, then $G[N(w^*)\cap V(Q)]=wv^1x_n$. Thus, there is an $L$-coloring $\phi$ of $\{x_1, w, x_n\}$ with $\phi(x_1)\neq\phi(x_n)$, where each vertex of $A$ has an $L_{\phi}$-list of size at least four. It follows from Lemma \ref{PartialPathColoringExtCL0} applied to $G^{\textnormal{small}}_Q$ that any extension of $\phi$ to an $L$-coloring of $V(Q)$ extends to $L$-color all of $G^{\textnormal{small}}_Q$. Thus, $\phi\in\textnormal{Crown}(Q, G^{\textnormal{small}}_Q)$, contradicting Claim \ref{SigmaSameColX1XNCL}. This proves Theorem \ref{StrengthendVerManColRes}.  \end{proof}

\section{Coloring and Deleting Paths}\label{IntermResSecPathDel}

In addition to Theorem \ref{StrengthendVerManColRes}, another ingredient we need for the proofs of Theorems \ref{FaceConnectionMainResult} and \ref{SingleFaceConnRes} is a result about coloring and deleting paths.

\begin{defn}\label{DefFilamentPathStrucAx}\emph{Given a graph $G$ embedded on a surface $\Sigma$ and a list-assignment $L$ for $G$, a \emph{filament} is a tuple $(P, T, T', \mathbf{f})$, where}
\begin{enumerate}[label=\emph{\arabic*)}]
\itemsep-0.1em
\item \emph{$P\subseteq G$ is a shortest path between its endpoints and $T, T'$ are disjoint, nonempty terminal subpaths of $P$}; AND
\item\emph{$P\setminus (T\cup T')$ is a path of length at least 30 and each vertex of distance at most one from $P\setminus (T\cup T')$ has an $L$-list of size at least five}; AND
\item\emph{Every facial subgraph of $G$ which contains a vertex of distance at most one from $P\setminus (T\cup T')$ is a triangle, and the pair $(\mathbf{f}, \textnormal{dom}(\mathbf{f}))$ is an $L$-reduction.}
\end{enumerate} \end{defn}

The key result we need about coloring and deleting paths is Theorem \ref{MainResultColorAndDeletePaths} below.

\begin{theorem}\label{MainResultColorAndDeletePaths}  Let $\Sigma$ be a surface and $G$ be a short-inseparable embedding on $\Sigma$, where $|V(G)|>5$. Let $L$ be a list-assignment for $G$. Then, for any filament $(P, T, T', \mathbf{f})$, there exists a subgraph $H$ of $G$ and such that
\begin{enumerate}[label=\arabic*)] 
\itemsep-0.1em
\item $P\subseteq H$ and $V(H\setminus P)\subseteq B_1(P\setminus (T\cup T'))$; AND
\item $d(H\setminus P, T\cup T')\geq 3$ and there is an extension of $\mathbf{f}$ to a partial $L$-coloring $\tau$ of $V(H)$ such that $(V(H), \tau)$ is an $L$-reduction. 
\end{enumerate} 
\end{theorem}

To make the proof of Theorem \ref{MainResultColorAndDeletePaths} easier to follow, it is useful to make explicit which vertices we are coloring and which vertices we are deleting but not coloring. 

\begin{defn}\label{AvoidNotationDef}\emph{Let $G$ be a graph, let $L$ be a list-assignment for $V(G)$, let $H$ be a subgraph of $G$, and let $U\subseteq V(H)$. We let $\textnormal{Red}_{G,L}(H\mid U)$ denote the set of $L$-colorings $\phi$ of $V(H\setminus U)$ such that $(\phi, V(H))$ is an $L$-reduction.} \end{defn}

Note that $\textnormal{Red}(H\mid U)$ only depends on $V(H)$, not on $H$. We usually drop the subscripts $G, L$ from the notation above if these are clear from the context. In the special case where $X=\varnothing$, we drop the $X$ from the notation above and write $\textnormal{Red}(H)$. In the case where $U=\{w\}$ for some $w$, we write $\textnormal{Red}(H\mid U)=\textnormal{Red}(H\mid w)$. To prove Theorem \ref{MainResultColorAndDeletePaths}, we first prove some intermediate results. We first introduce the following notation.

\begin{defn}\label{ListNotPathDelEndP}\emph{Let $G$ be a graph and let path $P\subseteq G$. We define the following.}
\begin{enumerate}[label=\emph{\arabic*)}] 
\itemsep-0.1em
\item\emph{Given an $x\in V(\mathring{P})$, we say that $x$ is a \emph{$P$-gap} if there does not exist a $y\in D_1(P)$ such that $|L(y)|\geq 5$ and $G[N(y)\cap V(P)]$ is a path of length two with midpoint $x$.}
\item\emph{Given and a $q\in V(P)$, we let $\textnormal{Mid}(P\mid q)$ denote the set of $v\in D_1(P)$ such that $|L(v)|\geq 5$ and $G[N(v)\cap V(P)]$ is a path of length two whose midpoint is $q$.}
\item\emph{We set $\textnormal{Mid}(P):=\bigcup_{q\in V(P)}\textnormal{Mid}(P\mid q)$. For any $w\in\textnormal{Mid}(P)$, we let $P^w$ be the path obtained from $P$ by replacing the midpoint of $G[N(w)\cap V(P)]$ with $w$.}
\end{enumerate}
\end{defn}

\begin{defn} \emph{Given $G, \Sigma$ and $L$ as in the statement of Theorem \ref{MainResultColorAndDeletePaths}, we say that a filament $(P, T, T', \mathbf{f})$ is \emph{bad} if there does not exist a subgraph $H$ of $G$ satisfying 1)-2) of Theorem \ref{MainResultColorAndDeletePaths}.} \end{defn}

The proof of Theorem \ref{MainResultColorAndDeletePaths} makes up Sections \ref{IntermResSecPathDel}-\ref{MainProofSubSecThm1PathCol}, and is organized as follows. Given $G, \Sigma, L$ as in the statement of Theorem \ref{MainResultColorAndDeletePaths}, we first prove that, given a bad filament $(P, T, T', \mathbf{f})$, and a subpath $Q$ of $P$ of length two with $d(Q, T\cup T')\geq 5$, at least one endpoint of $Q$ is a $P$-gap. This result is Lemma \ref{PathLenLemmaQGapMainDist2} and is the main ingredient in the proof of Theorem \ref{MainResultColorAndDeletePaths}. In Section \ref{BadFilPathLen4Sub}, we prove an intermediate result that we need in order to prove Lemma \ref{PathLenLemmaQGapMainDist2}, which is Lemma \ref{BadFil4PathLemmMain}. Finally, in Section \ref{MainProofSubSecThm1PathCol}, we apply Lemma \ref{PathLenLemmaQGapMainDist2} to prove that Theorem \ref{MainResultColorAndDeletePaths} holds. Throughout Sections \ref{IntermResSecPathDel}-\ref{MainProofSubSecThm1PathCol}, we repeatedly rely on the observation that, given $G, \Sigma, L$ be as in Theorem \ref{MainResultColorAndDeletePaths} and a path $P$ in $G$, $G$ has no copies of $K_{2,3}$ or $K_4$, and, in particular, for each $q\in V(P)$, $\textnormal{Mid}(P\mid q)$ has size at most one. 

\begin{lemma}\label{ExtfromAvoidQ'toAvoidQ} Let $G, \Sigma, L$ be as in the statement of Theorem \ref{MainResultColorAndDeletePaths}. Let $Q\subseteq G$ be a shortest path between its endpoints and let $Q'$ be a terminal subpath of $Q$. Suppose further that every vertex of $Q\setminus Q'$ has an $L$-list of size at least five. Then any element of $\textnormal{Red}(Q')$ extends to an element of $\textnormal{Red}(Q)$, and, in particular, if $Q'$ has length at most two, then any $L$-coloring of of $Q'$ extends to an element of $\textnormal{Red}(Q)$.  \end{lemma}

\begin{proof} Let $Q:=p_1\cdots p_{\ell}$ and $Q':=p_1\cdots p_{\ell}$ for some integers $1\leq\ell\leq r$. Let $\phi\in\textnormal{Red}(Q')$, and, for each $j=\ell, \cdots, r$, let $B_j$ be the set of extensions of $\phi$ to an element of $\textnormal{Red}(p_1Qp_j)$. We claim that $B_j\neq\varnothing$ for all $j=\ell, \cdots, r$. We show this by induction on $j$. This holds if $j=\ell$ since $B_{\ell}=\{\phi\}$. If $r=\ell$, then we are done, so let $j\in\{\ell, \cdots, r-1\}$ and let $\psi\in B_j$. As $Q$ is induced, $|L_{\psi}(p_{j+1})|\geq 4$. If $\textnormal{Mid}(Q\mid p_j)=\varnothing$, then any extension of $\psi$ to an $L$-coloring of $V(p_1Pp_{j+1})$ lies in $B_{j+1}$, so we are done in that case. Now suppose that $|\textnormal{Mid}(Q\mid p_j)|=1$ and let $w$ be the unique vertex of $\textnormal{Mid}(Q\mid p_j)$. As $|L(w)|\geq 5$, we have $|L_{\psi}(w)|\geq 3$, so there is an extension of $\psi$ to an $L$-coloring $\psi^*$ of $p_1Qp_{j+1}$ such that $|L_{\psi^*}(w)|\geq 3$, and thus $\psi^*\in B_{j+1}$. We conclude that $B_{j+1}\neq\varnothing$, as desired.  Note that, if $Q'$ has length two, then any $L$-coloring of $Q'$ is an element of $\textnormal{Red}(Q')$. \end{proof}

We use the following observation repeatedly, and it is immediate from the fact that $G$ is $K_4$-free and $P$ is a shortest path between its endpoints.

\begin{lemma}\label{ShortPathSecSFacts}
Let $G, \Sigma$ be as in the statement of Theorem \ref{MainResultColorAndDeletePaths} and $P$ be a shortest path between its endpoints. Then,
\begin{enumerate}[label=S\arabic*)]
\itemsep=-0.1em
\item Any two distinct vertices of $\textnormal{Mid}(P)$ are of distance at least two apart; AND
\item Let $p\in V(P)$ and let $w\in\textnormal{Mid}(P\mid p)$. If there is a $v\in N(w)\cap V(G\setminus P)$ with least two neighbors in $P$, then $G[N(v)\cap V(P)]$ is an edge which intersects the 2-path $G[N(w)\cap V(P)]$ precisely on a common endpoint, and, letting $qu$ be this edge and letting $S=(V(P)\setminus N(v))\cup (\textnormal{Mid}(P)\setminus\{w\})$, we have $d(v, S)\geq 2$. 
\end{enumerate}
\end{lemma}

The following is also immediate from two applications of Lemma \ref{ExtfromAvoidQ'toAvoidQ}.

\begin{lemma}\label{QSubPPTT'} Let $G, \Sigma, L$ be as in the statement of Theorem \ref{MainResultColorAndDeletePaths} and $(P, T, T', \mathbf{f})$ be a filament. Given a subpath $Q$ of $P\setminus (T\cup T')$ with $|E(Q)|\geq 2$,  $\mathbf{f}$ extends to an element of $\textnormal{Red}(P\setminus\mathring{Q})$. \end{lemma}

\section{Dealing with paths of length four}\label{BadFilPathLen4Sub}

Section \ref{BadFilPathLen4Sub} consists of the following result. 

\begin{lemma}\label{BadFil4PathLemmMain} Let $\mathbf{F}=(P, T, T', \mathbf{f})$ be a bad filament, let $Q$ be a subpath of $P\setminus (T\cup T')$ with $d(Q, T\cup T')\geq 4$ and let $p$ be the midpoint of $Q$. Then at least one vertex of $Q-p$ is a $P$-gap. \end{lemma}

\begin{proof} Let $Q=p_1p_2p_3p_4p_5$ of $P\setminus (T\cup T')$. Suppose toward a contradiction that no vertex of $Q-p_3$ is a $P$-gap. There exist distinct $w_1, w_2, w_4, w_5\in D_1(P)$ and $p_0, p_6\in V(P\setminus Q)$, where $p_0$ is the unique neighbor of $p_1$ on $P\setminus Q$, and $p_6$ is the unique neighbor of $p_5$ on $P\setminus Q$, and, for each $k\in\{1,2,4,5\}$, $G[N(w_k)\cap V(P)]=p_{k-1}p_kp_{k+1}$. We now define the following. 

\begin{enumerate}[label=\emph{\arabic*)}] 
\itemsep-0.1em
\item\emph{Let $X=\textnormal{Mid}(P\mid p_3)\cup\{w_2, w_4\}$ and let $G_X$ be the subgraph of $G$ induced by $V(P)\cup X$.}
\item\emph{Let $Y\subseteq\{p_1, p_5\}$, where $p_1\in Y$ if and only if the three vertices $p_0, p_1, w_2$ have a common neighbor, and, likewise, $p_5\in Y$ if and only if the three vertices $w_4, p_5, p_6$ have a common neighbor.}
\item\emph{Let $A$ be the set of $y\in V(G)\setminus(V(P)\cup X)$ such that $y$ has at least one neighbor in $\{p_3\}\cup X$ and at least three neighbors in $V(P)\cup X$. Let $A_Y$ be the set of $v\in V(G\setminus P)$ such that $N(v)$ contains either $\{p_0, p_1, w_2\}$ or $\{w_4, p_5, p_6\}$.}
\end{enumerate}

Note that $A_Y\subseteq A$, and each vertex of $Y$ has precisely five neighbors. To avoid repeatedly referencing Lemma \ref{QSubPPTT'}, we  introduce the following notation for the remainder of the proof of Lemma \ref{BadFil4PathLemmMain}: Given indices $0\leq i\leq j\leq 6$, we fix an extension of $\mathbf{f}$ to an element $\psi^{ij}$ of $\textnormal{Red}(P\setminus p_i\cdots p_j)$. In the case where $i=j$, we write $\psi^i$ in place of $\psi^{ii}$. 

\begin{claim}\label{ExtToTauP+XMinY} $|Y|\leq 1$. \end{claim}

\begin{claimproof} Suppose not. Thus, there exist vertices $v,v'$, where $A_Y=\{v, v'\}$, and, in particular, $v$ is adjacent to $p_0, p_1, w_4$, and $v'$ is adjacent to $w_4, p_5, p_6$. Now, $\psi^{15}$ extends to an $L$-coloring $\sigma$ of $\textnormal{dom}(\psi)\cup\{p_2, p_4\}$, where $|L_{\sigma}(p_1)|\geq 4$ and $|L_{\sigma}(p_5)|\geq 4$. Let $H$ be the subgraph of $G$ induced by $V(G_X)\setminus\{w_2, w_4\}$ and $U=\{p_1, p_5\}$. Now, $\sigma$ extends to an $L$-coloring $\tau$ of $V(H)\setminus U$, and $\tau\in\textnormal{Red}(H\mid U)$, contradicting our assumption on $\mathbf{F}$. \end{claimproof}

\begin{claim}\label{YNbrInPAddToXW2W4} If $A\setminus A_Y\neq\varnothing$, then one of the following holds. 
\begin{enumerate}[label=\alph*)] 
\itemsep-0.1em
\item $|A\setminus A_Y|=1$ and $N(y)\cap (V(P)\cup X)=\{w_2, p_3, w_4\}$ for the lone $y\in A\setminus A_Y$; OR
\item $|A\setminus A_Y|=2$ and $\textnormal{Mid}(P\mid p_3)=\varnothing$, and, in particular, there exist $y, y'$, where $A\setminus A_Y=\{y, y'\}$, and $N(y)\cap (V(P)\cup X)=\{w_2, p_3, p_4\}$ and $N(y')\cap (V(P)\cup X)=\{p_2, p_3, w_4\}$.
\end{enumerate}
\end{claim}

\begin{claimproof} Suppose not. Thus, $A\setminus A_Y\neq\varnothing$ and neither a) nor b) hold. Note that there is a $y\in A\setminus A_Y$ with $|N(y)\cap\{w_2, w_4\}|<2$, or else a) holds. It follows that each vertex of $A\setminus A_Y$ is adjacent to precisely one of $w_2, w_4$, and sice $A\setminus A_Y\neq\varnothing$, we get $\textnormal{Mid}(P\mid p_3)=\varnothing$. Furthermore, $|A\setminus A_Y|\leq 2$. If $|A\setminus A_Y|=2$, then b) of Claim \ref{YNbrInPAddToXW2W4} is satisfied, so we get $A\setminus A_Y=\{y\}$. We suppose without loss of generality that $N(y)\cap (V(P)\cup\{w_2, w_4\})=\{w_2, p_3, p_4\}$. 

\vspace*{-8mm}
\begin{addmargin}[2em]{0em}
\begin{subclaim}\label{SubCLForYSize1}  $Y=\{p_5\}$. \end{subclaim}

\begin{claimproof} we first show that $|Y|=1$. Suppose not. Thus, by Claim \ref{ExtToTauP+XMinY}, $Y=\varnothing$, so $A_Y=\varnothing$. Now, $\psi^3$ extends to an $L$-coloring $\tau$ of $V(G_X+y)$, and, since $\mathbf{F}$ is a bad filament, $\tau\not\in\textnormal{Red}(G_X+y)$, so there is a $z\in D_1(G_X+y)$ with $|L(z)|\geq 5$ and $|L_{\tau}(z)|<3$. There is an extension of $\psi$ to an $L$-coloring $\pi$ of $V(P)\cup\{y\}$, where $|L_{\pi}(y)|\geq 2$. Then $\pi\in\textnormal{Red}(G_X+y\mid w_2)$, contradicting our assumption that $\mathbf{F}$ is a bad filament. Thus, $|Y|=1$. Now suppose that $Y=\{p_1\}$. Thus, there is a $v\in V(G\setminus P)$ with $p_0, p_1, w_2\in N(v)$. There is a $c\in L_{\psi}(p_2)$ with $|L_{\psi}(p_1)\setminus\{c\}|\geq 4$. Now let $\tau$ be an arbitrary extension of $\psi^{45}$ to an $L$-coloring of $V(P-p_1)\cup\{w_4\}$. Then $\tau\in\textnormal{Red}(G_X-w_2\mid p_1)$, again contradicting our assumption that $\mathbf{F}$ is a bad filament. \end{claimproof}\end{addmargin}

As $Y=\{p_5\}$, there is a $v'\in V(G\setminus P)$ adjacent to $w_4, p_5, p_6$. Set $H:=(G_X+y)-w_4$. Then $\psi^{12}$ extends to an $L$-coloring $\tau$ of $V(H)\setminus\{p_5\}$ with $\tau(p_4)=d$, and $\tau\in\textnormal{Red}(H\mid p_5)$, contradicting our assumption on $\mathbf{F}$. \end{claimproof}

\begin{claim}\label{IfYEmptyA=2CL} If $Y=\varnothing$, then $|A|=2$. \end{claim}

\begin{claimproof} Suppose that $Y=\varnothing$ and $|A|\neq 2$. Thus, $A=A\setminus A_Y$ and $|A|\leq 1$. If $A=\varnothing$, then, letting $tau$ be an arbitrary extension of $\psi^3$ to an $L$-coloring of $V(G_X)$, we have $\tau\in\textnormal{Red}(G_X)$, contradicting our assumption on $\mathbf{F}$. Thus, $|A|=1$. Let $y$ be the lone vertex of $A$. By Claim \ref{YNbrInPAddToXW2W4}, $N(y)\cap (V(P)\cup X)=\{w_2, p_3, w_4\}$. Now, if $\textnormal{Mid}(P\mid p_3)=\varnothing$, then $\psi^3$ extends to an $L$-coloring $\tau$ of $V(P-p_3)\cup\{w_2, p_3, w_4\}$ with $|L_{\tau}(y)|\geq 3$, and $\tau\in\textnormal{Red}(G_X)$, contradicting our assumption on $\mathbf{F}$.  Thus, $\textnormal{Mid}(P\mid p_3)\neq \varnothing$. Let $w_3$ be the lone vertex of $\textnormal{Mid}(P\mid p_3)$. Note that $G[N(p_3)]$ is a 6-cycle, and $\psi^3$ extends to an $L$-coloring $\tau$ of $V(P-p_3)\cup\{w_2, w_3, w_4\}$, where $|L_{\tau}(p_3)|\geq 2$. Let $U=\{p_3\}$. Then $\tau\in\textnormal{Red}(G_X\mid U)$, again contradicting our assumption on $\mathbf{F}$.  \end{claimproof}

\begin{claim}\label{YNeqVarnothingSubCL} $|Y|=1$. \end{claim}

\begin{claimproof} Suppose not. By Claim \ref{ExtToTauP+XMinY}, $Y=\varnothing$. By Claim \ref{IfYEmptyA=2CL}, $|A|=2$. By Claim \ref{YNbrInPAddToXW2W4}, $\textnormal{Mid}(P\mid p_3)=\varnothing$, and there exist distinct vertices $y, y'$, where $A=\{y, y'\}$, and $N(y)\cap (V(P)\cup\{w_2, w_4\})=\{w_2, p_3, p_4\}$ and $N(y')\cap (V(P)\cup\{w_2, w_4\})=\{p_2, p_3, w_4\}$. We now fix $\psi=\psi^3$. 

\vspace*{-8mm}
\begin{addmargin}[2em]{0em}
\begin{subclaim}\label{AvoidThreeColSetsSimSubCL} $|L_{\psi}(p_3)|=3$ and $L_{\psi}(p_3)=L_{\psi}(w_2)=L_{\psi}(w_4)$. \end{subclaim}

\begin{claimproof} Suppose not. Thus, there exist a $w\in\{w_2, w_4\}$ and a $c\in L_{\psi}(w)$, where $|L_{\psi}(p_3)\setminus\{c\}|\geq 3$, say $w=w_2$ without loss of generality, and let $H:=G[V(P)\cup\{w_2, w_4, y'\}]$. Now, $\psi$ extends to an $L$-coloring $\tau$ of $V(P-p_3)\cup\{w_2, w_4, y'\}$, where $L_{\tau}(p_3)|\geq 2$. Let $U=\{p_3\}$. Since $\mathbf{F}$ is a bad filament, $\tau\not\in\textnormal{Red}(H\mid U)$, so there is a vertex $z$ of $D_1(H)$ with $|L(z)|\geq 5$ and $|L_{\tau}(z)|<3$. Note that $z$ is unique, and $N(z)\cap\textnormal{dom}(\tau)=\{y', w_4, p_5\}$. Furthermore, $G[N(w_4)]$ is a 5-cycle. We now uncolor $w_4$. Note that $\psi$ extends to an $L$-coloring $\pi$ of $\textnormal{dom}(\psi)\cup\{w_2, y'\}$, where $|L_{\pi}(w_4)|\geq 3$ and $\pi(w_2)=c$. Let $H^*:=H-w_4$ and $U^*:=\{p_3, w_4\}$. Then $\pi\in\textnormal{Red}(H^*\mid U^*)$, contradicting our assumption on $\mathbf{F}$. \end{claimproof}\end{addmargin}

It follows from Subclaim \ref{AvoidThreeColSetsSimSubCL} that $\psi$ extends to an $L$-coloring $\tau$ of $V(P-p_3)\cup\{y, y'\}$, where $|L_{\tau}(p_3)|\geq 3$. We set $H:=G[V(P)\cup\{y, y'\}]$. Then $\tau\in\textnormal{Red}(H\mid p_3)$, contradicting our assumption on $\mathbf{F}$. \end{claimproof}

Applying Claim \ref{YNeqVarnothingSubCL}, we suppose without loss of generality that $Y=\{p_1\}$. Thus, there is a $v\in V(G)$ which is adjacent to all three of $p_0, p_1, w_2$, and $N(v)\cap (V(P)\cup\textnormal{Mid}(P))=\{p_0, p_1, w_2\}$. We now set $\psi=\psi^{13}$.

\begin{claim}\label{AMinusAYSetSize2} $|A\setminus A_Y|=2$. \end{claim}

\begin{claimproof} Suppose not. Thus, by Claim \ref{YNbrInPAddToXW2W4}, $|A\setminus A_Y|\leq 1$, and, for any $y\in A\setminus A_Y$, we have $N(y)\cap (V(P)\cup X)=\{w_2, p_3, w_4\}$. Since $|L_{\psi}(p_1)|\geq 4$ and $|L_{\psi}(p_2)|\geq 5$, there is a $c\in L_{\psi}(y_2)$ with $|L_{\psi}(p_1)\setminus\{c\}|\geq 4$. Now, if $\textnormal{Mid}(P\mid p_3)=\varnothing$, then letting $\tau$ be an extension of $\psi$ to an $L$-coloring of $V(P-p_1)$ using $c$ on $p_2$, and letting $U=\{p_3\}$, we get $\tau\in\textnormal{Red}(P\mid U)$, contradicting our assumption on $\mathbf{F}$.  Thus, $\textnormal{Mid}(P\mid p_3)\neq\varnothing$. Let $w_3$ be the lone vertex of $\textnormal{Mid}(P\mid p_3)$. If $A\setminus A_Y=\varnothing$, then, letting $\pi$ be any extension of $\psi$ to an $L$-coloring of $\textnormal{dom}(\psi)\cup\{p_2, p_3\}\cup X$ using $c$ on $p_2$, we get $\pi\in\textnormal{Red}(G_X\mid p_1)$. Thus, $|A\setminus A_Y|=1$. Let $y$ be the lone vertex of $A\setminus A_Y$.  Note that $G[N(p_3)]$ is a 6-cycle, and $\psi$ extends to an $L$-coloring $\tau$ of $\textnormal{dom}(\psi)\cup\{p_2, w_2, w_3, w_4\}$, where $|L_{\tau}(p_2)=c$ and $|L_{\tau}(p_3)|\geq 2$. Letting $U=\{p_1, p_3\}$, we have $\tau\in\textnormal{Red}(G_X\mid U)$, again contradicting our assumption on $\mathbf{F}$. \end{claimproof}

By Claim \ref{AMinusAYSetSize2}, $|A\setminus A_Y|=2$, so $\textnormal{Mid}(P\mid p_3)=\varnothing$, and furthermore, there exist vertices $y, y'$, where $A\setminus A_Y=\{y, y'\}$, and $N(y)\cap V(G_X) =\{w_2, p_3, p_4\}$ and $N(y')\cap V(G_X)=\{p_2, p_3, w_4\}$. We have $|L_{\psi}(p_1)|\geq 4$ and $|L_{\psi}(w_2)|\geq 5$. Thus, there is a $c\in L_{\psi}(w_2)$ with $|L_{\psi}(p_1)\setminus\{c\}|\geq 4$. Let $\pi$ be an extension of $\psi$ to an $L$-coloring of $\textnormal{dom}(\psi)\cup\{w_2\}$, where $\pi(w_2)=c$. There is an extension of $\pi$ to an $L$-coloring $\tau$ of $\textnormal{dom}(\pi)\cup\{w_2, w_4, y\}$, where $|L_{\tau}(p_3)|\geq 2$. Then $\tau\in\textnormal{Red}(G_X+y\mid U^*)$, again contradicting our assumption on $\mathbf{F}$. completes the proof of Lemma \ref{BadFil4PathLemmMain}. \end{proof}

\section{Dealing with paths of length two}\label{InterSub1SecPathLen2Del}

With Lemma \ref{BadFil4PathLemmMain} in hand, we prove Lemma \ref{PathLenLemmaQGapMainDist2} below, whose proof makes up all of Section \ref{InterSub1SecPathLen2Del}. 

\begin{lemma}\label{PathLenLemmaQGapMainDist2} Let $\mathbf{F}=(P, T, T', \mathbf{f})$ be a bad filament and let $Q$ be a subpath of $P\setminus (T\cup T')$ of length two, where $d(Q, T\cup T')\geq 5$. Then at least one endpoint of $Q$ is a $P$-gap. \end{lemma}

\begin{proof} Let $Q=p_1p_2p_3$. Suppose that neither $p_1$ nor $p_3$ is a $P$-gap. Now, there exist $p_0, p_4\in V(P)$ and vertices $w_1, w_3\in\textnormal{Mid}(P)$, where $p_0$ is the unique neighbor of $p_1$ in $P-p_2$, and $p_4$ is the unique neighbor of $p_3$ in $P-p_2$, and $\{w_k\}=\textnormal{Mid}(P\mid p_k)$ for each $k=1,3$. We now define the following.

\begin{enumerate}[label=\emph{\arabic*)}]
\itemsep-0.1em
\item We let $X=\{w_1, w_3\}\cup\textnormal{Mid}(P\mid p_2)$.
\item We let $A$ be the set of $y\in V(G)\setminus (V(P)\cup X)$ such that $y$ has at least one neighbor in $\{p_2\}\cup X$ and at least three neighbors in $V(P)\cup X$. Furthemore, we let $B$ be the set of $y\in A$ such that $|N(y)\cap\{w_1, w_3\}|=1$.
\item We let $G_X=G[V(P)\cup X]$. Furthermore, we let $q_-$ be the lone neighbor of $p_0$ in $P-p_1$ and let $q_+$ be the lone neighbor of $p_4$ in $P-p_3$, so that $q_-Pq_+$ is a subpath of $P$ of length six.
\item We let $B=B^{\textnormal{in}}\cup B^{\textnormal{out}}$, where $B^{\textnormal{in}}$ is the set of $x\in B$ such that $N(x)\cap V(G_X)$ is either $\{w_1, p_2, p_3\}$ or $\{p_1, p_2, w_3\}$, and $B^{\textnormal{out}}$ is the set of $y\in B$ such that $N(y)\cap V(G_X)$ is either $\{q_-, p_0, w_1\}$ or $\{w_3, p_4, q_+\}$. Furthermore, we let $A^*$ be the set of $y\in A$ such that $|N(y)\cap X|\geq 2$.
\end{enumerate}

 Note that, for each $y\in A^*$, $N(y)\cap (V(P)\cup X)=\{w_1, p_2, w_3\}$, and $|A^*|\leq 1$. Likewise, for each $w\in\{w_1, w_3\}$, each of $N(w)\cap B^{\textnormal{in}}$ and $N(w)\cap B^{\textnormal{out}}$ has size at most one. Given a partial $L$-coloring $\psi$ of $G$, we say that $\psi$ is \emph{$B^{\textnormal{out}}$-avoiding} if, for each $y\in B^{\textnormal{out}}$, either $y\in\textnormal{dom}(\psi)$ or $|L_{\psi}(y)|\geq 3$. The following is immediate.

\begin{claim}\label{PMinusQSigmaExtUnionToAvB}
 Let $\psi$ be a partial $L$-coloring of $G$ with $\textnormal{dom}(\psi)\subseteq V(P)$ and let $S\subseteq\{w_1, w_3\}$, where each vertex of $S$ has at most one neighbor in $\textnormal{dom}(\psi)$. Then $\psi$ extends to a $B^{\textnormal{out}}$-avoiding $L$-coloring of $\textnormal{dom}(\psi)\cup S$.
\end{claim}

To avoid repeatedly referencing Lemma \ref{QSubPPTT'}, we also introduce the following notation for the remainder of the proof of Lemma \ref{PathLenLemmaQGapMainDist2}. Given indices $0\leq i\leq j\leq 4$, and, letting $P'=p_{i}\cdots p_{j}$, we fix an extension of $\mathbf{f}$ to an element $\psi^{ij}$ of $\textnormal{Red}(P\setminus P')$. In the case where $i=j$, we write $\psi^i$ in place of $\psi^{ii}$.

\begin{claim}\label{Bout+Bout-GapAtMostOne}
 If $p_4$ is a $P$-gap, then $B^{\textnormal{in}}\cap N(w_3)=\varnothing$. Likewise, if $p_0$ is a $P$-gap, then $B^{\textnormal{in}}\cap N(w_1)=\varnothing$.
\end{claim}

\begin{claimproof} Suppose not. WIthout loss of generality, we suppose toward a contradiction that $p_4$ is a $P$-gap, but $B^{\textnormal{in}}\cap N(w_3)\neq\varnothing$. Let $x_+$ be the lone vertex of $B^{\textnormal{in}}\cap N(w_3)$, and $H:=G[V(P)\cup\{x_+, w_3\}]$. Note that $A=B$.

\vspace*{-8mm}
\begin{addmargin}[2em]{0em}
\begin{subclaim}\label{AtMostOneCont5Cycle} Either there is no vertex adjacent to all three of $p_0, p_1, x_+$, or there is no vertex adjacent to all three of $x_+, w_3, p_4$.
\end{subclaim}

\begin{claimproof} Suppose there exist vertices $z, z_*$, where $z$ is adajcent to all of $p_0, p_1, x_+$ and $z^*$ is adjacent to all of $x_+, w_3, w_4$. Thus, $G[N(p_1)]$ and $G[N(w_3)]$ are 5-cycles. We now let $\psi=\psi^{13}$. There is a $c\in L_{\psi}(x_+)$ with $|L_{\psi}(x_+)\setminus\{c\}|\geq 4$, and $\psi$ extends to an $L$-coloring $\tau$ of $\textnormal{dom}(\psi)\cup\{x_+, p_2\}$ such that $|L_{\tau}(p_1)|\geq 3$ and $\tau(x_+)=c$. Let $U=\{p_1, w_3\}$. Then $\tau\in\textnormal{Red}(H\mid U)$, contradicting our assumption on $\mathbf{F}$. \end{claimproof}\end{addmargin}

\vspace*{-8mm}
\begin{addmargin}[2em]{0em}
\begin{subclaim}\label{RuleOutCommX+W3P4Sub} There is no vertex adjacent to all three of $x_+, w_3, p_4$. \end{subclaim}

\begin{claimproof} Suppose there is such a vertex $z_*$. Thus, $G[N(w_3)]$ is a 5-cycle and $\psi^2$ extends to an $L$-coloring $\tau$ of $\textnormal{dom}(\psi)\cup\{x_+, p_3\}$, where $|L_{\tau}(w_3)|\geq 2$. Then $\tau\in\textnormal{Red}(H\mid u_3)$, contradicting our assumption on $\mathbf{F}$. \end{claimproof}\end{addmargin}

Over the course of Subclaims \ref{p0PgapSubCLOp}-\ref{p0PGapSubCLM11}, we now show that $p_0$ is a $P$-gap. 

\vspace*{-8mm}
\begin{addmargin}[2em]{0em}
\begin{subclaim}\label{p0PgapSubCLOp} Either $p_0$ is a $P$-gap or $B^{\textnormal{out}}\cap N(w_1)=\varnothing$  \end{subclaim}

\begin{claimproof} Suppose not. Let $w_0$ be the lone vertex of $\textnormal{Mid}(P\mid p_0)$ and let $z$ be the lone vertex of $B^{\textnormal{out}}\cap N(w_1)$. Thus, $G[N(p_0)]$ is the 5-cycle $q_-zw_1p_1w_0$. Let $\psi=\psi^{03}$. By Claim \ref{PMinusQSigmaExtUnionToAvB}, $\psi$ extends to a $B^{\textnormal{out}}$-avoiding $L$-coloring $\sigma$ of $\textnormal{dom}(\psi)\cup\{w_3\}$. Since $|L_{\sigma}(p_1)|\geq 5$ and $|L_{\sigma}(p_0)|\geq 4$, $\sigma$ extends to an $L$-coloring $\sigma'$ of $\textnormal{dom}(\sigma)\cup V(Q)\cup\{x_+\}$ with $|L_{\sigma'}(p_0)|\geq 4$. Letting $U=\{p_0\}$, we have $\sigma'\in\textnormal{Red}(H\mid U)$, contradicting our assumtion on $\mathbf{F}$. \end{claimproof}\end{addmargin}

\vspace*{-8mm}
\begin{addmargin}[2em]{0em}
\begin{subclaim}\label{P0GapSubBInMinusVarn} Either $p_0$ is a $P$-gap or $B^{\textnormal{in}}\cap N(w_1)=\varnothing$. \end{subclaim}

\begin{claimproof} Suppose not. Let $w_0$ be the lone vertex of $\textnormal{Mid}(P\mid p_0)$, and let $x_-$ be the lone vertex of $B^{\textnormal{in}}\cap N(w_1)=\varnothing$. Note that $G[N(p_2)]$ is a 6-cycle Now, $\psi^{2,3}$ extends to an $L$-coloring $\psi^*$ of $\textnormal{dom}(\psi)\cup\{w_3\}$ such that each vertex of $B^{\textnormal{out}}\cap N(w_3)$ has an $L_{\psi^*}$-list of size at least three, and $\psi^*$ extends to an $L$-coloring $\psi'$ of $\textnormal{dom}(\psi^*)\cup\{w_1, x_+, p_3\}$ such that $|L_{\psi'}(p_2)|\geq 2$. As $\mathbf{F}$ is a bad filament, $\psi'\not\in\textnormal{Red}(G_X+x_+\mid U)$, so there is a $y\in D_1(G_X+x_+)$ with $|L(y)|\geq 5$ and $|L_{\psi'}(y)|<3$. Note that $x_+\in N(y)$, and, $w_1, p_3\not\in N(y)$, so $N(y)\cap\textnormal{dom}(\psi')\subseteq\{x_+\}\cup V(P^{w_3})$. By S2) of Lemma \ref{ShortPathSecSFacts} applied to $P^{w_3}$, $y$ is either adjacent to $p_0, p_1, x_+$ or $x_+, w_3, p_4$. Since $w_0\not\in N(x_+)$, it follows that $y$ is adjacent to each of $x_+, w_3, p_4$, contradicting Subclaim \ref{RuleOutCommX+W3P4Sub}. \end{claimproof}\end{addmargin} 

We can now rule out the possibility that $\textnormal{Mid}(P\mid p_0)\neq\varnothing$. 

\vspace*{-8mm}
\begin{addmargin}[2em]{0em}
\begin{subclaim}\label{p0PGapSubCLM11} Each of $p_0, p_2, p_4$ is a $P$-gap. \end{subclaim}

\begin{claimproof} We have $\textnormal{Mid}(P\mid p_4)=\varnothing$ by assumption, and $\textnormal{Mid}(P\mid p_2)=\varnothing$, since $A=B$. We just need to show that $p_0$ is a $P$-gap. Suppose not, and let $w_0$ be the lone vertex of $\textnormal{Mid}(P\mid p_0)$. Applying Subclaims \ref{p0PgapSubCLOp} and \ref{P0GapSubBInMinusVarn}, $B\cap N(w_1)=\varnothing$. Since $A=B$, we have $A\cap N(w_1)=\varnothing$. Now, $\psi^{2,3}$ extends to an $L$-coloring $\psi^*$ of $\textnormal{dom}(\psi)\cup\{w_3\}$ such that each vertex of $B^{\textnormal{out}}_+$ has an $L_{\psi^*}$-list of size at least three.  Let $\psi'$ be an arbitrary extension of $\psi^*$ to an $L$-coloring of $\textnormal{dom}(\psi^*)\cup\{x_+, p_2, p_3, w_1\}=V(G_X+x_+)$. Then $\psi'\in\textnormal{Red}(G_X+x_+)$, contradicting our assumption on $\mathbf{F}$.\end{claimproof}\end{addmargin} 

\vspace*{-8mm}
\begin{addmargin}[2em]{0em}
\begin{subclaim}\label{BinMinusNonEmpSubCL} $B^{\textnormal{in}}\cap N(w_1)\neq\varnothing$ \end{subclaim}

\begin{claimproof} Suppose $B^{\textnormal{in}}\cap N(w_1)=\varnothing$. Applying Claim \ref{PMinusQSigmaExtUnionToAvB}, let $\sigma$ be an extension of $\psi^{13}$ to a $B^{\textnormal{out}}$-avoiding $L$-coloring of $\textnormal{dom}(\psi^{13})\cup\{w_1, w_3\}$. Of there is no vertex adjacent to all three of $p_0, p_1, x_+$, then we take $\tau$ to be an arbitrary extension of $\sigma$ to an $L$-coloring of $\textnormal{dom}(\sigma)\cup V(Q)\cup\{x_+\}$, and $\tau\in\textnormal{Red}(G_X+x')$, contradicting our assumption on $\mathbf{F}$. Thus, there is a vertex $z$ adjacent to all three of $p_0, p_1, x_+$, and $G[N(p_1)]$ is a 5-cycle. Furthermore, $\sigma$ extends to an $L$-coloring $\tau$ of $\textnormal{dom}(\sigma)\cup\{x_+, p_2, p_3\}$, where $|L_{\tau}(p_1)|\geq 2$. Then $\tau\in\textnormal{Red}(G_X+x_+\mid p_1)$, contradicting our assumption on $\mathbf{F}$. \end{claimproof}\end{addmargin}

Applying Subclaim \ref{BinMinusNonEmpSubCL}, we now let $x_-$ be the lone vertex of $B^{\textnormal{in}}\cap N(w_1)$. 

\vspace*{-8mm}
\begin{addmargin}[2em]{0em}
\begin{subclaim}\label{ZExAdjAllOfP0P1X+SUBcL} There is a vertex $z$ adjacent to all three of $p_0, p_1, x_+$. \end{subclaim}

\begin{claimproof} Suppose not. As above, we set $\psi=\psi^{13}$. Applying Claim \ref{PMinusQSigmaExtUnionToAvB}, we let $\sigma$ be an extension of $\psi$ to a $B^{\textnormal{out}}$-avoiding $L$-coloring of $\textnormal{dom}(\psi)\cup\{w_1, w_3\}$. Now, $|L_{\sigma}(p_2)|\geq 3$ and $|L_{\sigma}(x_+)|\geq 4$, so there is a $d\in L_{\sigma}(x_+)$ with $|L_{\sigma}(p_2)\setminus\{d\}|\geq 3$. Since $x_+p_3\not\in E(G)$, $\sigma$ extends to an $L$-coloring $\tau$ of $\textnormal{dom}(\sigma)\cup\{p_1, p_3, x_+\}$ such that $|L_{\tau}(p_2)|\geq 2$. Then $\tau\in\textnormal{Red}(G_X+x'\mid p_2)$, contradicting our assumption on $\mathbf{F}$. \end{claimproof}\end{addmargin}

It follows from Subclaim \ref{ZExAdjAllOfP0P1X+SUBcL} that there is a unique vertex $z$ adjacent to all three of $p_0, p_1, x_+$, so $G[N(p_1)]$ is the 5-cycle $p_0w_1p_2z_+z$. We now fix $\psi=\psi^{13}$, and, applying Claim \ref{PMinusQSigmaExtUnionToAvB}, we let $\sigma$ be an extension of $\psi$ to an $B^{\textnormal{out}}$-avoiding $L$-coloring of $\textnormal{dom}(\psi)\cup\{w_1, w_3\}$. Now, $|L_{\sigma}(p_1)|\geq 3$ and $|L_{\sigma}(x_+)|\geq 4$, so there is a $c\in L_{\sigma}(x_+)$ with $|L_{\sigma}(p_1)\setminus\{c\}|\geq 3$. Furthermore, $|L_{\sigma}(p_2)\setminus\{c\}|\geq 2$ and $|L_{\sigma}(p_3)|\geq 3$. Since $x_+p_3\not\in E(G)$, it follows that $\sigma$ extends to an $L$-coloring $\tau$ of $\textnormal{dom}(\sigma)\cup\{x_+, p_3\}$, where $\tau(x_+)=c$ and $|L_{\tau}(p_2)|\geq 2$. Thus, letting $U=\{p_2, p_3\}$, $U$ is $(L, \tau)$-inert in $G$, and $\tau\in\textnormal{Red}(G_X+x_+\mid U)$, contradicting our assumption on $\mathbf{F}$. \end{claimproof}

Applying Claim \ref{Bout+Bout-GapAtMostOne}, we have the following. 

\begin{claim}\label{CLMidPp2Nonempty1} $\textnormal{Mid}(P\mid p_2)\neq\varnothing$. \end{claim}

\begin{claimproof} Suppose toward a contradiction that $\textnormal{Mid}(P\mid p_2)=\varnothing$. We now have the following.

\vspace*{-8mm}
\begin{addmargin}[2em]{0em}
\begin{subclaim}\label{atmostonep0p04PgSubCLM13} At most one of $p_0, p_4$ is a $P$-gap. \end{subclaim}

\begin{claimproof} Suppose each of $p_0, p_4$ is a $P$-gap. By Claim \ref{Bout+Bout-GapAtMostOne}, $B^{\textnormal{in}}=\varnothing$. Applying Claim \ref{PMinusQSigmaExtUnionToAvB}, let $\sigma$ be an extension of $\psi^{13}$ to a $B^{\textnormal{out}}$-avoiding $L$-coloring of $V(P\setminus Q)\cup\{w_1, w_3\}$. Since $B^{\textnormal{in}}=\varnothing$, each vertex of $B$ has an $L_{\sigma}$-list of size at least three. Let $\tau$ be an extension of $\sigma$ to an $L$-coloring of $V(G_X)=V(P)\cup\{w_1, w_3\}$. Since $\mathbf{F}$ is a bad filament, there is a $y\in D_1(G_X)$ with $|L(y)|\geq 5$ and $|L_{\tau}(y)|<3$. Since each of $p_0, p_4$ is a $P$-gap, it follows that $z\in A\setminus B$. Since $\textnormal{Mid}(P\mid p_2)\neq\varnothing$, we get $A=B\cup\{y\}$, and $N(y)\cap (V(P)\cup X)=\{w_1, p_2, w_3\}$. In particular, $\{y\}=A^*$ and $|L_{\tau}(y)|=2$. Thus, $\tau$ extends to an $L$-coloring $\tau'$ of $V(G_X)\cup\{y\}$. We now let $H=G[V(P)\cup X\cup\{y\}]$. Then $\tau'\in\textnormal{Red}(H)$, contradicting our assumption on $\mathbf{F}$. \end{claimproof}\end{addmargin}

Applying Subclaim \ref{atmostonep0p04PgSubCLM13}, we suppose without loss of generality that $p_0$ is not a $P$-gap. Let $w_0$ be the unique vertex of $\textnormal{Mid}(P\mid p_0)$. By Lemma \ref{BadFil4PathLemmMain}, $p_4$ is a $P$-gap, so $\textnormal{Mid}(P\mid p_4)=\varnothing$. Thus, by Claim \ref{Bout+Bout-GapAtMostOne}, $B^{\textnormal{in}}\cap N(w_3)=\varnothing$. 

\vspace*{-8mm}
\begin{addmargin}[2em]{0em}
\begin{subclaim}\label{B-NonemptySubCL131} $B^{\textnormal{out}}\cap N(w_1)=\varnothing$. \end{subclaim}

\begin{claimproof} Suppose not, and let $z$ be the lone vertex of $B^{\textnormal{out}}\cap N(w_1)$. Note that $G[N(p_0)]$ is a 5-cycle. We now fix $\psi=\psi^{03}$. By Claim \ref{PMinusQSigmaExtUnionToAvB}, $\psi$ extends to a $B^{\textnormal{out}}$-avoiding $L$-coloring $\sigma$ of $\textnormal{dom}(\psi)\cup\{w_3\}$. Now, $|L_{\sigma}(p_1)|\geq 5$ and $|L_{\sigma}(p_0)|\geq 4$, so $\sigma$ extends to an $L$-coloring $\tau$ of $\textnormal{dom}(\sigma)\cup V(Q)$ with $|L_{\tau}(p_0)\geq 3$. Thus, $\tau\in\textnormal{Mid}(G_X-w_1\mid p_0)$, contradicting our assumption on $\mathbf{F}$. \end{claimproof}\end{addmargin}

\vspace*{-8mm}
\begin{addmargin}[2em]{0em}
\begin{subclaim}\label{AStarVertSetNeqVAR} $A^*\neq\varnothing$. \end{subclaim}

\begin{claimproof} Suppose $A^*=\varnothing$. Since $\textnormal{Mid}(P\mid p_2)=\varnothing$ by assumption, we have $A=B$. By Claim \ref{PMinusQSigmaExtUnionToAvB} that $\psi^{23}$ extends to a $B^{\textnormal{out}}$-avoiding $L$-coloring $\sigma$ of $\textnormal{dom}(\psi)\cup\{w_3\}$.  Now, $|L_{\sigma}(w_1)|\geq 3$ and $|L_{\sigma}(p_3)|\geq 3$. Any vertex of $B^{\textnormal{in}}\cap N(w_1)$, if it exists, is unique, and since $w_1p_3\not\in E(G)$ and $|L_{\sigma}(w_1)|+|L_{\sigma}(p_3)|\geq 6$, it follows that $\sigma$ extends to an $L$-coloring $\sigma'$ of $\textnormal{dom}(\sigma)\cup\{w_1, p_3\}$ such that each vertex of $B^{\textnormal{in}}\cap N(w_1)$ has an $L_{\sigma'}$-list of size at least four. Note that $|L_{\sigma'}(p_2)|\geq 1$, so $\sigma'$ extends to an $L$-coloring $\tau$ of $\textnormal{dom}(\sigma)\cup\{w_1, p_2, p_3\}$. Then, $\tau\in\textnormal{Red}(G_X)$, contradicting our assumption on $\mathbf{F}$. \end{claimproof}\end{addmargin}

Applying Subclaim \ref{AStarVertSetNeqVAR}, we let $y^*$ be the lone vertex of $A^*$. By Subclaim \ref{B-NonemptySubCL131}, $B^{\textnormal{out}}\cap N(w_1)=\varnothing$. Since $A^*\neq\varnothing$ and $p_2$ is a $P$-gap, it follows  that $A=(B^{\textnormal{out}}\cap N(w_3))\cup\{y^*\}$.

\vspace*{-8mm}
\begin{addmargin}[2em]{0em}
\begin{subclaim}\label{BOutW3PreciselyOneVertexSub44} $B^{\textnormal{out}}\cap N(w_3))\neq\varnothing$ as well. \end{subclaim}

\begin{claimproof} Suppose that $B^{\textnormal{out}}\cap N(w_3)=\varnothing$. Thus, $A=A^*=\{y\}$. Now, $\psi^2$ extends to an $L$-coloring $\tau$ of $\textnormal{dom}(\psi)\cup\{p_2, w_1, w_3\}$ such that $|L_{\tau}(y^*)|\geq 3$. Since $p_2$ is a $P$-gap, $\textnormal{dom}(\tau)=V(G_X)$, and,  $\tau\in\textnormal{Red}(G_X)$, contradicting our assumption on $\mathbf{F}$ \end{claimproof}\end{addmargin}

By Claim \ref{PMinusQSigmaExtUnionToAvB}, $\psi^{23}$ extends to a $B^{\textnormal{out}}$-avoiding $L$-coloring $\sigma$ of $\textnormal{dom}(\psi^{23})\cup\{w_3\}$. Let $\tau$ be an extension of $\sigma$ to an $L$-coloring of $V(P)\cup\{w_1, w_3\}\cup\{y^*\}$. Since $\mathbf{F}$ is a bad filament, $\tau\not\in\textnormal{Red}(G_X+y^*)$, so there is a $u\in D_1(G_X+y^*)$ with $|L(u)|\geq 5$ and $|L_{\tau}(u)|<3$. Note that Subclaim \ref{BOutW3PreciselyOneVertexSub44} implies that $N(u)\cap\textnormal{dom}(\tau)=\{p_0, w_1, y^*\}$, so $G[N(w_1)]$ is the 5-cycle $p_0uy^*p_2p_1$. Now, $\sigma$ extends to an $L$-coloring $\pi$ of $\textnormal{dom}(\sigma)\cup\{y^*, p_2, p_3\}$ with $|L_{\pi}(w_1)|\geq 2$, and $\pi\in\textnormal{Red}(G_X+y^*\mid p_2)$, contradicting our assumption on $\mathbf{F}$. This completes the proof of Claim \ref{CLMidPp2Nonempty1}. \end{claimproof} 

Applying Claim \ref{CLMidPp2Nonempty1}, let $w_2$ be the lone vertex of $\textnormal{Mid}(P\mid p_2)$. Note that $B^{\textnormal{in}}=\varnothing$.

\begin{claim}\label{TwoVerticesLeftARemoveBcupY} $A\setminus (A^*\cup B)|=2$. \end{claim}

\begin{claimproof} Suppose not. Thus, $|A\setminus (A^*\cup B)|<2$.  We first show that $|A\setminus (A^*\cup B)|\neq 1$.  Suppose that $|A\setminus (A^*\cup B)|=1$ and let $z$ be the lone vertex of $A\setminus (A^*\cup B)$. We suppose without loss of generality that $N(z)\cap (V(P)\cup X)=\{p_0, p_1, w_2\}$. Note that $G[N(p_1)]$ is the cycle $p_0zw_2p_2w_1$. We now let $\psi=\psi^{13}$. Applying Claim \ref{PMinusQSigmaExtUnionToAvB}, we let $\sigma$ be an extension of $\psi$ to a $B^{\textnormal{out}}$-avoiding $L$-coloring of $\textnormal{dom}(\psi)\cup\{w_3\}$. Now, $|L_{\sigma}(p_1)|\geq 4$ and $|L_{\sigma}(p_2)|\geq 5$, and $|L_{\sigma}(p_3)|\geq 3$. Since $p_1p_3\not\in E(G)$, $\sigma$ extends to an $L$-coloring $\tau$ of $\textnormal{dom}(\sigma)\cup\{p_2, p_3\}$ with $|L_{\tau}(p_1)|\geq 4$, and so $\tau\in\textnormal{Red}(P+w_3\mid p_1)$, contradicting our assumption on $\mathbf{F}$. Thus, $|A\setminus (A^*\cup B)|\neq 1$, so $A\setminus (A^*\cup B)=\varnothing$.

\vspace*{-8mm}
\begin{addmargin}[2em]{0em}
\begin{subclaim}\label{AtLOneP0P4GapSubM} At least one of $p_0, p_4$ is a $P$-gap. \end{subclaim}

\begin{claimproof} Suppose neither $p_0$ nor $p_4$ is a a $P$-gap, and let $\psi=\psi^{13}$. Applying Claim \ref{PMinusQSigmaExtUnionToAvB}, we let $\sigma$ be an extension of $\psi$ to a $B^{\textnormal{out}}$-avoiding $L$-coloring of $V(P\setminus Q)\cup\{w_1, w_3\}$. Since $B^{\textnormal{in}}=\varnothing$, each vertex of $B$ has an $L_{\sigma}$-list of size at least three. If $A^*=\varnothing$, then, letting $\tau$ be an arbitrary extension of $\sigma$ to an $L$-coloring of $\textnormal{dom}(\sigma)\cup\{p_2, w_2\}$, we have $\tau\in\textnormal{Red}(G_X)$, contradicting our assumption on $\mathbf{F}$. Thus, $A^*\neq\varnothing$. Let $y^*$ be the lone vertex of $A^*$. Note that $G[N(p_2)]$ is a 6-cycle, and $\sigma$ extends to an $L$-coloring $\tau$ of $\textnormal{dom}(\sigma)\cup\{p_1, w_2, p_3\}$, where $|L_{\tau}(p_2)|\geq 2$. Then $\tau\in\textnormal{Red}(G_X\mid p_2)$, contradicting our assumption on $\mathbf{F}$. \end{claimproof}\end{addmargin}

Applying Subclaim \ref{AtLOneP0P4GapSubM}, we suppose without loss of generality that $p_0$ is not a $P$-gap, and we let $w_0$ be the lone vertex of $\textnormal{Mid}(P\mid p_0)$. By Lemma \ref{BadFil4PathLemmMain}, $p_4$ is a $P$-gap. 

\vspace*{-8mm}
\begin{addmargin}[2em]{0em}
\begin{subclaim}\label{SubCLBcapNw1varnothing} $B\cap N(w_1)=\varnothing$. \end{subclaim}

\begin{claimproof} Suppose not. Thus, $B\cap N(w_1)=\{x\}$ for some $x$, and $G[N(p_0)]$ is a 5-cycle. By Claim \ref{PMinusQSigmaExtUnionToAvB}, $\psi^{03}$ extends to a $B^{\textnormal{out}}$-avoiding $L$-coloring $\sigma$ of $\textnormal{dom}(\psi)\cup\{w_3\}$, and $\sigma$ extends to an $L$-coloring $\tau$ of $\textnormal{dom}(\sigma)\cup V(Q)\cup\{w_2\}$ with $|L_{\tau}(p_0)|\geq 4$. Then $\tau\not\in\textnormal{Red}(G_X-w_1\mid p_0)$, contradicting our assumption on $\mathbf{F}$. \end{claimproof}\end{addmargin}

Now, $B\cap N(w_3)=B^{\textnormal{out}}\cap N(w_3)$. By Claim \ref{PMinusQSigmaExtUnionToAvB}, $\psi^{23}$ extends to a $B^{\textnormal{out}}$-avoiding $L$-coloring $\sigma$ of $\textnormal{dom}(\psi)\cup\{w_3\}$. Let $\sigma'$ be an extension of $\sigma$ to an $L$-coloring $V(G_X)$. We have $\sigma'\not\in\textnormal{Red}(G_X)$, as $\mathbf{F}$ is a bad filament. Subclaim \ref{SubCLBcapNw1varnothing} implies that $A^*=\{x\}$ for some vertex $x$, so $G[N(p_2)]$ is a 6-cycle, and there is a $c\in L_{\sigma}(w_2)$ with $|L_{\sigma}(p_2)\setminus\{c\}|\geq 3$. Thus, there is an extension of $\sigma$ to an $L$-coloring $\tau$ of $V(G_X-p_2)$ such that $\tau(w_2)=c$, so $|L_{\tau}(p_2)|\geq 2$, and $\tau\in\textnormal{Red}(G_X\mid p_2)$, contradicting our assumption on $\mathbf{F}$. \end{claimproof}

Since $|A\setminus (A^*\cup B)|=2$,  there exist vertices $z, z'$ such that $A\setminus B=A^*\cup\{z, z'\}$, where $N(z)\cap (V(P)\cup X)=\{p_0, p_1, w_2\}$ and $N(z')\cap V(P)\cup X)=\{w_2, p_3, p_4\}$. Set $U=\{p_1, p_3\}$, and note that $\psi^{13}$ extends to an $L$-coloring $\tau$ of $V(P\setminus U)\cup\{w_2\}$ such that $U$ is $(L, \tau)$-inert in $G$. Letting $H:=V(P)\cup\{w_2\}$, we get $\tau\in\textnormal{Red}(H\mid U)$, contradicting our assumption on $\mathbf{F}$. This completes the proof of Lemma \ref{PathLenLemmaQGapMainDist2}. \end{proof}

\section{Completing the proof of Theorem \ref{MainResultColorAndDeletePaths}}\label{MainProofSubSecThm1PathCol}

With Lemma \ref{PathLenLemmaQGapMainDist2} in hand, we now complete the proof of Theorem \ref{MainResultColorAndDeletePaths}.

\begin{proof} Suppose toward a contradiction that Theorem \ref{MainResultColorAndDeletePaths} does not hold and let $\mathbf{F}=(P, T, T', \mathbf{f})$ be a bad filament.

\begin{claim}\label{CLMainThNoThreeConsecGaps} For any subpath $Q$ of $P\setminus (T\cup T')$ of length two, at least one vertex of $Q$ is not a $P$-gap. \end{claim}

\begin{claimproof} Let $Q=p_1p_2p_3$ and suppose toward a contradiction that each of $p_1, p_2, p_3$ is a $P$-gap. By Lemma \ref{QSubPPTT'}, $\mathbf{f}$ extends to an element $\sigma$ of $\textnormal{Red}(P-p_2)$, and $\sigma$ extends to an $L$-coloring $\tau$ of $V(P)$. Then $\tau\in\textnormal{Red}(P)$, contradicting our assumption that $\mathbf{F}$ is a bad filament. \end{claimproof}

\begin{claim}\label{MidPoint7PCLMain} Let $Q$ be a subpath of $P\setminus (T\cup T')$ of length six, where $d(Q, T\cup T')\geq 5$ and neither endpoint of $Q$ is a $P$-gap. Then the midpoint of $Q$ is a $P$-gap. \end{claim}

\begin{claimproof} Let $Q:=p_1p_2p_3p_4p_5p_6p_7$, and suppose $p_4$ is not a $P$-gap. Let $w_m$ be the lone vertex of $\textnormal{Mid}(P\mid p_4)$. By assumption, there exist vertices $w, w'$, where $\{w\}=\textnormal{Mid}(P\mid p_1)$ and $\{w'\}=\textnormal{Mid}(P\mid p_7)$. By Lemma \ref{PathLenLemmaQGapMainDist2}, each vertex of $Q\setminus\{p_1, p_4, p_7\}$ is a $P$-gap. Let $R=p_3p_4p_5$. By Lemma \ref{QSubPPTT'}, $\mathbf{f}$ extends to an element $\sigma$ of $\textnormal{Red}(P\setminus R)$, and $\sigma$ extends to an $L$-coloring $\tau$ of $V(P)$ with $|L_{\tau}(w_m)|\geq 3$. As each vertex of $Q\setminus\{p_1, p_4, p_7\}$ is a $P$-gap, we get $\tau\in\textnormal{Red}(P)$, contradicting our assumption that $\mathbf{F}$ is a bad filament. \end{claimproof}

\begin{claim}\label{xyEdgeConsecVerticesPGaps} Let $xy$ be an edge of $P\setminus (T\cup T')$, where $d(\{x,y\}, T\cup T')\geq 5$. Then at least one of $x,y$ is a $P$-gap. \end{claim}

\begin{claimproof} By Definition \ref{DefFilamentPathStrucAx} , $P\setminus (T\cup T')$ has length at least 30, and since $P$ is a shortest path between its endpoints, it follows that there is a subpath $R$ of $P\setminus (T\cup T')$ of length 7 such that $xy$ is a terminal edge of $R$ and $d(R, T\cup T')\geq 5$. Let $R=p_1p_2\cdots p_8$ be this path, where $x=p_1$ and $y=p_2$.

\vspace*{-8mm}
\begin{addmargin}[2em]{0em}
\begin{subclaim}\label{PGapInsideRMidPSubCL} Each vertex of $R\setminus\{x, y, p_5, p_6\}$ is a $P$-gap, and no vertex of $\{x, y, p_5, p_6\}$ is a $P$-gap. \end{subclaim}

\begin{claimproof} Since neither $x$ nor $y$ is a $P$-gap, it follows from Lemma \ref{PathLenLemmaQGapMainDist2} that $p_3, p_4$ are $P$-gaps. By Claim \ref{CLMainThNoThreeConsecGaps}, $p_5$ is not a $P$-gap. If $p_6$ is also not a $P$-gap, then it follows from Lemma \ref{PathLenLemmaQGapMainDist2} applied to $p_5p_6$ that neither $p_7$ nor $p_8$ is a $P$-gap, and then we are done, so it suffices to show that $p_6$ is not a $P$-gap. Suppose $p_6$ is a $P$-gap. By Lemma \ref{PathLenLemmaQGapMainDist2}, $p_7$ is a $P$-gap as well. By Claim \ref{CLMainThNoThreeConsecGaps}, $p_8$ is not a $P$-gap, contradicting Claim \ref{MidPoint7PCLMain} applied to $R-p_1$. \end{claimproof}\end{addmargin}

Since $p_5, p_6$ are not $P$-gaps, there exist vertices $w, w'$, where $\{w\}=\textnormal{Mid}(P\mid p_5)$ and $\{w'\}=\textnormal{Mid}(P\mid p_6)$. Let $R^*=p_4p_5p_6p_7$. By Lemma \ref{QSubPPTT'}, $\mathbf{f}$ extends to an element $\sigma$ of $\textnormal{Red}(P\setminus R^*)$, and $\sigma$ extends to an $L$-coloring $\sigma^+$ of $\textnormal{dom}(\sigma)\cup\{p_4, p_6\}$, where $|L_{\sigma^+}(w)|\geq 4$. Likewise, $\sigma^+$ extends to an $L$-coloring $\tau$ of $V(P)$ with $|L_{\tau}(w')|\geq 3$. By Subclaim \ref{PGapInsideRMidPSubCL}, no vertex of $\textnormal{Mid}(P)$ has a neighbor in $R^*$, except for $w, w'$. As $\sigma\in\textnormal{Red}(P\setminus R^*)$, we have $\tau\in\textnormal{Red}(P)$, contradicting our assumption that $\mathbf{F}$ is a bad filament.  \end{claimproof}

By Claim \ref{CLMainThNoThreeConsecGaps}, since $P\setminus (T\cup T')$ has length at least 30, there is an $x\in V(P)\setminus V(T\cup T')$, where $x$ is not a $P$-gap and $d(x, T\cup T')\geq 11$. Let $Q$ be one of the two subpaths of $P$ which has length six and has $x$ as a terminal vertex. Let $Q=p_1\cdots p_7$, where $x=p_1$. By Claim \ref{xyEdgeConsecVerticesPGaps}, $p_2$ is a $P$-gap, and, by Lemma \ref{PathLenLemmaQGapMainDist2}, $p_3$ is a $P$-gap as well. Thus, again by Claim \ref{CLMainThNoThreeConsecGaps}, $p_4$ is not a $P$-gap. Likewise, Claim \ref{xyEdgeConsecVerticesPGaps} and Lemma \ref{PathLenLemmaQGapMainDist2}imply that each of $p_5, p_6$ is a $P$-gap, and, by Claim \ref{CLMainThNoThreeConsecGaps}, $p_7$ is not a $P$-gap, so we contradict Claim \ref{MidPoint7PCLMain}. This completes the proof of Theorem \ref{MainResultColorAndDeletePaths}. \end{proof}

\section{The proof of Theorems \ref{FaceConnectionMainResult} and \ref{SingleFaceConnRes}}\label{ProofThmMainSec}

Now we prove Theorem \ref{FaceConnectionMainResult}, which we restate below.

\begin{thmn}[\ref{FaceConnectionMainResult}] There exists a constant $d\geq 0$ such that the following holds: Let $G$ be a short-inseparable embedding on a surface $\Sigma$ with $\textnormal{fw}(G)\geq 6$, and let $\mathcal{C}$ be a family of facial cycles of $G$, where the elements of $\mathcal{C}$ are pairwise of distance $\geq d$ apart and every facial subgraph of $G$ not lying in $\mathcal{C}$ is a triangle. Let $L$ be a list-assignment for $V(G)$, where each vertex of $V(G)\setminus\left(\bigcup_{C\in\mathcal{C}}V(C)\right)$ has a list of size at least five. Let $\mathcal{D}\subseteq\mathcal{C}$ with $|\mathcal{D}|>1$, where each element of $\mathcal{D}$ is uniquely $4$-determined and each vertex of $\bigcup_{C\in\mathcal{D}}V(C)$ has a list of size three. Let $F\in\mathcal{D}$ and $\{P_C: C\in\mathcal{D}\setminus\{F\}\}$ be a family of paths, pairwise of distance $\geq d$ apart, where each $P_C$ is a shortest $(C,F)$-path. Then there exists a complete $L$-reduction $(A, \phi)$ such that
\begin{enumerate}[label=\arabic*)]
\item $G[A]$ is connected and has nonempty intersection with each element of $\mathcal{D}$; AND
\item $A\subseteq\left(B_2(F)\cup\textnormal{Sh}^4(F)\right)\cup\bigcup_{C\in\mathcal{D}\setminus\{F\}}\left(\textnormal{Sh}^4(C)\cup B_2(C\cup P_C)\right)$.
\end{enumerate}
In particular, any $d\geq 34$ suffices. 
 \end{thmn}

\begin{proof} Let $\mathcal{D}\setminus\{F\}=\{C_1, \cdots, C_m\}$. For each $i=1, \cdots, m$, let $w_iu_i$ be the unique edge of $P_{C_i}$ with $w_i\in D_2(F)$ and $u_i\in D_3(F)$, and let $Q_i$ be a maximal $w_i$-enclosure with respect to $[\Sigma, G, F, L]$. Note that, for any $1\leq i<j\leq m$, the graphs $G^{\textnormal{small}}_{Q_i}$ and $G^{\textnormal{small}}_{Q_j}$ are of distance at least $d-6\geq 28$ apart (and, in particular, pairwise-disjoint) and, for each $i=1, \cdots, m$, the graph $G^{}\cap F$ is a path, i.e $G\neq F$. We may also suppose that $C_1, \cdots, C_m$ are labelled so that the paths $R_1, \cdots, R_m$ are arranged in cyclic order on $F$. We let $P_1, \cdots, P_m$ denote the subgraphs of $F\setminus (\mathring{R}_1\cup\cdots\cup\mathring{R}_m)$ such that, for each $i=1, \cdots, m$, reading the indices mod $m$, $P_i$ intersects with $R_1\cup\cdots R_m$ on one endpoint of $R_i$ and one endpoint of $R_{i+1}$ (in particular, if $m=1$, then either $|V(R_1)|=1$ and $P_1=F$, or $|V(R_1)|>1$ and $P_1$ is the unique path which has the same endpoints as $R_1$ but is edge-disjoint to $R_1$). For each $i=1, \cdots, m$, let $x^i, y^i$ be the endpoints of $P_i$, where $x^i\in V(R_i)$ and $y^i\in V(R_{i+1})$ (possibly $x^i=y^{i-1}$). This is illustrated in Figure \ref{ElipFinProofHardCase}, in the case where $m=4$ and each of $w_1, w_2, w_3, w_4$ is non-degenerate. Each of $Q_1, Q_2, Q_3, Q_4$ is illustrated in bold red.

\begin{center}\begin{tikzpicture}
\draw (0,0) ellipse (8cm and 4cm);
\node[shape=circle, draw=black, inner sep=0pt, minimum size=0.55cm] (W) at (5, 0) {\tiny $w_3$};
\node[shape=circle, draw=black, inner sep=0pt, minimum size=0.55cm] (U) at (4, 0) {\tiny $u_3$};
\node[shape=circle, draw=white, inner sep=0pt, minimum size=0.55cm] (Udots) at (3, 0) { $\hdots$};
\node[shape=circle, draw=white, inner sep=0pt, minimum size=0.55cm] (GSm) at (6.5, 0) {$G^{\textnormal{small}}_{Q_3}$};
\node[shape=circle, fill=black, inner sep=0pt, minimum size=0.25cm] (X1) at (7.6, 1.2) {};
\node[shape=circle, fill=black, inner sep=0pt, minimum size=0.25cm] (XN) at (7.6, -1.2) {};
\node[shape=circle, fill=black, inner sep=0pt, minimum size=0.25cm] (Z) at (5.5, 0.985) {};
\node[shape=circle, fill=black, inner sep=0pt, minimum size=0.25cm] (Z') at (5.5, -0.985) {};
\draw[-, line width=1.8pt, color=red] (X1) to (Z) to (W) to (Z') to (XN);
\draw[-] (Udots) to (U) to (W);

\node[shape=circle, draw=black, inner sep=0pt, minimum size=0.55cm] (W*) at (-5, 0) {\tiny $w_1$};
\node[shape=circle, draw=black, inner sep=0pt, minimum size=0.55cm] (U*) at (-4, 0) {\tiny $u_1$};
\node[shape=circle, draw=white, inner sep=0pt, minimum size=0.55cm] (U*dots) at (-3, 0) { $\hdots$};
\node[shape=circle, draw=white, inner sep=0pt, minimum size=0.55cm] (GSm*) at (-6.5, 0) {$G^{\textnormal{small}}_{Q_1}$};
\node[shape=circle, fill=black, inner sep=0pt, minimum size=0.25cm] (X1*) at (-7.6, 1.20) {};
\node[shape=circle, fill=black, inner sep=0pt, minimum size=0.25cm] (XN*) at (-7.6, -1.20) {};
\node[shape=circle, fill=black, inner sep=0pt, minimum size=0.25cm] (Z*) at (-5.5, 0.985) {};
\node[shape=circle, fill=black, inner sep=0pt, minimum size=0.25cm] (Z'*) at (-5.5, -0.985) {};
\draw[-, line width=1.8pt, color=red] (X1*) to (Z*) to (W*) to (Z'*) to (XN*);
\draw[-] (U*dots) to (U*) to (W*);

\node[shape=circle, fill=black, inner sep=0pt, minimum size=0.25cm] (EL1) at (1.2, -3.95) {};
\node[shape=circle, fill=black, inner sep=0pt, minimum size=0.25cm] (EL2) at (-1.2, -3.95) {};
\node[shape=circle, fill=black, inner sep=0pt, minimum size=0.25cm] (HL1) at (0.985, -3.1) {};
\node[shape=circle, fill=black, inner sep=0pt, minimum size=0.25cm] (HL2) at (-0.985, -3.1) {};
\node[shape=circle, draw=black, inner sep=0pt, minimum size=0.55cm] (W2) at (0, -2.8) {\tiny $w_2$};
\node[shape=circle, draw=black, inner sep=0pt, minimum size=0.55cm] (U2) at (0, -1.8) {\tiny $u_2$};
\node[shape=circle, draw=white, inner sep=0pt, minimum size=0.55cm] (Vdots1) at (0, -0.8) { $\vdots$};
\draw[-] (W2) to (U2) to (Vdots1);
\draw[-, line width=1.8pt, color=red] (EL1) to (HL1) to (W2) to (HL2) to (EL2);
\node[shape=circle, draw=white, inner sep=0pt, minimum size=0.55cm] (GSm) at (0, -3.55) {\small $G^{\textnormal{small}}_{Q_2}$};

\node[shape=circle, fill=black, inner sep=0pt, minimum size=0.25cm] (EL1+) at (1.2, 3.95) {};
\node[shape=circle, fill=black, inner sep=0pt, minimum size=0.25cm] (EL2+) at (-1.2, 3.95) {};
\node[shape=circle, fill=black, inner sep=0pt, minimum size=0.25cm] (HL1+) at (0.985, 3.1) {};
\node[shape=circle, fill=black, inner sep=0pt, minimum size=0.25cm] (HL2+) at (-0.985, 3.1) {};
\node[shape=circle, draw=black, inner sep=0pt, minimum size=0.55cm] (W2+) at (0, 2.8) {\tiny $w_4$};
\node[shape=circle, draw=black, inner sep=0pt, minimum size=0.55cm] (U2+) at (0, 1.8) {\tiny $u_4$};
\node[shape=circle, draw=white, inner sep=0pt, minimum size=0.55cm] (Vdots1+) at (0, 0.8) { $\vdots$};
\draw[-] (W2+) to (U2+) to (Vdots1+);
\draw[-, line width=1.8pt, color=red] (EL1+) to (HL1+) to (W2+) to (HL2+) to (EL2+);
\node[shape=circle, draw=white, inner sep=0pt, minimum size=0.55cm] (GSm) at (0, 3.55) {\small $G^{\textnormal{small}}_{Q_4}$};
\node[shape=circle, draw=white, inner sep=0pt, minimum size=0.55cm] (GSR1) at (-8.5, 0) {$R_1$};
\node[shape=circle, draw=white, inner sep=0pt, minimum size=0.55cm] (GSR3) at (8.5, 0) {$R_3$};

\node[shape=circle, draw=white, inner sep=0pt, minimum size=0.55cm] (GSR2) at (0, 4.5) {$R_2$};
\node[shape=circle, draw=white, inner sep=0pt, minimum size=0.55cm] (GSR4) at (0, -4.5) {$R_4$};

\node[shape=circle, draw=white, inner sep=0pt, minimum size=0.55cm] (GSP1) at (-5, 3.6) {$P_1$};
\node[shape=circle, draw=white, inner sep=0pt, minimum size=0.55cm] (GSP4) at (-5, -3.6) {$P_4$};
\node[shape=circle, draw=white, inner sep=0pt, minimum size=0.55cm] (GSP2) at (5, 3.6) {$P_2$};
\node[shape=circle, draw=white, inner sep=0pt, minimum size=0.55cm] (GSP3) at (5, -3.6) {$P_3$};

\node[shape=circle, draw=white, inner sep=0pt, minimum size=0.25cm] (Xu1) at (-8, 1.3) {\small $x^1$};
\node[shape=circle, draw=white, inner sep=0pt, minimum size=0.25cm] (Yu4) at (-8, -1.3) {\small $y^4$};
\node[shape=circle, draw=white, inner sep=0pt, minimum size=0.25cm] (Yu2) at (8, 1.3) {\small $y^2$};
\node[shape=circle, draw=white, inner sep=0pt, minimum size=0.25cm] (Xu3) at (8, -1.3) {\small $x^3$};

\node[shape=circle, draw=white, inner sep=0pt, minimum size=0.25cm] (Yu1) at (-1.3, 4.3) {\small $y^1$};
\node[shape=circle, draw=white, inner sep=0pt, minimum size=0.25cm] (Xu2) at (1.3, 4.3) {\small $x^2$};
\node[shape=circle, draw=white, inner sep=0pt, minimum size=0.25cm] (Xu4) at (-1.3, -4.3) {\small $x^4$};
\node[shape=circle, draw=white, inner sep=0pt, minimum size=0.25cm] (Yu3) at (1.3, -4.3) {\small $y^3$};
\end{tikzpicture}\captionof{figure}{The ellipse enclosing the diagram represents the cycle $F$}\label{ElipFinProofHardCase}\end{center}

Our face-width conditions, together with the maximality of $Q_1, \cdots, Q_m$, immediately imply the following:

\begin{claim} Each of $P_1, \cdots, P_m$ is 4-consistent. Furthermore, for any $1\leq i<j\leq m$, if there is a $(P_i, P_j)$-path $M$ of length at most four, then $j\equiv i+1$ mod $m$ and $M\subseteq Q_i$ (where the index is again read mod $m$). \end{claim}

We now note the following.

\begin{claim}\label{FamilyOf2m-1Color} There is a family $\psi_1, \cdots, \psi_m, \pi_1, \cdots, pi_{m-1}$ of partial $L$-colorings of $G$ such that
\begin{enumerate}[label=\arabic*)]
\itemsep-0.1em
\item For each $i=1, \cdots, m$, $\psi_i$ is a $(Q_i, u_iw_i)$-target; AND
\item For each $i=1, \cdots, m-1$, $\pi_i\in\textnormal{Link}(P_i)$; AND
\item $\psi_1\cup\cdots\cup\psi_m\cup\pi_1\cup\cdots\cup\pi_{m-1}$ is well-defined, i.e a proper $L$-coloring of its domain
\end{enumerate}
\end{claim}

\begin{claimproof} If $m=1$, then this just follows from one application of Theorem \ref{GenThmTargetConnector}, so suppose that $m=2$. Again applying Theorem \ref{GenThmTargetConnector}, we fix a $(Q_1, u_1w_1)$-target $\psi_1$. By Theorem \ref{MainLinkingResultAlongFacialCycle}, there is a $\pi_1\in\textnormal{Link}(P_1)$ with $\pi_1(x^1)=\psi_1(x^1)$. Again by Theorem \ref{GenThmTargetConnector}, there is a $(Q_2, u_2w_2)$-target $\psi_2$ with $\psi_2(y^1)=\pi_1(y^1)$. Proceeding in this way, alternating between Theorems \ref{GenThmTargetConnector} and \ref{MainLinkingResultAlongFacialCycle}, we produce the desired $\psi_1, \cdots, \psi_m, \pi_1, \cdots, \pi_{m-1}$. \end{claimproof}

We now let $\psi_1, \cdots, \psi_m, \pi_1, \cdots, \pi_{m-1}$ be as in Claim \ref{FamilyOf2m-1Color}. We have not yet colored $P_m$, except for its endpoints, but we can apply Theorem \ref{LinkPlusOneMoreVertx}.

\begin{claim} There is a $P^m$-peak $w$ and an $L$-reduction $(T, \tau)$ such that $V(P^m)\subseteq T\subseteq V(P^m+w)\cup\textnormal{Sh}^3(P^m)$, where $\tau$ uses $\psi_m(x^m), \psi_1(y^m)$ on $x^m, y^m$ respectively, and furthermore, $d(w, Q_1\cup Q_m)>3$. \end{claim}

\begin{claimproof} Since $d(Q_1, Q_m)\geq d-6\geq 28$, our trianguation conditions imply there are at least four internal $P^m$-peaks $w$ such that $d(w, Q_1\cup Q_m)>3$, so the claim now follows from Theorem \ref{LinkPlusOneMoreVertx}. \end{claimproof}

Now we deal with the reductions near the elements of $\mathcal{D}\setminus\{F\}$. For each $i=1, \cdots, m$, let $u_i'w_i'$ be the unique edge of $P_{C_i}$ such that $u_i'\in D_3(C_i)$ and $w_i'\in D_2(C_i')$. It follows from Theorem \ref{MainCollarResultCutProc} applied to the collar $[\Sigma, G, C_i, L]$, that there is a complete $L$-reduction $(S_i, \sigma_i)$ such that

\begin{enumerate}[label=\arabic*)]
\itemsep-0.1em
\item $u_i\in\textnormal{dom}(\sigma_i)$ and $V(C_i)\subseteq S_i\subseteq B_3(C_i)\cup\textnormal{Sh}^4(C_i)$, and $S_i$ is topologically reachable from $C_i$; AND
\item $(S_i\setminus\textnormal{Sh}^4(C_i))\cap D_3(C_i)=\{u_i'\}$ and $(S_i\setminus\textnormal{Sh}^4(C_i))\cap D_2(C_i)=\{w_i'\}$; AND
\end{enumerate}

Since $S_i$ is topologically reachable from $C_i$, our triangulation conditions imply that $G[S_i]$ is connected. Now, let $P_{C_i}^*$ be the unique subpath of $w_iP_{C_i}w_i'$ such that the endpoints of $P^*_{C_i}$ lie in $\textnormal{dom}(\psi_i\cup\sigma_i)$ and $P^*_{C_i}$ contains all the vertices of $\textnormal{dom}(\psi_i\cup\sigma_i)\cap V(w_iP_{C_i}w_i')$. In particular, if each of $w_i, w_i'$ is nondegenerate, then $P^*_{C_i}=w_iP_{C_i}w_i'$. We regard $P^*_{C_i}$ as path with terminal subpaths of length at most one precolored by $\psi_i\cup\sigma_i$. By Theorem \ref{MainResultColorAndDeletePaths}, there is an an $L$-reduction $(A^*_i, \tau_i)$ where $\tau_i\cup\psi_i\cup\sigma_i$ is a proper $L$-coloring of its domain and $V(P^*_{C_i})\subseteq A^*_i$ and furthermore, $A^*_i\setminus V(P^*_{C_i})\subseteq B_1(P^*_{C_i})\setminus\{w_i, u_i, u_i', w_i'\}$ and $d(A^*_i\setminus V(P^*_{C_i}), \{u_i, u_i'\})\geq 3$. Now, we construct $A, \phi$ satisfying Theorem \ref{SingleFaceConnRes}. Consider the list of the $4m$ $L$-reductions enumerated above, i.e
\begin{itemize}
\itemsep-0.1em
\item $(V(G^{\textnormal{small}}_Q\setminus (Q\setminus \textnormal{dom}(\psi_i))), \psi_i)$ for each $i=1, \cdots, m$
\item $(V(P_i)\cup\textnormal{Sh}^2(P_i), \pi_i)$ for each $i=1, \cdots, m-1$
\item $(T, \tau)$
\item $(S_i, \sigma_i)$ for each $i=1, \cdots, m$
\item $(A_i^*, \tau_i)$ for each $i=1, \cdots, m$
\end{itemize}

It follows from Claim \ref{P0P1PathReachAcc}, toegether with our choice of $P^m$-peak $w$, that the family above is a consistent family of $L$-reductions. We take $A$ to be the union of the first coordinates of these $L$-reductions and $\phi$ to be the union of the second coordinates. By Observation \ref{UnionConsistentObs}, $(A, \phi)$ is an $L$-reduction, and this pair satisfies Theorem \ref{FaceConnectionMainResult}. \end{proof}

We now prove Theorem \ref{SingleFaceConnRes} (which we restate below). 

\begin{thmn}[\ref{SingleFaceConnRes}] Let $\Sigma, G, \mathcal{C}, L, d$ be as in Theorem \ref{FaceConnectionMainResult}. Let $F\in\mathcal{C}$, where each vertex of $F$ has a list of size at least three and $F$ is uniquely 4-determined (possiby $\mathcal{C}=\{F\}$). Let $P:=v_0\cdots v_n$ be a path of length $\geq d$ with both endpoints in $F$, where $P$ is a shortest path between its endpoints, with $V(P)\cap D_k(F)=\{v_k, v_{n-k}\}$ for each $0\leq k\leq 3$, and furthermore, there is a noncontractible closed curve of $\Sigma$ contained in $F\cup P$. Then there exist a complete $L$-reduction $(A, \phi)$ such that $G[A]$ is connected and $V(F)\cup V(v_3Pv_{n-3})\subseteq A\subseteq B_2(F\cup P)\cup\textnormal{Sh}^4(F)$.
 \end{thmn}

\begin{proof} We essentially follow the argument from the proof of Theorem \ref{FaceConnectionMainResult}. Letting $\mathcal{K}:=[\Sigma, G, F, L]$, we let $Q$ be a maximal $v_2$-enclosure and $Q'$ be a maximal $v_{n-2}$-enclosure. Likewise, we set $R:=G^{\textnormal{small}}_{Q}\cap F$ and $R':=G^{\textnormal{small}}_{Q'}\cap F$. Using our conditions on $P$, we immediately get that $v_3Pv_{n-3}$ is disjoint to $G^{\textnormal{small}}_{Q}\cup G^{\textnormal{small}}_{Q'}$. Furthermore, each of $R$ and $R'$  is a path i.e  $R\neq F$ and $R'\neq F$. Let $P^0, P^1$ be the two paths of $C\setminus (\mathring{R}\cup\mathring{R}')$ which each intersect with $R$ on precisely one endpoint of $R$ and with $R'$ on precisely one endpoint of $R'$. Let $x^0, y^0$ be the endpoints of $P^0$ and $x^1, y^1$ be the endpoints of $P^1$, where $x^0, x^1$ are the endpoints of $R$ and $y^0, y^1$ are the endpoints of $R'$, as illustrated in Figure \ref{ElipseSecondForFinProof}, where $Q$ and $Q'$ are denoted in red bold (for simplicity, Figure \ref{ElipseSecondForFinProof} illustrates the case where each of $v_2, v_{n-2}$ are non-degenerate). 

\begin{center}\begin{tikzpicture}
\draw (0,0) ellipse (8cm and 3cm);
\node[shape=circle, draw=black, inner sep=0pt, minimum size=0.55cm] (W) at (5, 0) {\tiny $v_{n-2}$};
\node[shape=circle, draw=black, inner sep=0pt, minimum size=0.55cm] (U) at (3.5, 0) {\tiny $v_{n-3}$};
\node[shape=circle, draw=white, inner sep=0pt, minimum size=0.55cm] (GSm) at (6.5, 0) {$G^{\textnormal{small}}_{Q'}$};
\node[shape=circle, fill=black, inner sep=0pt, minimum size=0.25cm] (X1) at (6, 1.97) {};
\node[shape=circle, fill=black, inner sep=0pt, minimum size=0.25cm] (XN) at (6, -1.97) {};
\node[shape=circle, fill=black, inner sep=0pt, minimum size=0.25cm] (Z) at (5.5, 0.985) {};
\node[shape=circle, fill=black, inner sep=0pt, minimum size=0.25cm] (Z') at (5.5, -0.985) {};
\draw[-, line width=1.8pt, color=red] (X1) to (Z) to (W) to (Z') to (XN);
\draw[-] (U) to (W);

\node[shape=circle, draw=black, inner sep=0pt, minimum size=0.55cm] (W*) at (-5, 0) {\tiny $v_2$};
\node[shape=circle, draw=black, inner sep=0pt, minimum size=0.55cm] (U*) at (-3.5, 0) {\tiny $v_3$};
\node[shape=circle, draw=white, inner sep=0pt, minimum size=0.55cm] (GSm*) at (-6.5, 0) {$G^{\textnormal{small}}_{Q}$};
\node[shape=circle, fill=black, inner sep=0pt, minimum size=0.25cm] (X1*) at (-6, 1.97) {};
\node[shape=circle, fill=black, inner sep=0pt, minimum size=0.25cm] (XN*) at (-6, -1.97) {};
\node[shape=circle, fill=black, inner sep=0pt, minimum size=0.25cm] (Z*) at (-5.5, 0.985) {};
\node[shape=circle, fill=black, inner sep=0pt, minimum size=0.25cm] (Z'*) at (-5.5, -0.985) {};
\draw[-, line width=1.8pt, color=red] (X1*) to (Z*) to (W*) to (Z'*) to (XN*);
\draw[-] (U*) to (W*);

\node[shape=circle, draw=white, inner sep=0pt, minimum size=0.55cm] (Udot) at (0, 0) {$\cdots$};
\node[shape=circle, draw=black, inner sep=0pt, minimum size=0.55cm] (UdotMin) at (1, 0) {};
\node[shape=circle, draw=black, inner sep=0pt, minimum size=0.55cm] (UdotMax) at (-1, 0) {};
\draw[-] (U) to (UdotMin) to (Udot) to (UdotMax) to (U*);

\node[shape=circle, draw=white, inner sep=0pt, minimum size=0.55cm] (GSR-) at (-8.5, 0) {$R$};
\node[shape=circle, draw=white, inner sep=0pt, minimum size=0.55cm] (GSR) at (8.5, 0) {$R'$};
\node[shape=circle, draw=white, inner sep=0pt, minimum size=0.55cm] (P0) at (0, 3.3) {$P^0$};
\node[shape=circle, draw=white, inner sep=0pt, minimum size=0.55cm] (P1) at (0, -3.3) {$P^1$};

\node[shape=circle, draw=white, inner sep=0pt, minimum size=0.25cm] (X1L) at (6, 2.3) {\small $y^0$};
\node[shape=circle, draw=white, inner sep=0pt, minimum size=0.25cm] (X2L) at (6, -2.4) {\small $y^1$};

\node[shape=circle, draw=white, inner sep=0pt, minimum size=0.25cm] (X3L) at (-6, 2.3) {\small $x^0$};
\node[shape=circle, draw=white, inner sep=0pt, minimum size=0.25cm] (X4L) at (-6, -2.4) {\small $x^1$};

\end{tikzpicture}
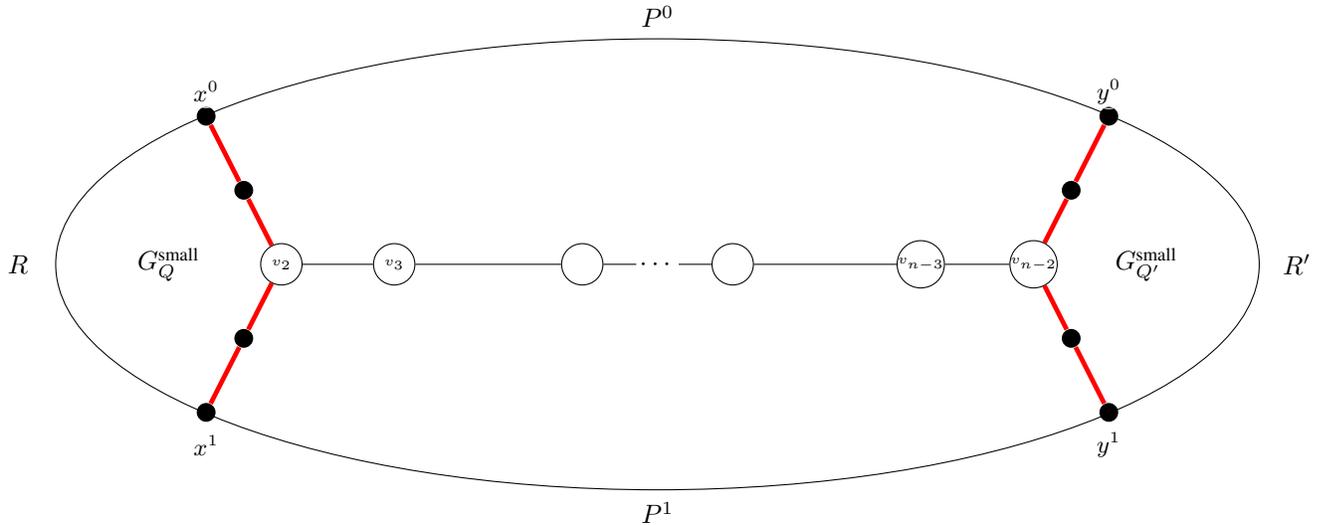
\captionof{figure}{The ellipse enclosing the diagram represents the cycle $F$}\label{ElipseSecondForFinProof}\end{center}

As $P$ is a shortest path between its endpoints, $d(x^0, y^0)\geq |E(P)|-6$, and likewise, $d(x^1, y^1)\geq |E(P)|-6$. 

\begin{claim}\label{P0P1PathReachAcc}  Each of $P^0, P^1$ is 4-consistent and any $(P^0, P^1)$-path of length at most four is contained in $G^{\textnormal{small}}_{Q}\cup G^{\textnormal{small}}_{Q'}$. \end{claim}

\begin{claimproof} Since there is a noncontractible closed curve of $\Sigma$ contained in $F\cup P$, it follows that each of $P^0, P^1$ is 4-consistent. Suppose there is a $(P^0, P^1)$-path of length at most four which is not contained in $G^{\textnormal{small}}_{Q}\cup G^{\textnormal{small}}_{Q'}$. Since $V(P)\cap D_k(F)=\{v_k, v_{n-k}\}$ for each $0\leq k\leq 3$, the maximality of $Q, Q'$ implies that there is a noncontractible cycle of $G$ contained in at most five facial subgraphs of $G$, contradicting our assumption that $\textnormal{fw}(G)\geq 6$. \end{claimproof}

Applying Theorem \ref{GenThmTargetConnector}, we now fix a $(Q, v_3v_2)$-target $\psi$. By Theorem \ref{MainLinkingResultAlongFacialCycle}, there is a $\pi\in\textnormal{Link}(P^1)$ with $\pi(x^1)=\psi(x^1)$. Again by Theorem \ref{GenThmTargetConnector}, there is a $(Q', v_{n-3}v_{n-2})$-target $\psi'$ with $\psi'(y^1)=\pi(y^1)$, so $\psi\cup\pi\cup\psi'$ is a proper $L$-coloring of its domain. Let $c=\psi(x^0)$ and $d:=\psi'(y^0)$. 

\begin{claim}\label{EdgeXyInSh2} There is a $P^0$-peak $w$ and an $L$-reduction $(T, \tau)$ such that $V(P^0)\subseteq T\subseteq V(P^0+w)\cup\textnormal{Sh}^3(P^0)$, where $\tau$ uses $c,d$ on $x^0, y^0$ respectively, and furthermore, $d(w, Q\cup Q')>3$. 
 \end{claim}

\begin{claimproof} Since $d(Q, Q')\geq |E(P)|-6\geq 28$, there are at least four internal $P^0$-peaks $w$ such that $d(w, Q\cup Q')>3$, so the claim now follows from Theorem \ref{LinkPlusOneMoreVertx}. \end{claimproof}

Now, let $P^*$ be the unique subpath of $v_2Pv_{n-2}$ such that the endpoints of $P^*$ lie in $\textnormal{dom}(\psi\cup\psi')$ and $P^*$ contains all the vertices of $\textnormal{dom}(\psi\cup\psi')\cap V(v_2Pv_{n-2})$. In particular, if each of $v_2, v_{n-2}$ is nondegenerate, then $P^*=v_2Pv_{n-2}$. We regard $P^*$ as path with terminal subpaths of length at most one precolored by $\psi\cup\psi'$. By Theorem \ref{MainResultColorAndDeletePaths}, there is an an $L$-reduction $(A^*, \sigma)$ where $\sigma\cup\psi\cup\psi'$ is a proper $L$-coloring of its domain and $V(P^*)\subseteq A^*$ and furthermore, $A^*\setminus V(P^*)\subseteq B_1(v_4Pv_{n-4})$ and $d(A^*\setminus V(P^*), \{v_3, v_{n-3}\})\geq 3$. Now, we construct $A, \phi$ satisfying Theorem \ref{SingleFaceConnRes}. Consider the list of the six $L$-reductions enumerated above, i.e
\begin{itemize}
\itemsep-0.1em
\item $(V(G^{\textnormal{small}}_Q\setminus (Q\setminus \textnormal{dom}(\psi))), \psi)$ and $(V(G^{\textnormal{small}}_{Q'}\setminus (Q'\setminus \textnormal{dom}(\psi'))), \psi')$
\item $(V(P_1)\cup\textnormal{Sh}^2(P^1), \pi)$
\item $(T, \tau)$ and $(A^*, \sigma)$
\end{itemize}

It follows from Claim \ref{P0P1PathReachAcc}, toegether with our choice of $P^0$-peak $w$, that the family above is a consistent family of $L$-reductions. We take $A$ to be the union of the first coordinates of these $L$-reductions and $\phi$ to be the union of the second coordinates. By Observation \ref{UnionConsistentObs}, $(A, \phi)$ is an $L$-reduction, and this pair satisfies Theorem \ref{SingleFaceConnRes}. \end{proof}

\end{document}